\documentclass[11pt]{amsart}
\usepackage{amsmath}
\usepackage{amssymb}
\usepackage{amsthm}
\usepackage{latexsym}
\usepackage{graphicx}
\usepackage{hyperref}
\usepackage{enumerate}
\usepackage[all]{xy}
\usepackage{chngcntr}

\newtheorem{thm}{Theorem}[section]
\newtheorem{cor}[thm]{Corollary}
\newtheorem{lem}[thm]{Lemma}
\newtheorem{prop}[thm]{Proposition}

\newtheorem{cond}[thm]{Condition}

\newtheorem{thmintro}{Theorem}
\newtheorem{conj}[thmintro]{Conjecture}
\newtheorem{ex}[thm]{Example}

\newcommand{\Z}{\mathbb Z}
\newcommand{\Q}{\mathbb Q}
\newcommand{\R}{\mathbb R}
\newcommand{\C}{\mathbb C}
\newcommand{\mf}{\mathfrak}
\newcommand{\mc}{\mathcal}
\newcommand{\mb}{\mathbf}
\newcommand{\mh}{\mathbb}
\def\Irr{{\rm Irr}}
\newcommand{\mr}{\mathrm}
\newcommand{\ind}{\mathrm{ind}}
\newcommand{\enuma}[1]{\begin{enumerate}[\textup{(}a\textup{)}] {#1} \end{enumerate}}
\newcommand{\Fr}{\mathrm{Frob}}
\newcommand{\Sc}{\mathrm{sc}}
\newcommand{\ad}{\mathrm{ad}}
\newcommand{\cusp}{\mathrm{cusp}}
\newcommand{\nr}{\mathrm{nr}}

\newcommand{\Rep}{\mathrm{Rep}}
\newcommand{\Res}{\mathrm{Res}}

\newcommand{\der}{\mathrm{der}}
\newcommand{\red}{\mathrm{red}}

\newcommand{\Mod}{\mathrm{Mod}}
\newcommand{\fMod}{\mathrm{Mod}_{\mr f}}
\newcommand{\Hom}{\mathrm{Hom}}
\newcommand{\End}{\mathrm{End}}

\begin{document}

\title[Parameters of Hecke algebras]{Parameters of Hecke algebras\\ 
for Bernstein components of $p$-adic groups}
\date{\today}
\maketitle

\begin{center}
{\Large Maarten Solleveld} \\[1mm]
IMAPP, Radboud Universiteit Nijmegen\\
Heyendaalseweg 135, 6525AJ Nijmegen, the Netherlands \\
email: m.solleveld@science.ru.nl \\
\end{center}

\begin{abstract}
Let $G$ be a reductive group over a non-archimedean local field $F$. Consider an arbitrary
Bernstein block $\Rep (G)^{\mf s}$ in the category of complex smooth $G$-representations.
In earlier work the author showed that there exists an affine Hecke algebra $\mc H (\mc O,G)$
whose category of right modules is closely related to $\Rep (G)^{\mf s}$. In many cases this is
in fact an equivalence of categories, like for Iwahori-spherical representations. 

In this paper we study the $q$-parameters of the affine Hecke algebras $\mc H (\mc O,G)$. 
We compute them in many cases, in particular for principal series representations of 
quasi-split groups and for classical groups. 

Lusztig conjectured that the $q$-parameters are always integral powers of the cardinality of
the residue field of $F$, and that they coincide with the $q$-parameters coming from some 
Bernstein block of unipotent representations. We reduce this conjecture to the case of 
absolutely simple $p$-adic groups, and we prove it for most of those.
\end{abstract}

\tableofcontents

\section*{Introduction}

It is well-known that affine Hecke algebras play an important role in the representation
theory of a reductive group $G$ over a non-archimedean local field $F$. In many cases a Bernstein 
block $\Rep (G)^{\mf s}$ in the category of smooth complex $G$-representations is equivalent with 
the module category of an affine Hecke algebra (maybe extended with some finite group). This was 
first shown for Iwahori-spherical representations \cite{IwMa,Bor} and for depth zero representations
\cite{Mor1}. Such an equivalence of categories was established for representations of $GL_n (F)$
and of inner forms of $GL_n (F)$ \cite{SeSt4,SeSt6}, by using the theory of types \cite{BuKu2}.
From there, almost the same structure has been established for inner
forms of $SL_n (F)$ \cite{ABPS2}. 

An alternative approach goes via the algebra of $G$-endomorphisms of a progenerator $\Pi^{\mf s}$
of $\Rep (G)^{\mf s}$. The category of right modules of $\End_G (\Pi^{\mf s})$ is naturally
equivalent with $\Rep (G)^{\mf s}$. Heiermann \cite{Hei2,Hei3} showed that for symplectic groups,
special orthogonal groups, unitary groups and inner forms of $GL_n (F)$, $\End_G (\Pi^{\mf s})$
is always Morita equivalent with an (extended) affine Hecke algebra. 

Recently the author generalized this to all Bernstein components of all reductive $p$-adic groups
\cite{SolEnd}. In the most general setting some subtleties have to be taken into account: the
involved affine Hecke algebra must be extended with the group algebra of a finite group, but
that group algebra might be twisted by a 2-cocycle. One of the techniques used to get there is 
localization with respect to certain sets of central characters. As a consequence the resulting 
equivalence with $\Rep (G)^{\mf s}$ works for finite length representations, but cannot be
guaranteed entirely for representations of infinite length. (No concrete counterexamples are
known, but the place to look would be the non-split inner forms of $SL_n (F)$.) Nevertheless, 
the bottom line is that $\Rep (G)^{\mf s}$ is largely governed by an affine Hecke algebra from 
$\End_G (\Pi^{\mf s})$. 

Let $M$ a Levi factor $M$ of a parabolic subgroup $P$ of $G$ such that $\Rep (G)^{\mf s}$ arises 
by parabolic induction from a supercuspidal representation $\sigma$ of $M$. We denote the 
variety of unramified twists of $\sigma$ by $\mc O \subset \Irr (M)$, and the affine Hecke
algebra described above by $\mc H (\mc O,G)$. 
If at the same time a $\mf s$-type $(J,\rho)$ is known, then its Hecke algebra 
\[
\mc H (G,J,\rho) \cong \End_G (\ind_J^G (\rho))^{op}
\] 
is Morita equivalent with $\End_G (\Pi^{\mf s})^{op}$. In fact \cite[Appendix A]{BaSa} shows that 
in most cases $\mr{ind}_J^G (\rho)$ is isomorphic with $\Pi^{\mf s}$. In this setting 
$\mc H (\mc O,G)$ can also be constructed from $\mc H (G,J,\rho)$.

The next question is of course: what does
$\mc H (\mc O,G)$ look like? Like all affine Hecke algebras, it is determined by a root datum
and some $q$-parameters. The lattice $X$ (from that root datum) can be identified with the 
character lattice of $\mc O$, once the latter has been made into a complex torus by choosing
a base point. The root system $\Sigma_{\mc O}^\vee$ (also from the root datum) is contained in
$X$ and determined by the reducibility points of the family of representations
$\{ I_P^G (\sigma') : \sigma' \in \mc O \}$. Then $\mc H (\mc O,G)$ contains a maximal commutative
subalgebra $\C [X] \cong \C [\mc O]$ and a finite dimensional Iwahori--Hecke algebra
$\mc H (W(\Sigma_{\mc O}^\vee), q_F^{\lambda})$ such that
\[
\mc H (\mc O,G) = \C [\mc O] \otimes_\C \mc H (W(\Sigma_{\mc O}^\vee), q_F^{\lambda})
\text{ as vector spaces.}
\]
Here $q_F$ denotes the cardinality of the residue field of $F$, while $\lambda$ will be defined
soon. For every $X_\alpha \in \Sigma_{\mc O}^\vee$ there is a $q_\alpha \in \R_{>1}$ such that
\[
I_P^G (\sigma') \text{ is reducible for all } \sigma' \in \mc O \text{ with }
X_\alpha (\sigma' ) = q_\alpha .
\]
Sometimes there is also a number $q_{\alpha*} \in (1,q_\alpha]$ with the property
\[
I_P^G (\sigma') \text{ is reducible for all } \sigma' \in \mc O \text{ with }
X_\alpha (\sigma' ) = -q_{\alpha *} .
\]
When such a real number does not exist, we put $q_{\alpha *} = 1$. These $q$-parameters $q_\alpha$
and $q_{\alpha *}$ appear in the Hecke relations of $\mc H (W(\Sigma_{\mc O}^\vee), q_F^{\lambda})$:
\[
0 = (T_{s_\alpha} + 1)(T_{s_\alpha} - q_F^{\lambda (\alpha)}) \quad \text{with} \quad
q_F^{\lambda (\alpha)} = q_\alpha q_{\alpha *} \in \R_{>1}.
\]
Further, we define $\lambda^* (\alpha) \in \R_{\geq 0}$ by 
\[
q_F^{\lambda^* (\alpha)} = q_\alpha q_{\alpha *}^{-1}.
\]
Knowing $q_\alpha, q_{\alpha*}$ is also equivalent to knowing the poles of the Harish-Chandra
$\mu$-function on $\mc O$ associated to $\alpha$. 
See Section \ref{sec:progen} for more details on the above setup.

The representation theory of $\mc H (\mc O,G)$ depends in a subtle way on the $q$-parameters
$q_\alpha, q_{\alpha*}$ for $X_\alpha \in \Sigma_{\mc O}^\vee$, so knowing them helps to
understand $\Rep (G)^{\mf s}$. That brings us to the main goal of this paper: 
\emph{determine the $q$-parameters of $\mc H (\mc O,G)$ for as many Bernstein blocks
$\Rep (G)^{\mf s}$ as possible}.

Like for all affine Hecke algebras, there are some constraints on the $q_\alpha$ and $q_{\alpha *}$: 
\begin{itemize}
\item if $X_\alpha, X_\beta \in \Sigma_{\mc O}^\vee$ are $W(\Sigma_{\mc O}^\vee)$-associate,
then $q_\alpha = q_\beta$ and $q_{\alpha *} = q_{\beta *}$,
\item $q_{\alpha *} > 1$ is only possible if $X_\alpha$ is a short root in a type $B_n$ root
system.
\end{itemize}
Notice that $q_\alpha$ and $q_{\alpha *}$ can be expressed in terms of the ``$q$-base" $q_F$
and the labels $\lambda (\alpha), \lambda^* (\alpha)$. It has turned out \cite{KaLu,SolAHA}
that the representation theory of an affine Hecke algebra hardly changes if one replaces
$q_F$ by another $q$-base (in $\R_{>1}$) while keeping all labels fixed. If we replace the
$q$-base $q_F$ by $q_F^r$ and $\lambda (\alpha), \lambda^* (\alpha)$ by 
$\lambda (\alpha) / r, \lambda^* (\alpha) /r$ for some $r \in \R_{>0}$, then $q_\alpha$ and
$q_{\alpha *}$ do not change, and in fact $\mc H (\mc O,G)$ is not affected at all.
In this way one can always scale one of the labels to 1.
Hence the representation theory of $\mc H (\mc O,G)$ depends mainly on the ratios between the 
labels $\lambda (\alpha), \lambda^* (\alpha)$ for $X_\alpha \in \Sigma_{\mc O}^\vee$.
\begin{itemize}
\item For irreducible root systems of type $A_n, D_n$ and $E_n$, $\lambda (\alpha) = 
\lambda^* (\alpha) = \lambda (\beta)$, for any roots $X_\alpha, X_\beta \in 
\Sigma_{\mc O}^\vee$. There is essentially only one label $\lambda (\alpha)$, and it can
be scaled to 1 by fixing $q_\alpha$ but replacing $q_F$ by $q_\alpha$.
\item For the irreducible root systems $C_n, F_4$ and $G_2$, again $\lambda (\alpha)$ always
equals $\lambda^* (\alpha)$. There are two independent labels $\lambda (\alpha)$: one for the
short roots and one for the long roots.
\item For an irreducible root system of type $B_n$, $\lambda^* (\alpha)$ need not equal 
$\lambda (\alpha)$ if $X_\alpha$ is short. Here we have three independent labels: $\lambda (\beta)$
for $X_\beta$ long, $\lambda (\alpha)$ for $X_\alpha$ short and $\lambda^* (\alpha)$ for
$X_\alpha$ short.
\end{itemize} 
Lusztig \cite{Lus-open} has conjectured:

\begin{conj}\label{conj:1}
Let $G$ be a reductive group over a non-archimedean local field $F$, with an arbitrary Bernstein
block $\Rep (G)^{\mf s}$. Let $\Sigma_{\mc O,j}^\vee$ be an irreducible component of the root
system $\Sigma_{\mc O}^\vee$ underlying $\mc H (\mc O,G)$. Then:
\begin{itemize}
\item[(i)] the $q$-parameters $q_\alpha, q_{\alpha *}$ are powers of $q_F$, except that 
for a short root  $\alpha$ in a type $B_n$ root system the $q$-parameters can also be powers of 
$q_F^{1/2}$ (and then $q_\alpha q_{\alpha^*}^{\pm 1}$ is still a power of $q_F$).
\item[(ii)] the label functions $\lambda, \lambda^*$ on $\Sigma_{\mc O,j}^\vee$ agree with those
obtained in the same way from a Bernstein block of unipotent representations of some adjoint 
simple $p$-adic group, as in \cite{Lus-Uni,Lus-Uni2}.
\end{itemize}
\end{conj}

We note that the affine root systems in Lusztig's notation for affine Hecke algebras
correspond to affine extensions of our root systems $\Sigma_{\mc O}$. Below we list all possible 
label functions from \cite{Lus-Uni,Lus-Uni2}, for a given irreducible root system 
(taking a remark at the end of Paragraph \ref{par:F4} into account).
\begin{table}[h] 
\renewcommand{\tablename}{\ \hspace{-1cm} Table}
\caption{Labels for affine Hecke algebras from unipotent representations \label{fig:1}}
$\begin{array}{cccc}
\Sigma_{\mc O}^\vee & \lambda (\text{long root}) & \lambda (\text{short root})
& \lambda^* (\text{short root}) \\
\hline
A_n, D_n, E_n & - & \in \Z_{>0} & \lambda^* = \lambda \\
B_n & 1 \text{ or } 2 & \in \Z_{>0} & \in \Z_{\geq 0} \\
C_n & \in \Z_{>0} & 1 \text{ or } 2 & \lambda^* = \lambda \\
F_4 & 1 \text{ or } 2 & 1 & 1 \\
F_4 & 1 & 2 & 2 \\
F_4 & 4 & 1 & 1 \\
G_2 & 1 \text{ or } 3 & 1 & 1 \\
G_2 & 1 & 3 & 3 \\
G_2 & 9 & 1 & 1 \\
A_1 & - & 15 & 1 \\
A_1 & - & 9 & 7 
\end{array}$
\end{table}

Conjecture \ref{conj:1}.(i) is related to a conjecture of Langlands about Harish-Chandra 
$\mu$-functions \cite[\S 2]{Sha}. For generic representations of quasi-split reductive groups over
$p$-adic fields, \cite[\S 3]{Sha} translates Conjecture \ref{conj:1}.(i) to a question about poles
of adjoint $\gamma$-factors. (We do not pursue that special case here.)

Motivation for Conjecture \ref{conj:1}.(ii) comes from the local Langlands correspondence. It is
believed \cite{AMS1} that $\Irr (G) \cap \Rep (G)^{\mf s}$ corresponds to a Bernstein component 
$\Phi_e (G)^{\mf s^\vee}$ of enhanced L-parameters for $G$. To $\Phi_e (G)^{\mf s^\vee}$ one can 
canonically associate an affine Hecke algebra $\mc H (\mf s^\vee,q_F^{1/2})$, possibly extended 
with a twisted group algebra \cite[\S 3.3]{AMS3}. It is expected that the module category of 
$\mc H (\mf s^\vee,q_F^{1/2})$ is very closely related to $\Rep (G)^{\mf s}$, at least the two 
subcategories of finite length modules should be equivalent.

The nonextended version $\mc H^\circ (\mf s^\vee,q_F^{1/2})$ of
$\mc H (\mf s^\vee,q_F^{1/2})$ can be constructed with complex geometry from a connected reductive 
group $H^\vee$ (the connected centralizer in $G^\vee$ of the image of the inertia group $\mb I_F$
under the Langlands parameter) and a cuspidal local system $\rho$ on a unipotent orbit
for a Levi subgroup $L^\vee$ of $H^\vee$. The exact same data $(H^\vee,L^\vee,\rho)$ also arise
from enhanced Langlands parameters (for some reductive $p$-adic group $G'$) which are trivial on
$\mb I_F$. By the local Langlands correspondence from \cite{Lus-Uni,Lus-Uni2,SolLLCunip,SolRamif},
a Bernstein component of such enhanced L-parameters corresponds to a Bernstein component 
$\Rep (G')^{\mf s'}$ of unipotent $G'$-representations. 

It follows that $\mc H^\circ (\mf s^\vee,q_F^{1/2})$ is isomorphic to 
$\mc H^\circ (\mf s'^\vee,q_{F'}^{1/2})$.
By \cite[Theorem 4.4]{SolLLCunip}, $\mc H^\circ (\mf s'^\vee,q_{F'}^{1/2})$ is isomorphic to
$\mc H (\mc O',G')$, which is an affine Hecke algebra associated to a Bernstein block of 
unipotent representations of $G'$. If desired one can replace $G'$ by its adjoint group, by
\cite[Lemma 3.5]{SolLLCunip} that operation changes the affine Hecke algebras a little but 
preserves the root systems and the $q$-parameters. 

Thus, if there exists a local Langlands correspondence with good properties, Conjecture \ref{conj:1}
is a consequence of what happens on the Galois side of the correspondence. Conversely, new cases
of Conjecture \ref{conj:1} might contribute to new instances of a local Langlands correspondence,
via a comparison of possible Hecke algebras on both sides as in \cite{Lus-Uni}.\\

An important and accessible class of representations is formed by the principal series 
representations of quasi-split groups $G$. When $G$ is $F$-split, the Hecke algebras for
Bernstein blocks of such representations were already analysed in \cite{Roc1} via types,
under some mild restrictions on the residual characteristic. To every root of a quasi-split
group $G$ (relative to a maximal $F$-split torus) one can associate a splitting field 
$F_\alpha$, a finite extension of $F$.

\begin{thmintro} \label{thm:B}
\textup{(see Theorem \ref{thm:3.9} and Corollary \ref{cor:3.12})} \\
Conjecture \ref{conj:1} holds for all Bernstein blocks in the principal series of a quasi-split
connected reductive group over $F$. For $X_\alpha \in \Sigma_{\mc O}^\vee$ (with one exception
in type ${}^2 A_{2n}$ that we analyse as well)
$q_{\alpha *} = 1$ and $q_\alpha$ is the cardinality of the residue field of $F_\alpha$.
\end{thmintro}

Theorem \ref{thm:B} will be employed to establish a canonical local Langlands correspondence for
principal series representations of quasi-split $F$-groups \cite{SolLLCQ}.

For parameter computations in Hecke algebras associated to more complicated Bernstein components,
we need a reduction strategy. That is the topic of Section \ref{sec:red}, which culminates in:

\begin{thmintro} \label{thm:C}
\textup{(see Corollaries \ref{cor:2.6} and \ref{cor:2.7})} \\
Conjecture \ref{conj:1} holds for $G$ if and only if it holds for its derived group $G_\der$.
If Conjecture \ref{conj:1} holds for the simply connected cover $G_\Sc$ of $G_\der$, then
it holds for $G$.
\end{thmintro}

This enables us to reduce the verification of Conjecture \ref{conj:1} to absolutely simple, 
simply connected groups. For (absolutely) simple groups quite a few results about the parameters
of Hecke algebras can be found in the literature, e.g. \cite{BuKu1,Sec,Hei1}. With our current
framework we can easily generalize those results, in particular from one group to an isogenous
group.\\

S\'echerre and Stevens \cite{Sec,SeSt4,SeSt6} determined the Hecke algebras for all Bernstein
blocks for inner forms of $GL_n (F)$. Together with Theorem \ref{thm:C} that proves Conjecture
\ref{conj:1} for all inner forms of a group of type $A$.

For classical groups (symplectic, special orthogonal, unitary) we run into the problem that
some representation theoretic results have been proven over $p$-adic fields but not (yet) over
local function fields. We overcome this with the method of close fields \cite{Kaz}, which 
Ganapathy recently generalized to arbitrary connected reductive groups \cite{Gan1,Gan2}.

\begin{thmintro} \label{thm:D}
\textup{(see Corollary \ref{cor:4.5})} \\
Let $\Rep (G)^{\mf s}$ be a Bernstein block for a reductive group $G$ over a local function
field. Then there exists a Bernstein block $\Rep (\tilde G )^{\tilde{\mf s}}$ for a reductive
group $\tilde G$ over a $p$-adic field, such that:
\begin{itemize}
\item $G$ and $\tilde G$ come from ``the same" algebraic group,
\item $\Rep (G)^{\mf s} \cong \Rep (\tilde G)^{\tilde{\mf s}}$ and
$\mc H (\mc O,G) \cong \mc H (\tilde{\mc O},\tilde G )$,
\item the parameters for both these affine Hecke algebras are the same.
\end{itemize}
\end{thmintro}

For classical groups over $p$-adic fields the parameters of the Hecke algebras were determined in 
\cite{Hei1,Hei3}, in terms of M\oe glin's classification of discrete series representations \cite{Moe3}. 
With a generalization of this method and a closer analysis of the resulting parameters we prove:

\begin{thmintro}\label{thm:E}
\textup{(see Paragraph \ref{par:classical})} \\
Conjecture \ref{conj:1} holds for all pure inner forms of quasi-split classical groups, and for all 
groups isogenous with one of those. This includes all simple groups of type $A_n, B_n, C_n, D_n$, 
except those associated to Hermitian forms on vector spaces over quaternion algebras.
\end{thmintro}

Theorem \ref{thm:E} is useful to study Hecke algebras and the local Langlands 
correspondence for general spin groups \cite{AMS4}.
Among classical groups associated to Hermitian forms, Conjecture \ref{conj:1} only remains open
for the non-pure inner forms of quasi-split classical groups. Unfortunately, the current understanding 
of their representations does not suffice to carry out the strategies we applied to other groups.\\

Finally, we consider exceptional groups. For most Bernstein components we can reduce the computation
of the Hecke algebra parameters to groups of Lie type $A_n,B_n,C_n$ and $D_n$, but sometimes that
does not work. We establish partial results for all simple exceptional groups, most of which can be
summarized as follows:

\begin{thmintro}\label{thm:F}
\textup{(see Paragraphs \ref{par:G2}, \ref{par:F4} and \ref{par:E})} \\
Conjecture \ref{conj:1} holds for all simple $F$-groups of type $G_2, F_4, E_6, {}^2 E_6, E_6^{(3)}, 
{}^3 D_4$. 

If (for any reductive $p$-adic group $G$) $\Sigma_{\mc O}^\vee$ has an irreducible component 
$\Sigma_{\mc O,j}^\vee$ of type $F_4$, then Conjecture \ref{conj:1} holds for $\Sigma_{\mc O,j}^\vee$.
\end{thmintro}

Our results about $F_4$ are useful in combination with \cite[\S 6]{SolHecke}. There we related 
the irreducible representations of an affine Hecke algebra with arbitrary positive $q$-parameters to 
the irreducible representations of the analogous algebra that has all $q$-parameters equal to 1.
The problem was only that we could not handle certain label functions for type $F_4$ root systems.
Theorem \ref{thm:F} shows that the label functions which could be handled well in 
\cite[\S 6]{SolHecke} exhaust the label functions that can appear for type $F_4$ root systems among 
affine Hecke algebras coming from reductive $p$-adic groups.\\

\textbf{Acknowledgements.}\\
We are grateful to Anne-Marie Aubert, Geo Tam and Stefan Dawydiak for their comments on earlier 
versions, and in particular for pointing out some problems. 
We thank the referee for his or her report and comments, which helped to clarify some parts.

\renewcommand{\theequation}{\arabic{section}.\arabic{equation}}
\counterwithin*{equation}{section}

\section{Progenerators and endomorphism algebras for Bernstein blocks}
\label{sec:progen}

We fix some notations and recall relevant material from \cite{SolEnd}. 
Let $F$ be a non-archimedean local field with ring of integers $\mf o_F$. Pick a uniformizing 
element $\varpi_F \in \mf o_F$. We denote the cardinality of the residue field 
$k_F = \mf o_F / \varpi \mf o_F$ by $q_F$. Let $|\cdot|_F$ be the norm on $F$,
normalized so that $|\varpi_F |_F = q_F^{-1}$.

Let $\mc G$ be a connected reductive $F$-group and let $G = \mc G (F)$ be its group of
$F$-rational points. We briefly call $G$ a reductive $p$-adic group. We consider the
category $\Rep (G)$ of smooth $G$-representations on complex vector spaces. Let $\Irr (G)$
be the set of equivalence classes of irreducible objects in Rep$(G)$, and $\Irr_\cusp (G)
\subset \Irr (G)$ the subset of supercuspidal representations.

Let $\mc M$ be a $F$-Levi subgroup of $\mc G$ and write $M = \mc M (F)$. The group of 
un\-ra\-mi\-fied characters of $M$ is denoted $X_\nr (M)$. We fix $(\sigma,E) \in \Irr_\cusp (M)$.
The set of unramified twists of $\sigma$ is
\[
\mc O = \{ \sigma \otimes \chi : \chi \in X_\nr (M) \} \subset \Irr (M) .
\]
It can be identified with the inertial equivalence class $\mf s_M = [M,\sigma ]_M$. Let
$\mf s = [M,\sigma ]_G$ be the associated inertial equivalence class for $G$.

Recall that the supercuspidal support Sc$(\pi)$ of $\pi \in \Irr (G)$ consists of a Levi
subgroup of $G$ and an irreducible supercuspidal representation thereof. Although Sc$(\pi)$
is only defined up to $G$-conjugacy, we shall only be interested in supercuspidal supports
with Levi subgroup $M$, and then the supercuspidal representation is uniquely defined up
to the natural action of $N_G (M)$ on $\Irr (M)$.

This setup yields a Bernstein component
\[
\Irr (G)^{\mf s} = \{ \pi \in \Irr (G) : \mr{Sc}(\pi) \in (M,\mc O) \}
\]
of $\Irr (G)$. It generates a Bernstein block $\Rep (G)^{\mf s}$ of $\Rep (G)$, see \cite{BeDe}.

Let $M^1 \subset M$ be the group generated by all compact subgroups of $M$, so that 
$X_\nr (M) = \Irr (M / M^1)$. Then
\begin{equation}\label{eq:1.1}
\mr{ind}_{M^1}^M (\sigma,E) \cong E \otimes_\C \C [M / M^1] \cong E \otimes_\C \C [X_\nr (M)] , 
\end{equation}
where $\C [M / M^1]$ is the group algebra of the discrete group $M / M^1$ and $\C [X_\nr (M)]$ 
is the ring of regular functions on the complex torus $X_\nr (M)$. Supercuspidality implies 
that \eqref{eq:1.1} is a progenerator of $\Rep (M)^{\mf s_M}$. 
Let $P \subset G$ be a parabolic subgroup with Levi factor $M$, chosen as prescribed by
\cite[Lemma 9.1]{SolEnd}. Let 
\[
I_P^G : \Rep (M) \to \Rep (G)
\] 
be the parabolic induction functor, normalized so that it preserves unitarity. As a consequence of 
Bernstein's second adjointness theorem, $I_P^G$ preserves projectivity. The representation
\[
\Pi^{\mf s} := I_P^G (E \otimes \C [X_\nr (M)])
\]
is a progenerator of $\Rep (G)^{\mf s}$, see \cite[\S III.4.1]{BeRu} or 
\cite[Th\'eor\`eme VI.10.1]{Ren}. That means \cite[Theorem 1.8.2.1]{Roc2} that the functor
\[
\begin{array}{ccc}
\Rep (G)^{\mf s} & \longrightarrow & \End_G (\Pi^{\mf s}) -\Mod \\
V & \mapsto & \Hom_G (\Pi^{\mf s}, V) 
\end{array}
\]
is an equivalence of categories. This motivates the study of the endomorphism algebra 
$\End_G (\Pi^{\mf s})$, which was carried out in \cite{Roc2,Hei2,SolEnd}. To be precise,
our $\Pi^{\mf s}$ is a direct sum of finitely many copies of the progenerators
studied (in special cases) in \cite{Roc2,Hei2}. To describe the structure of
$\End_G (\Pi^{\mf s})$, we have to recall several objects which lead to the appropriate 
root datum. The set 
\[
X_\nr (M,\sigma) = \{ \chi \in X_\nr (M) : \sigma \otimes \chi \cong \chi \} 
\]
is a finite subgroup of $X_\nr (M)$. The map
\[
X_\nr (M) / X_\nr (M,\sigma) \to \mc O : \chi \mapsto \sigma \otimes \chi
\]
is a bijection, and in this way we provide $\mc O$ with the structure of a complex variety
(a torus, but without a canonical base point). The group
\[
M_\sigma^2 := \bigcap\nolimits_{\chi \in X_\nr (M,\sigma)} \ker \chi
\]
has finite index in $M$, and there are natural isomorphisms
\begin{align*}
& \Irr (M_\sigma^2 / M^1) \cong X_\nr (M) / X_\nr (M,\sigma) ,\\
& \C [M_\sigma^2 / M^1] \cong \C [ X_\nr (M) / X_\nr (M,\sigma) ] .
\end{align*}
Here and later on, the notation $\C [?]$ must be interpreted as in \eqref{eq:1.1}.
The group 
\[
W(G,M) := N_G (M) / M
\]
is a Weyl group in most cases (and if it is not, then it is still very close to a Weyl group). 
The natural action of $N_G (M)$ on $\Rep (M)$ induces an action of $W(G,M)$ on $\Irr (M)$. 
Let $N_G (M,\mc O)$ be the stabilizer of $\mc O$ in $N_G (M)$ and write
\[
W(M,\mc O) = N_G (M,\mc O) / M .
\]
Thus $W(M,\mc O)$ acts naturally on the complex algebraic variety $\mc O$. This finite group
figures prominently in the Bernstein theory, for instance because the centres of 
$\Rep (G)^{\mf s}$ and of $\End_G (\Pi^{\mf s})$ are naturally isomorphic with 
$\C [\mc O]^{W(M,\mc O)}$.

Let $\mc A_M$ be the maximal $F$-split torus in $Z(\mc M)$, put $A_M = \mc A_M (F)$ and let $
X_* (\mc A_M) = X_* (A_M)$ be the cocharacter lattice. We write 
\[
\mf a_M = X_* (A_M) \otimes_\Z \R \quad \text{and} \quad \mf a_M^* = X^* (A_M) \otimes_\Z \R .
\]
Let $\Sigma (G,A_M) \subset X^* (A_M)$ be the set of nonzero weights occurring 
in the adjoint representation of $A_M$ on the Lie algebra of $G$, and let $\Sigma_\red (A_M)$
be the set of indivisible elements therein.

For every $\alpha \in \Sigma_\red (A_M)$ there is a unique Levi subgroup $M_\alpha$ of $G$ which 
contains $M$ and the root subgroups $U_\alpha, U_{-\alpha}$, and whose semisimple rank is one 
higher than that of $M$. Let $\alpha^\vee \in \mf a_M$ be the unique element which is orthogonal 
to $X^* (A_{M_\alpha})$ and satisfies $\langle \alpha^\vee, \alpha \rangle = 2$.

Recall the Harish-Chandra $\mu$-functions from \cite[\S 1]{Sil2} and \cite[\S V.2]{Wal}. 
The restriction of $\mu^G$ to $\mc O$ is a rational, $W(M,\mc O)$-invariant function on $\mc O$ 
\cite[Lemma V.2.1]{Wal}. It determines a reduced root system \cite[Proposition 1.3]{Hei2}
\begin{equation}\label{eq:1.5}
\Sigma_{\mc O,\mu} = \{ \alpha \in \Sigma_\red (A_M) : \mu^{M_\alpha} \text{ has a zero on } \mc O \}.
\end{equation}
For $\alpha \in \Sigma_\red (A_M)$ the function $\chi \mapsto \mu^{M_\alpha}(\sigma \otimes \chi)$ 
factors through the quotient map $X_\nr (M) \to X_\nr (T_\alpha)$, where $T_\alpha$ is the 
onedimensional subtorus of $A_M$ with Lie algebra spanned by (a multiple of) $\alpha^\vee$. 
The associated system of coroots is
\[
\Sigma_{\mc O,\mu}^\vee = \{ \alpha^\vee \in \mf a_M :  \mu^{M_\alpha} \text{ has a zero on } \mc O \}.
\]
By the aforementioned $W(M,\mc O)$-invariance of $\mu^G$, $W(M,\mc O)$ acts naturally on
$\Sigma_{\mc O,\mu}$ and on $\Sigma_{\mc O,\mu}^\vee$.
Let $s_\alpha$ be the unique nontrivial element of $W(M_\alpha,M)$. By \cite[Proposition 1.3]{Hei2}
the Weyl group $W(\Sigma_{\mc O,\mu})$ can be identified with the subgroup of $W(G,M)$ generated
by the reflections $s_\alpha$ with $\alpha \in \Sigma_{\mc O,\mu}$, and as such it is a normal
subgroup of $W(M,\mc O)$.

The parabolic subgroup $P = MU$ of $G$ determines a set of positive roots $\Sigma_{\mc O,\mu}^+$ 
and a basis $\Delta_{\mc O,\mu}$ of $\Sigma_{\mc O,\mu}$. Let $\ell_{\mc O}$ be the length function on
$W(\Sigma_{\mc O,\mu})$ specified by $\Delta_{\mc O,\mu}$. Since $W(M,\mc O)$ acts on 
$\Sigma_{\mc O,\mu}$, $\ell_{\mc O}$ extends naturally to $W(M,\mc O)$, by
\[
\ell_{\mc O}(w) = | w (\Sigma_{\mc O,\mu}^+) \cap -\Sigma_{\mc O,\mu}^+ |. 
\]
The set of positive roots also determines a subgroup of $W(M,\mc O)$:
\begin{equation}\label{eq:3.19}
\begin{aligned}
R(\mc O) & = \{ w \in W(M,\mc O) : w (\Sigma_{\mc O,\mu}^+) = \Sigma_{\mc O,\mu}^+ \} \\
& = \{ w \in W(M,\mc O) : \ell_{\mc O}(w) = 0 \} .
\end{aligned}
\end{equation}
The simple transitivity of the action of $W(\Sigma_{\mc O,\mu})$ on the set of positive systems of
$\Sigma_{\mc O,\mu}$ \cite[Theorem 1.8]{Hum} implies that
\begin{equation}\label{eq:3.8}
W(M,\mc O) = R (\mc O) \ltimes W(\Sigma_{\mc O,\mu})  . 
\end{equation}
Recall that $X_\nr (M) / X_\nr (M,\sigma)$ is isomorphic to the character group of the lattice 
$M_\sigma^2 / M^1$. Since $M_\sigma^2$ depends only on $\mc O$, it is normalized by $N_G (M,\mc O)$. 
In particular the conjugation action of $N_G (M,\mc O)$ on $M_\sigma^2 / M^1$ induces an action of 
$W(M,\mc O)$ on $M_\sigma^2 / M^1$.

Let $h^\vee_\alpha$ be the unique 
generator of $(M_\sigma^2 \cap M_\alpha^1) / M^1 \cong \Z$ such that $|\alpha (h^\vee_\alpha)|_F > 1$. 
Recall the injective homomorphism $H_M : M/M^1 \to \mf a_M$ defined by 
\[
q_F^{\langle H_M (m), \gamma \rangle} = |\gamma (m) |_F \qquad \text{for } m \in M, \gamma \in X^* (M).
\]
\textbf{Remark.} \emph{This definition is motivated by the correction to \cite{SolEnd}. 
In earlier versions we used an alternative convention, which differs from the above by multiplying 
$h_\alpha^\vee$ and $H_M$ by a factor -1. That does not change the Hecke algebras, it only amounts 
to a different choice of generators.}\\

In these terms $H_M (h^\vee_\alpha) \in \R_{>0} \alpha^\vee$. Since $M_\sigma^2$ has finite index in
$M$, $H_M (M_\sigma^2 / M^1)$ is a lattice of full rank in $\mf a_M$. We write 
\[
(M_\sigma^2 / M^1)^\vee = \Hom_\Z (M_\sigma^2 / M^1, \Z) .
\]
Composition with $H_M$ and $\R$-linear extension of maps $H_M (M_\sigma^2 / M^1) \to \Z$ determines 
an embedding 
\[
H_M^\vee : (M_\sigma^2 / M^1)^\vee \to \mf a_M^*.
\]
Then $H_M^\vee (M_\sigma^2 / M^1)^\vee$ is a lattice of full rank in $\mf a_M^*$.

\begin{prop}\label{prop:3.5} \textup{\cite[Proposition 3.5]{SolEnd}}\\
Let $\alpha \in \Sigma_{\mc O,\mu}$.
\enuma{
\item For $w \in W(M,\mc O)$: $w(h^\vee_\alpha) = h^\vee_{w(\alpha)}$.
\item There exists a unique $\alpha^\sharp \in (M_\sigma^2 / M^1)^\vee$ such that 
$H_M^\vee (\alpha^\sharp) \in \R \alpha$ and $\langle h^\vee_\alpha, \alpha^\sharp \rangle = 2$.
\item Write 
\[
\begin{array}{lll}
\Sigma_{\mc O} & = & \{ \alpha^\sharp : \alpha \in \Sigma_{\mc O,\mu} \} , \\
\Sigma^\vee_{\mc O} & = & \{ h^\vee_\alpha : \alpha \in \Sigma_{\mc O,\mu} \} .
\end{array}
\]
Then $(\Sigma_{\mc O}^\vee, M_\sigma^2 / M^1, \Sigma_{\mc O}, (M_\sigma^2 / M^1)^\vee)$
is a root datum with Weyl group $W(\Sigma_{\mc O,\mu})$.
\item The group $W(M,\mc O)$ acts naturally on this root datum, and $R(\mc O)$ is the
stabilizer of the basis $\Delta_{\mc O}^\vee$ determined by $P$.
}
\end{prop}

We note that $\Sigma_{\mc O}$ and $\Sigma_{\mc O}^\vee$ have almost the same type as
$\Sigma_{\mc O,\mu}$. Indeed, the roots $H_M^\vee (\alpha^\sharp)$ are scalar multiples
of the $\alpha \in \Sigma_{\mc O,\mu}$, so the angles between the elements of $\Sigma_{\mc O}$
are the same as the angles between the corresponding elements of $\Sigma_{\mc O,\mu}$. 
It follows that every irreducible component of $\Sigma_{\mc O,\mu}$ has the same type
as the corresponding components of $\Sigma_{\mc O}$ and $\Sigma_{\mc O}^\vee$, except
that type $B_n / C_n$ might be replaced by type $C_n / B_n$.

For $\alpha \in \Sigma_\red (M) \setminus \Sigma_{\mc O,\mu}$, the function $
\mu^{M_\alpha}$ is constant on $\mc O$. In contrast, for $\alpha \in \Sigma_{\mc O,\mu}$ 
it has both zeros and poles on $\mc O$. By \cite[\S 5.4.2]{Sil2} 
\begin{equation}\label{eq:3.4}
s_\alpha \cdot \sigma' \cong \sigma' \quad \text{whenever } \mu^{M_\alpha}(\sigma') = 0 .
\end{equation}
As $\Delta_{\mc O,\mu}$ is linearly independent in $X^* (A_M)$ and $\mu^{M_\alpha}$ factors
through $A_M / A_{M_\alpha}$, there exists a $\tilde \sigma \in \mc O$ such that 
$\mu^{M_\alpha}(\tilde \sigma) = 0$ for all $\alpha \in \Delta_{\mc O,\mu}$. In view of 
\cite[\S 1]{Sil3} this can even be achieved with a unitary $\tilde \sigma$. 
We replace $\sigma$ by $\tilde \sigma$, which means that from now on we adhere to:

\begin{cond}\label{cond:3.1}
$(\sigma,E) \in \Irr (M)$ is unitary supercuspidal and $\mu^{M_\alpha}(\sigma) = 0$ 
for all $\alpha \in \Delta_{\mc O,\mu}$. We identify $X_\nr (M) / X_\nr (M,\sigma)$ with 
$\mc O$ via $\chi \mapsto \sigma \otimes \chi$.
\end{cond}

By \eqref{eq:3.4} the entire Weyl group $W(\Sigma_{\mc O,\mu})$ stabilizes the isomorphism 
class of this $\sigma$. However, in general $R(\mc O)$ need not stabilize $\sigma$. We define
\begin{equation}\label{eq:3.20}
X_\alpha \in \C [X_\nr (M) / X_\nr (M,\sigma)] \cong \C [\mc O] \quad \text{by} \quad 
X_\alpha (\chi) = X_\alpha (\sigma \otimes \chi)  = \chi (h_\alpha^\vee) .
\end{equation}
For any $w \in W(M,\mc O)$ which stabilizes $\sigma$ in $\Irr (M)$, Proposition 
\ref{prop:3.5}.a implies
\begin{equation}\label{eq:3.21}
w (X_\alpha) = X_{w (\alpha)} \quad \text{for all } \alpha \in \Sigma_{\mc O,\mu} . 
\end{equation}
According to \cite[\S 1]{Sil2} there exist $q_\alpha, q_{\alpha*} \in \R_{\geq 1}$,
$c'_{s_\alpha} \in \R_{>0}$ for $\alpha \in \Sigma_{\mc O,\mu}$, such that
\begin{equation}\label{eq:3.22}
\mu^{M_\alpha}(\sigma \otimes \cdot) = 
\frac{c'_{s_\alpha}  (1 - X_\alpha) (1 - X_\alpha^{-1})}{(1 - q_\alpha^{-1} X_\alpha)
(1 - q_\alpha^{-1} X_\alpha^{-1})} \frac{(1 + X_\alpha) (1 + X_\alpha^{-1})
}{(1 + q_{\alpha*}^{-1} X_\alpha)(1 + q_{\alpha*}^{-1} X_\alpha^{-1})} 
\end{equation}
as rational functions on $X_\nr (M) / X_\nr (M,\sigma) \cong \mc O$.
We may modify the choice of $\sigma$ in Condition \ref{cond:3.1}, so that, as in 
\cite[Remark 1.7]{Hei2}:
\begin{equation}\label{eq:3.25}
q_\alpha \geq q_{\alpha*} \text{ for all } \alpha \in \Delta_{\mc O,\mu} .
\end{equation}
Then \cite[Lemma 3.4]{SolEnd} guarantees that the maps $\Sigma_{\mc O,\mu} \to \R_{\geq 0}$
given by $q_\alpha$ and $q_{\alpha *}$ are $W(M,\mc O)$-invariant.
Comparing \eqref{eq:3.22}, Condition \ref{cond:3.1} and \eqref{eq:3.25}, we see that 
$q_{\alpha} > 1$ for all $\alpha \in \Sigma_{\mc O,\mu}$.
In particular the zeros of $\mu^{M_\alpha}$ occur at
\[
\{ X_\alpha = 1 \} = \{ \sigma' \in \mc O : X_\alpha (\sigma') = 1 \} 
\]
and sometimes at
\[
\{ X_\alpha = -1 \} = \{ \sigma' \in \mc O : X_\alpha (\sigma') = -1 \} . 
\]
When $\mu^{M_\alpha}$ has a zero at both $\{ X_\alpha = 1 \}$ and $\{ X_\alpha = -1 \}$, the 
irreducible component of $\Sigma_{\mc O}^\vee$ containing $h^\vee_\alpha$ has type
$B_n \; (n \geq 1)$ and $h^\vee_\alpha$ is a short root \cite[Lemma 3.3]{SolEnd}.

For another characterization of $\mu_\alpha$, we write down an explicit 
construction. Let $\delta_P : P \to \R_{>0}$ be the modular function. 
We realize $I_P^G (\sigma \otimes \chi,E)$ on the vector space
\[
\big\{ f : G \to E \mid f \text{ is smooth}, f(umg) = \sigma (m) (\chi \delta_P^{1/2})(m) f(g) \;
\forall u \in U, m \in M, g \in G \big\} ,
\]
with $G$ acting by right translations. Let $P' = M U'$ be another parabolic subgroup of $G$ with 
Levi factor $M$. Following \cite[\S IV.1]{Wal} we consider the map
\begin{equation}\label{eq:1.2}
\begin{array}{cccc}
J_{P'|P}(\sigma \otimes \chi) : & I_P^G (\sigma \otimes \chi,E) & \to & I_{P'}^G (\sigma \otimes \chi,E) \\
& f & \mapsto & [g \mapsto \int_{(U \cap U') \backslash U'} f (u' g) \textup{d} u']  .
\end{array}
\end{equation}
Here d$u'$ denotes a quotient of Haar measures on $U'$ and $U \cap U'$. This integral converges for
$\chi$ in an open subset of $X_\nr (M)$ (independent of $f$). As such it defines a map
\[
\begin{array}{ccc}
X_\nr (M) \times I_P^G (E) & \to & I_{P'}^G (E) , \\
(\chi, f) & \mapsto & J_{P'|P}(\sigma \otimes \chi) f ,
\end{array}
\]
which is rational in $\chi$ and linear in $f$ \cite[Th\'eor\`eme IV.1.1]{Wal}. Moreover it intertwines
the $G$-representation $I_P^G (\sigma \otimes \chi)$ with $I_{P'}^G (\sigma \otimes \chi)$ whenever 
it converges. Then
\[
J_{P|P'}(\sigma \otimes \chi) J_{P'|P}(\sigma \otimes \chi) \in \End_G (I_P^G (\sigma \otimes \chi,E))
= \C \, \mr{id} ,
\]
at least for $\chi$ in a Zariski-open subset of $X_\nr (M)$. For any $\alpha \in \Sigma_\red (M)$ there
exists by construction \cite[\S IV.3]{Wal} a nonzero constant such that
\begin{equation}\label{eq:1.3}
J_{M_\alpha \cap P | s_\alpha (M_\alpha \cap P)}(\sigma \otimes \chi) J_{s_\alpha (M_\alpha \cap P) | 
M_\alpha \cap P}(\sigma \otimes \chi) = \frac{\mr{constant}}{\mu^{M_\alpha}(\sigma \otimes \chi)} ,
\end{equation}
as rational functions of $\chi \in X_\nr (M)$. We note that 
\[
(U \cap s_\alpha (U)) \backslash s_\alpha (U) = U_{-\alpha} \quad \text{and} \quad
(U \cap s_\alpha (U)) \backslash U = U_\alpha ,
\]
where $U_{\pm \alpha}$ denotes a root subgroup with respect to $A_M$. That allows us to simplify 
\eqref{eq:1.3} to
\begin{equation}\label{eq:1.4}
\begin{array}{lll}
J_{s_\alpha (M_\alpha \cap P) | M_\alpha \cap P}(\sigma \otimes \chi) f & = &
[g \mapsto \int_{U_{-\alpha}} f (u_- g) \textup{d} u_- ], \\
J_{M_\alpha \cap P | s_\alpha (M_\alpha \cap P)}(\sigma \otimes \chi) f & = &
[g \mapsto \int_{U_{\alpha}} f (u_+ g) \textup{d} u_+ ], 
\end{array}
\end{equation}
where d$u_\pm$ is a Haar measure on $U_{\pm \alpha}$. The numbers $q_\alpha, q_\alpha^{-1}$ (and
$q_{\alpha*}, q_{\alpha *}^{-1}$ when $q_{\alpha*} \neq 1$) are precisely the values of
$X_\alpha (\chi) = X_\alpha (\sigma \otimes \chi)$ at which $\mu^{M_\alpha}(\sigma \otimes \chi)$ has 
a pole, and in view of \eqref{eq:1.3} these are also given by the $\chi$ for which
\[
J_{M_\alpha \cap P | s_\alpha (M_\alpha \cap P)}(\sigma \otimes \chi) J_{s_\alpha (M_\alpha \cap P) | 
M_\alpha \cap P}(\sigma \otimes \chi) = 0 .
\]
For other non-unitary $\sigma \otimes \chi \in \mc O$ the operators \eqref{eq:1.4} are invertible,
and by the Langlands classfication \cite[Th\'eor\`eme VII.4.2]{Ren} 
$I_{P \cap M_\alpha}^{M_\alpha} (\sigma \otimes \chi)$ is irreducible.

\begin{cor}\label{cor:1.1}
The poles of $\mu^{M_\alpha}$ are precisely the non-unitary $\sigma \otimes \chi \in \mc O$  
for which $I_{P \cap M_\alpha}^{M_\alpha} (\sigma \otimes \chi)$ is reducible.
\end{cor}

We endow the based root datum
\[
\big( \Sigma_{\mc O}^\vee, M_\sigma^2 / M^1, \Sigma_{\mc O}, 
(M_\sigma^2 / M^1)^\vee, \Delta_{\mc O}^\vee \big)
\]
with the parameter $q_F$ and the labels 
\[
\lambda (\alpha) = \log (q_\alpha q_{\alpha *}) / \log (q_F) ,\quad
\lambda^* (\alpha) = \log (q_\alpha q_{\alpha *}^{-1}) / \log (q_F) .
\]
To avoid ambiguous terminology, we will call the $q_\alpha$ and $q_{\alpha*}$ $q$-parameters
and refer to $q_F$ as the $q$-base. Replacing the $q$-base by another real number
$>1$ hardly changes the representation theory of Hecke algebras.

To these data we associate the affine Hecke algebra 
\[
\mc H (\mc O, G) = \mc H \big( \Sigma_{\mc O}^\vee, M_\sigma^2 / M^1, \Sigma_{\mc O}, 
(M_\sigma^2 / M^1)^\vee, \lambda, \lambda^*, q_F \big) .
\]
By definition it is the vector space 
\[
\C [M_\sigma^2 / M^1] \otimes_\C \C [W(\Sigma_{\mc O,\mu})]
\]
with multiplication given by the following rules:
\begin{itemize}
\item $\C [M_\sigma^2 / M^1] \cong \C [\mc O]$ is embedded as subalgebra,
\item $\C [W(\Sigma_{\mc O,\mu})] = \mr{span} \{ T_w : w \in W(\Sigma_{\mc O,\mu}) \}$ is embedded
as the Iwahori--Hecke algebra $\mc H ( W(\Sigma_{\mc O,\mu}), q_F^\lambda )$, that is,
\begin{align*}
& T_w T_v = T_{wv} \quad \text{if } \ell_{\mc O}(w) + \ell_{\mc O}(v) = \ell_{\mc O}(wv),  \\
& (T_{s_\alpha} + 1)(T_{s_\alpha} - q_F^{\lambda (\alpha)}) = (T_{s_\alpha} + 1)(T_{s_\alpha} - 
q_\alpha q_{\alpha *}) = 0 \quad \text{if } \alpha \in \Delta_{\mc O,\mu} ,
\end{align*}
\item for $\alpha \in \Delta_{\mc O,\mu}$ and $m \in M_\sigma^2 / M^1$ (corresponding to
$X_m \in \C [M_\sigma^2 / M^1]$):
\[
X_m T_{s_\alpha} - T_{s_\alpha} X_{s_\alpha (m)} =
\big( q_\alpha q_{\alpha *} - 1 + X_\alpha^{-1} (q_\alpha - q_{\alpha *}) \big) 
\frac{X_m - X_{s_\alpha (m)}}{1 - X_\alpha^{-2}} .
\]
\end{itemize}
This affine Hecke algebra is related to $\End_G (\Pi^{\mf s})$ in the following way.
Let $\End_G^\circ (\Pi^{\mf s})$ be the subalgebra of $\End_G (\Pi^{\mf s})$ built, as in 
\cite[\S 5.2]{SolEnd}, using only $\C [X_\nr (M)]$,\\ 
$X_\nr (M,\sigma)$ and $W(\Sigma_{\mc O,\mu})$ -- so omitting $R(\mc O)$. 
By \cite[Corollary 5.8]{SolEnd} there exist elements
$\mc T_r \in \End_G (\Pi^{\mf s})^\times$ for $r \in R(\mc O)$, such that
\begin{equation}\label{eq:1.6}
\End_G (\Pi^{\mf s}) = \bigoplus\nolimits_{r \in R (\mc O)} \End_G^\circ (\Pi^{\mf s}) \mc T_r .
\end{equation}
The calculations in \cite[\S 6--8]{SolEnd} apply also to $\End_G^\circ (\Pi^{\mf s})$ and they
imply, as in \cite[Corollary 9.4]{SolEnd}, an equivalence of categories 
\begin{equation}\label{eq:1.7}
\End_G^\circ (\Pi^{\mf s}) -\fMod \longleftrightarrow \mc H (\mc O,G) -\fMod .
\end{equation}
Here $-\fMod$ denotes the category of finite length right modules. To go from 
$\End_G^\circ (\Pi^{\mf s}) -\fMod$ to $\End_G (\Pi^{\mf s}) -\fMod$ is basically an instance of 
Clifford theory for a finite group
acting on an algebra. In reality it is more complicated \cite[\S 9]{SolEnd}, but still relatively
easy. Consequently the essence of the representation theory of $\End_G (\Pi^{\mf s})$ (and thus
of $\Rep (G)^{\mf s}$) is contained in the affine Hecke algebra $\mc H (\mc O,G)$.

Slightly better results can be obtained if we assume that the restriction of $(\sigma,E)$ to
$M^1$ decomposes without multiplicities bigger than one -- which by \cite[Remark 1.6.1.3]{Roc1} holds 
for very large classes of reductive $p$-adic groups. Assuming it for $(\sigma,E)$, 
\cite[Theorem 10.9]{SolEnd} says that there exist:
\begin{itemize}
\item a smaller progenerator $(\Pi^{\mf s})^{X_\nr (M,\sigma)}$ of $\Rep (G)^{\mf s}$,
\item a Morita equivalent subalgebra $\End_G \big( (\Pi^{\mf s})^{X_\nr (M,\sigma)} \big)$ of
$\End_G (\Pi^{\mf s})$,
\item a subalgebra $\End_G^\circ \big( (\Pi^{\mf s})^{X_\nr (M,\sigma)} \big)$ of
$\End_G \big( (\Pi^{\mf s})^{X_\nr (M,\sigma)} \big)$, which is canonically isomorphic with
$\mc H (\mc O,G)$,
\item elements $J_r \in \End_G \big( (\Pi^{\mf s})^{X_\nr (M,\sigma)} \big)^\times$ for $r \in R(\mc O)$,
such that
\[
\End_G \big( (\Pi^{\mf s})^{X_\nr (M,\sigma)} \big) = 
\bigoplus\nolimits_{r \in R (\mc O)} \End_G^\circ \big( (\Pi^{\mf s})^{X_\nr (M,\sigma)} \big) J_r .
\]
\end{itemize}
As announced in the introduction, we want to determine the parameters $q_\alpha, q_{\alpha *}$ for 
$\alpha \in \Delta_{\mc O,\mu}$, or equivalently the label functions $\lambda, \lambda^* : 
\Sigma_{\mc O,\mu} \to \R_{\geq 0}$ of $\mc H (\mc O,G)$. 

When $\Sigma_{\mc O,\mu}$ is empty, $\mc H (\mc O,G) \cong \C [\mc O]$ and it does not have
parameters or labels. When $\Sigma_{\mc O,\mu} = \{ \alpha, -\alpha\}$, it can
already be quite difficult to identify $q_\alpha$ and $q_{\alpha*}$. 
For instance, when $G$ is split of type $G_2$ and $M$ has semisimple rank one, we did not manage 
to compute $q_\alpha$ and $q_{\alpha*}$ for all supercuspidal representations of $M$. 
(This was achieved later in \cite{AuXu}, and it fits with Conjecture \ref{conj:1}.) 

Yet, for $\mc H (\mc O,G)$ this is hardly troublesome. Namely, any affine Hecke algebra $\mc H$
with $\Sigma_{\mc O,\mu} = \{\alpha, -\alpha\}$ and $q_\alpha, q_{\alpha*} \in \C \setminus \! \{0,-1\}$ 
can be analysed very well. Firstly, one can determine all its irreducible representations directly,
as done in \cite[\S 2.2]{SolHecke}. From \cite[Theorem 2.4]{SolHecke} one sees that as far as
representation theory is concerned there are essentially only two different labels for such an 
algebra $\mc H$: $\lambda (\alpha) = \lambda^* (\alpha)$ and $\lambda (\alpha) \neq
\lambda^* (\alpha)$. The cases with $\lambda = \lambda^*$ arise from the Iwahori-spherical 
representations of a split reductive $p$-adic group of semisimple rank 1, and cases with
$\lambda \neq \lambda^*$ can be obtained for instance from the Iwahori-spherical
representations of a unitary group $U_3 (F)$. 

Secondly, with \cite{Lus-Gr} the representation theory of
$\mc H$ can be reduced to that of two graded Hecke algebras $\mh H_k$ with root system of
rank $\leq 1$. One of them has label $k_\alpha = \log (q_\alpha) / \log (q_F)$ and 
underlying vector space $T_1 (\mc O)$, while the other has label $k_{\alpha*} = \log (q_{\alpha*}) /
\log (q_F)$ and underlying vector space $T_{\chi_-}(\mc O)$ for some $\chi_- \in \mc O$
with $X_\alpha (\chi_-) = -1$. 

Recall that any graded Hecke algebra $\mh H_k$ is isomorphic to $\mh H_{\epsilon k}$ with 
$\epsilon \in \C^\times$ via a scaling isomorphism \cite[(1.15)]{SolAHA}.
For graded Hecke algebras with root system $\{\alpha,-\alpha\}$ and a fixed underlying vector
space $\C \alpha^\vee \oplus \C^d$, there are just two isomorphism classes: one with label 
$k \neq 0$ and one with label $k = 0$. For both there is a nice geometric construction of the 
irreducible representations of $\mh H_k$, see \cite{Lus-Cusp1} and \cite[Theorem 3.11]{AMS2}. 
These two graded Hecke algebras arise already from the Iwahori-spherical representations of 
$SL_2 (F) \times (F^\times)^d$: $\mh H_k$ with $k \neq 0$ via localization around $\sigma = 1$ and 
$\mh H_0$ via localization 
around $\sigma = \chi_-$. The construction of their irreducible representations 
is an instance of how that can be done for affine/graded Hecke algebras associated to unipotent 
representations of $p$-adic groups \cite{Lus-Uni,Lus-Uni2,AMS2,SolLLCunip}. Let us summarise that:
\begin{multline}\label{eq:1.8}
\text{for } \mc H (\mc O,G) \text{ with rk}(\Sigma_{\mc O,\mu}) = 1, \\
\text{Conjecture \ref{conj:1} holds on the level of the underlying graded Hecke algebras.}
\end{multline}
While this does not settle Conjecture \ref{conj:1} for all affine Hecke algebras in the rank 
one cases, it looks like a satisfactory outcome.

\section{Reduction to simply connected groups}
\label{sec:red}

In this section we reduce the analysis of the parameters of $\mc H (\mc O,G)$ to the case where 
$\mc G$ is absolutely simple and simply connected. Consider a homomorphism between connected 
reductive $F$-groups $\eta : \tilde{\mc G} \to \mc G$ such that:
\begin{itemize}
\item the kernel of d$\eta : \mr{Lie}(\tilde{\mc G}) \to \mr{Lie}(\mc G)$ is central,
\item the cokernel of $\eta$ is a commutative $F$-group.
\end{itemize}
These properties imply \cite[Lemma 5.1]{SolFunct} that on the derived groups $\eta$ restricts to 
\begin{equation}\label{eq:2.1}
\text{a central isogeny } \eta_\der : \tilde{\mc G}_\der \to \mc G_\der
\end{equation}
Such a map induces a homomorphism on $F$-rational points
\[
\eta : \tilde G = \tilde{\mc G}(F) \to \mc G (F) = G
\]
and a pullback functor $\eta^* : \Rep (G) \to \Rep (\tilde G)$. 

\begin{lem}\label{lem:2.1}
Let $\pi \in \Irr (G)$. Then $\eta^* (\pi)$ is a finite direct sum of irreducible $\tilde G$-representations.
\end{lem} 
\begin{proof}
According to \cite[Lemma 2.1]{Tad} this holds for the inclusion of $\mc G_\der$ in $\mc G$.
Taking that into account, \cite{Sil1} says that pullback along $\eta_\der : \tilde{G}_\der \to G_\der$
has the desired property. This shows that $\Res^{\tilde G}_{\tilde{G}_\der} \eta^* (\pi)$ is a finite
direct sum of irreducible $\tilde{G}_\der$-representations. As in the proof of \cite[Lemma 2.1]{Tad},
that implies the same property for $\eta^* (\pi)$.
\end{proof}

By \eqref{eq:2.1}, $\eta$ induces a bijection
\[
\begin{array}{ccc}
\{ \text{Levi subgroups of } G \} & \to & \{ \text{Levi subgroups of } \tilde G \} \\
M & \mapsto & \tilde M = \eta^{-1} (M) 
\end{array}.
\]
One also sees from \eqref{eq:2.1} that $\eta$ induces a bijection
\[
\begin{array}{ccc}
\Sigma (G,A_M) & \to & \Sigma (\tilde G,A_{\tilde M}) \\
 \alpha & \mapsto & \tilde \alpha = \alpha \circ \eta 
\end{array}.
\]
For each $\alpha \in \Sigma_\red (A_M)$ this yields an isomorphism of $F$-groups
\[
\eta_\alpha : \mc U_{\tilde \alpha} \to \mc U_\alpha .
\]
This implies that $\eta^*$ preserves cuspidality \cite[Lemma 1]{Sil1}. Further, pullback along 
$\eta$ restricts to an algebraic group homomorphism $\eta^*: X_\nr (M) \to X_\nr (\tilde M)$.

\begin{prop}\label{prop:2.2}
Let $(\sigma,E) \in \Irr_\cusp (M)$ and let $\tilde \sigma \in \Irr_\cusp (\tilde M )$ be a
constituent of $\eta^* (\sigma)$. For $\alpha \in \Sigma_\red (A_M)$
there exists $\tilde c_\alpha \in \C^\times$ such that
\[
\mu^{M_\alpha} (\sigma \otimes \chi) = 
\tilde c_\alpha \mu^{\tilde{M}_{\tilde \alpha}} (\tilde \sigma \otimes \eta^* (\chi))
\]
as rational functions of $\chi \in X_\nr (M)$.
\end{prop}
\begin{proof}
In view of the explicit shape \eqref{eq:3.22}, it suffices to show that the two rational functions
have precisely the same poles. Using the relation \eqref{eq:1.3}, it suffices to show that
\begin{align}\label{eq:2.2}
J_{M_\alpha \cap P | s_\alpha (M_\alpha \cap P)}(\sigma \otimes \chi) J_{s_\alpha (M_\alpha \cap P) | 
M_\alpha \cap P}(\sigma \otimes \chi) = 0 \quad \Longleftrightarrow \\
\nonumber J_{\eta^{-1}(M_\alpha \cap P) | \eta^{-1}(s_\alpha (M_\alpha \cap P))}(\tilde \sigma 
\otimes \eta^* (\chi)) J_{\eta^{-1}(s_\alpha (M_\alpha \cap P)) | \eta^{-1}(M_\alpha \cap P)}
(\tilde \sigma \otimes \eta^* (\chi)) = 0 . 
\end{align}
Since $\eta_\alpha : U_{\tilde \alpha} \to U_\alpha$ is an isomorphism, we may choose Haar measures
on $U_\alpha$ and $U_{\tilde \alpha}$ such that the latter is pullback along $\eta_\alpha$ of the
former. Then \eqref{eq:1.4} shows that
the $J$-operators on both lines of \eqref{eq:2.2} do the same thing, namely
\[
f \mapsto \big[ g \mapsto \int_U  f(ug) \textup{d}u \big] ,
\] 
where $U$ stands for $U_\alpha$ or $U_{-\alpha}$. The only real difference between the two lines of
\eqref{eq:2.2} lies in their domain. Since $\tilde \sigma \otimes \eta^* (\chi)$ is a subrepresentation
of $\eta^* (\sigma \otimes \chi)$, it is clear that the implication $\Rightarrow$ holds.

Conversely, suppose that the second line of \eqref{eq:2.2} is 0, for a particular $\chi$. Let $\tilde E
\subset E$ be the subspace on which $\tilde \sigma$ is defined, so that $I_{\eta^{-1} P}^{\tilde G}
(\tilde E) \cong I_P^G (\tilde E)$ is the vector space underlying 
$I_{\eta^{-1} P}^{\tilde G} (\tilde \sigma \otimes \eta^* (\chi))$. It is a linear subspace of 
$I_P^G (E)$, on which $I_P^G (\sigma \otimes \chi)$ is defined. Then 
\begin{equation}\label{eq:2.4}
J_{M_\alpha \cap P | s_\alpha (M_\alpha \cap P)}(\sigma \otimes \chi) J_{s_\alpha (M_\alpha \cap P) | 
M_\alpha \cap P}(\sigma \otimes \chi)
\end{equation}
coincides on $I_P^G (\tilde E)$ with
\[
J_{\eta^{-1}(M_\alpha \cap P) | \eta^{-1}(s_\alpha (M_\alpha \cap P))}(\tilde \sigma 
\otimes \eta^* (\chi)) J_{\eta^{-1}(s_\alpha (M_\alpha \cap P)) | \eta^{-1}(M_\alpha \cap P)}
(\tilde \sigma \otimes \eta^* (\chi)) ,
\]
so annihilates $I_P^G (\tilde E)$. But by \eqref{eq:1.3} the operator \eqref{eq:2.4} is a scalar on 
$I_P^G (E)$, so it annihilates that entire space.
\end{proof}

From Proposition \ref{prop:2.2} and \eqref{eq:1.5} we deduce:

\begin{cor}\label{cor:2.3}
In the setting of Proposition \ref{prop:2.2}, write $\tilde{\mc O} = X_\nr (\tilde M) \tilde \sigma$.
Then $\Sigma_{\tilde{\mc O},\mu}$ equals 
\[
\eta^* (\Sigma_{\mc O,\mu}) =  \{ \tilde \alpha = \alpha \circ \eta : \alpha \in \Sigma_{\mc O,\mu} \}.
\]
\end{cor}

We warn that Proposition \ref{prop:2.2} and Corollary \ref{cor:2.3} do not imply that $q_\alpha =
q_{\tilde \alpha}$. The problem is that $X_\alpha$ need not equal $X_{\tilde \alpha} \circ \eta^*$.
To make the relation precise, we have to consider $h_\alpha^\vee$, $h_{\tilde \alpha}^\vee$ and
their images (via $H_M$ and $H_{\tilde M}$) in $\mf a_M$ and $\mf a_{\tilde M}$.
We note that d$\eta : \mr{Lie}(A_{\tilde M}) \to \mr{Lie}(A_M)$ induces a linear map 
$\mf a_\eta : \mf a_{\tilde M} \to \mf a_M$. Further, $\eta$ induces a group homomorphism
\begin{equation}\label{eq:2.3}
\eta : (\tilde M \cap \tilde{M}_{\tilde \alpha}^1 ) / \tilde{M}^1 \to (M \cap M_\alpha^1 ) / M^1 . 
\end{equation}
Both the source and the target of \eqref{eq:2.3} are isomorphic to $\Z$, so the map is injective.

\begin{prop}\label{prop:2.4}
\enuma{
\item For $\alpha \in \Sigma_{\mc O,\mu}$, there exists a $N_\alpha \in \{1/2, 1, 2\}$
such that 
\[
H_M (h_\alpha^\vee) = N_\alpha \mf a_\eta \big( H_{\tilde M} (h_{\tilde \alpha}^\vee ) \big).
\]
\item If \eqref{eq:2.3} is bijective, then $N_\alpha \in \{1,2\}$. This happens for instance when 
$\eta$ restricts to an isomorphism between the almost direct $F$-simple factors of $\tilde{\mc G}$ and 
$\mc G$ corresponding to $\tilde \alpha$ and $\alpha$, 
\item If $\eta^* (\sigma)$ is irreducible, then $N_\alpha \in \{1/2, 1\}$.
\item Let $\Sigma_{\mc O,j}^\vee$ be an irreducible component of $\Sigma_{\mc O}^\vee$, and regard
it as a subset of $\mf a_M$ via $H_M$. Consider the irreducible component 
\[
\Sigma_{\tilde{\mc O},j}^\vee = \{ h_{\tilde \alpha}^\vee : h_\alpha^\vee \in \Sigma_{\mc O,j}^\vee \}
\]
of $\Sigma_{\tilde{\mc O}}^\vee$. There are three possibilities:
\begin{enumerate}[(i)]
\item $N_\alpha = 1$ for all $h_\alpha^\vee \in \Sigma_{\mc O,j}^\vee$. 
\item $\Sigma_{\mc O,j}^\vee \cong B_n$, $\Sigma_{\tilde{\mc O},j}^\vee \cong C_n$,
$N_\alpha = 1$ for $h_\alpha^\vee \in \Sigma_{\mc O,j}^\vee$ long and $N_\beta = 1/2$ for
$h_\beta^\vee \in \Sigma_{\mc O,j}^\vee$ short. Then 
\[
q_{\tilde \beta *} = 1 ,\; q_\beta = q_{\beta*} = q_{\tilde \beta}^{1/2} ,\; 
\lambda^* (\beta) = 0 \text{ and } 
\lambda (\beta) = \lambda (\tilde \beta ) = \lambda^* (\tilde \beta ).
\]
\item $\Sigma_{\mc O,j}^\vee \cong C_n$, $\Sigma_{\tilde{\mc O},j}^\vee \cong B_n$,
$N_\alpha = 1$ for $h_\alpha^\vee \in \Sigma_{\mc O,j}^\vee$ short and $N_\beta = 2$ for
$h_\beta^\vee \in \Sigma_{\mc O,j}^\vee$ long. Then 
\[
q_{\beta *} = 1 ,\; q_{\tilde \beta}^2 = q_{\tilde \beta*}^2 = q_\beta ,\; 
\lambda^* (\tilde \beta ) = 0 \text{ and } 
\lambda (\tilde \beta ) = \lambda (\beta ) = \lambda^* (\beta ).
\]
\end{enumerate}
\item The modifications of the labels in part (d) preserve the class of labels in Table \ref{fig:1}.
}
\end{prop}
\begin{proof}
(a) As both sides of \eqref{eq:2.3} are isomorphic to $\Z$, the definition of $h_\alpha^\vee$ implies
that the statement holds for some $N_\alpha \in \Q_{>0}$. Then $\eta (X_{\tilde \alpha}) =
X_\alpha^{1/N_\alpha}$, and this is a well-defined function on $X_\nr (M)$ because it equals
evaluation at $\eta (h_{\tilde \alpha}^\vee)$. We plug this into the equality of $\mu$-functions from
Proposition \ref{prop:2.2}, and we use the formula \eqref{eq:3.22} both for $M$ and for $\tilde M$. 
That yields an equality of two 
rational functions on $X_\nr (M)$, one built from $X_\alpha$ and one built from $X_\alpha^{1/N_\alpha}$.
The equality of the numerators of these two functions reads
\begin{multline}\label{eq:2.5}
c'_{s_\alpha} (1 - X_\alpha)(1 - X_\alpha^{-1}) \, \big[ (1 + X_\alpha) (1 + X_\alpha^{-1}) \big] = \\
\tilde {c}_\alpha c'_{s_{\tilde \alpha}} (1 - X_\alpha^{1/N_\alpha})(1 - X_\alpha^{-1/ N_\alpha}) \,
\big[(1 + X_\alpha^{1/N_\alpha}) (1 + X_\alpha^{-1/N_\alpha}) \big] .
\end{multline}
Here the term $(1 + X_\alpha) (1 + X_\alpha^{-1})$ must be omitted when $q_{\alpha *} = 1$. On the
other hand $q_\alpha > 1$ because $\alpha \in \Sigma_{\mc O,\mu}$, so the zeros at $X_\alpha = 1$
do not cancel against something in the denominator of \eqref{eq:3.22}. Analogous considerations 
apply to the second line of \eqref{eq:2.5}. Now we see that there are only three values of $N_\alpha$
for which \eqref{eq:2.5} is possible: $N_\alpha = 1$, $N_\alpha = 1/2$ (when the factor 
$(1 + X_\alpha^{1/N_\alpha}) (1 + X_\alpha^{-1/N_\alpha})$ is not there) and $N_\alpha = 2$
(when $(1 + X_\alpha) (1 + X_\alpha^{-1})$ is omitted).\\
(b) By \eqref{eq:2.1} $\eta$ induces an isomorphism between the respective adjoint groups. From
$G \to G_\ad \to \tilde{G}_\ad$ we get an action of $G$ on $\tilde G$, by ``conjugation". All the 
$\tilde M$-constituents of $\eta^* (\sigma)$ are associated (up to isomorphism) by elements of 
$M$. For $m \in M$, Ad$(m) : \tilde M \to \tilde M$ does not affect unramified characters
of $\tilde M$. It follows that any $\chi \in X_\nr (\tilde M)$ which stabilizes $\tilde \sigma$,
also stabilizes $\eta^* (\sigma)$. That implies
\[
\eta^{-1} \big( (M_\sigma^2 \cap M_\alpha^1 ) / M^1 \big) \subset 
(\tilde M_{\tilde \sigma}^2 \cap \tilde{M}_{\tilde \alpha}^1 ) / \tilde{M}^1 .
\]
That and the assumed bijectivity show that
\[
h_\alpha^\vee \in \eta (\tilde M_{\tilde \sigma}^2 \cap \tilde{M}_{\tilde \alpha}^1 ) / \tilde{M}^1 .
\]
By definition $h_{\tilde \alpha}^\vee$ generates $(\tilde M^2_{\tilde \sigma} \cap 
\tilde{M}_{\tilde \alpha}^1 ) / \tilde{M}^1$, so an integer multiple of its image under $\eta$ 
equals $h_\alpha^\vee$. By part (a) the multiplication factor is at most 2.\\
(c) If $\chi \in X_\nr (M,\sigma)$, then $\eta^* (\sigma) \otimes \eta^* (\chi) = \eta^* (\sigma
\otimes \chi)$ is isomorphic with $\eta^* (\sigma)$. Hence 
\[
\eta^* (X_\nr (M,\sigma)) \subset X_\nr (\tilde M, \eta^* (\sigma)) ,
\]
which implies that $\eta (\tilde{M}^2_{\eta^* (\sigma)}) \subset M_\sigma^2$. As $h_\alpha^\vee$
generates $(M_\sigma^2 \cap M_\alpha^1) / M^1$ and $\eta (h_{\tilde \alpha})$ lies in that group,
$\eta (h_{\tilde \alpha})$ is a multiple of $h_\alpha^\vee$. Combine that with part (a).\\
(d) In case $N_\alpha = 2$, $2 \eta (X_{\tilde \alpha}) = X_\alpha$. Then Proposition \ref{prop:2.2}
and \eqref{eq:3.22} entail $q_{\alpha} = 1$ and $q_{\tilde \alpha} = q_{\tilde \alpha*} = 
q_{\alpha}^{1/2}$. Notice that this is only possible when $\Sigma_{\tilde{\mc O},j}^\vee \cong C_n$.

When $N_\alpha = 1/2$, we have $\eta (X_{\tilde \alpha}) = 2 X_\alpha$. For the same reasons as 
above, $q_{\tilde \alpha} = 1$ and $q_\alpha = q_{\alpha*} =  q_{\tilde \alpha}^{1/2}$. By 
\cite[Lemma 3.3]{SolEnd} this is only possible if $\Sigma_{\mc O,j}^\vee$ has type $B_n$.\\
(e) Parts (d,ii) and (d,iii) just switch the second line 
(with $\lambda^* = 0$) and the third line of Table \ref{fig:1}.
\end{proof}

We remark that examples of case (ii) are easy to find, it already occurs for $SL_2 (F) \to PGL_2 (F)$
and the unramified principal series (as worked out in Paragraph \ref{par:split}).
For an instance of case (iii) see Example \ref{ex:A}.

We can apply Propositions \ref{prop:2.2} and \ref{prop:2.4} in particular with $\tilde{\mc G}$
equal to the simply connected cover $\mc G_\Sc$ of $\mc G_\der$, that yields:

\begin{cor}\label{cor:2.6}
Suppose that Conjecture \ref{conj:1} holds for $\tilde G = G_\Sc$ and 
$[\tilde M, \tilde \sigma]_{G_\Sc}$. Then it holds for $G$ and $[M,\sigma]_G$.
\end{cor}

Every simply connected $F$-group is a direct product of $F$-simple simply connected groups, say 
\[
\mc G_\Sc = \prod\nolimits_i \mc G_\Sc^{(i)} .
\]
Everything described in Section \ref{sec:progen} decomposes accordingly, for instance any
$\tilde \sigma \in \Irr_\cusp (G_\Sc )$ can be factorized as
\[
\tilde \sigma = \boxtimes_i \, \sigma^{(i)} \quad \text{with} \quad 
\sigma^{(i)} \in \Irr_\cusp (G_\Sc^{(i)}) .
\]
For every $F$-simple simply connected $F$-group $\mc G_\Sc^{(i)}$ there exists a finite separable
field extension $F' / F$ and an absolutely simple, simply connected $F'$-group $\mc G_\Sc^{'(i)}$,
such that $\mc G_\Sc^{(i)}$ is the restriction of scalars from $F'$ to $F$ of $\mc G_\Sc^{'(i)}$. Then
\[
G_\Sc^{(i)} = \mc G_\Sc^{(i)}(F) = \mc G_\Sc^{'(i)} (F') = G_\Sc^{'(i)} ,
\] 
so $\sigma^{(i)}$ can be regarded as a supercuspidal representation of $G_\Sc^{'(i)}$. Of course 
that last step does not change the parameters $q_\alpha$ and $q_{\alpha *}$ associated to $\sigma^{(i)}$.
On the other hand, that step does replace $q_F$ by $q_{F'}$ and changes the labels 
$\lambda (\alpha)$ and $\lambda^* (\alpha)$ by a factor $\log (q_F) / \log (q_{F'})$.
As this is the same scalar factor for all $\alpha \in \Sigma_{\tilde{\mc O}^{(i)},\mu}$, it is
innocent. With these steps we reduced the computation of the parameters $q_\alpha, q_{\alpha *}, 
\lambda (\alpha)$ and $\lambda^* (\alpha)$ to the case where $\mc G$ is absolutely
simple and simply connected. 
 
Sometime it is more convenient to study, instead of a simply connected simple group, a reductive
group with that as derived group. For instance, the groups $GL_n, U_n, \mr{GSpin}_n$ are often easier
than, respectively, $SL_n, SU_n, \mr{Spin}_n$. In such situations, the following result comes in handy.

\begin{prop}\label{prop:2.5} \textup{\cite[Propositions 2.2 and 2.7]{Tad}} \\
Suppose that $\tilde{\mc G}$ is a connected reductive $F$-subgroup of $\mc G$ that contains 
$\mc G_\der$. For every $\tilde \pi \in \Irr (\tilde G)$ there exists a $\pi \in \Irr (G)$ such that
$\Res^{\tilde G}_G (\pi)$ contains $\tilde \pi$. Moreover $\tilde \pi$ is supercuspidal if and only
if $\pi$ is supercuspidal.
\end{prop}

We note that in this setting the inclusion $\imath : \tilde{\mc G} \to \mc G$ satisfies the conditions
stated at the start of the paragraph. 

\begin{cor}\label{cor:2.7}
Let $\mc G, \tilde{\mc G}$ be as in Proposition \ref{prop:2.5}. Then Conjecture \ref{conj:1} holds
for $G$ if and only if it holds for $\tilde G$.
\end{cor}
\begin{proof}
Let $M$ be a Levi subgroup of $G$ and let $\tilde \pi \in \Irr_\cusp ( M \cap \tilde G)$. 
An appropriate $\pi$ is obtained from Proposition \ref{prop:2.5} applied to 
$\imath : \mc M \cap \tilde{\mc G} \to \mc M$. Then $\tilde \pi$ is a constituent of $\imath^* (\pi)$.
This also works the other way round: if we start with $\pi \in \Irr_\cusp (M)$ we can choose as
$\tilde \pi$ any constituent of $\imath^* (\pi)$.
Now we can apply Proposition \ref{prop:2.4}, which says that the Hecke algebras 
$\mc H (X_\nr (M \cap \tilde G) \tilde \pi, \tilde G)$ and $\mc H (X_\nr (M) \pi, G)$ have root 
systems and parameters related as in cases (i) or (iii) of Proposition \ref{prop:2.4}.d.
\end{proof}

\section{Reduction to characteristic zero}
\label{sec:char0}

For several classes of reductive groups, stronger results are available over $p$-adic fields than
over local function fields. With the method of close local fields \cite{Kaz,Gan2}, we will show 
that all relevant results about affine Hecke algebras associated to Bernstein components can be 
transferred from characteristic zero to positive characteristic.

We start with an arbitrary local field of characteristic $p$. Choose a $p$-adic field $\tilde F$
which is $\ell$-close to $F$, that is
\begin{equation}\label{eq:4.1}
\mf o_F / \varpi_F^\ell \mf o_F \cong \mf o_{\tilde F} / \varpi_{\tilde F}^\ell \mf o_{\tilde F}
\qquad \text{as rings.}
\end{equation}
As remarked in \cite{Del}, such a field $\tilde F$ exists for every given $\ell \in \Z_{>0}$.
If \eqref{eq:4.1} holds, then it is also valid for every $m < \ell$, and in particular the
residue fields $\mf o_F / \varpi_F \mf o_F$ and $\mf o_{\tilde F} / \varpi_{\tilde F} \mf o_{\tilde F}$
are isomorphic. We note that
\begin{equation}\label{eq:4.2}
F^\times / (1 + \varpi_F^\ell \mf o_F) \cong \Z \times \mf o_F^\times /  (1 + \varpi_F^\ell \mf o_F)
\cong \Z \times \mf o_{\tilde F}^\times /  (1 + \varpi_{\tilde F}^\ell \mf o_{\tilde F}) =
{\tilde F}^\times / (1 + \varpi_{\tilde F}^\ell \mf o_{\tilde F}) .
\end{equation}
Let $\mb I_F^\ell$ be the $\ell$-th ramification subgroup of Gal$(F_s / F)$.
By \cite[(3.5.1)]{Del} there is a group isomorphism (unique up to conjugation)
\begin{equation}\label{eq:4.3}
\mr{Gal}(F_s / F) / \mb I_F^\ell \cong \mr{Gal} (\tilde{F}_s / \tilde F) / \mb I_{\tilde F}^\ell ,
\end{equation}
and similarly with Weil groups. According to \cite[Proposition 3.6.1]{Del}, for $m < \ell$ this 
isomorphism is compatible with the Artin reciprocity map
\[
\mb W_F / \mb I_F^\ell \to F^\times / (1 + \varpi_F^m \mf o_F) .
\]
Let $\mc G$ be a connected reductive $F$-group. We want to exhibit ``the same'' group over a $p$-adic 
field. The quasi-split inner form $\mc G^*$ of $\mc G$ is determined by the action of Gal$(F_s / F)$
on the based absolute root datum of $\mc G$. That action factors through a finite quotient of
Gal$(F_s/F)$, so there exists a $\ell \in \Z_{>0}$ such that $\mb I_F^\ell$ acts trivially. The group
$\mc G$ is an inner twist of $\mc G^*$, and the inner twists of $\mc G^*$ are parametrized naturally by
\begin{equation}\label{eq:4.4}
H^1 (F, \mc G^*_\ad) \cong \Irr \big( Z (G^{*\vee}_\Sc)^{\mb W_F} \big) .
\end{equation}
Now we pick a $p$-adic field $\tilde F$ which is $\ell$-close to $F$, and we define $\tilde{\mc G}^*$
to be the quasi-split $\tilde F$-group with the same based root absolute root datum as $\mc G^*$
and Galois action transferred from that of $\mc G^*$ via \eqref{eq:4.3}. Then $\mc G^*$ and 
$\tilde{\mc G}^*$ have the same Langlands dual group (in a form where $\mb I_F^\ell$ has been divided
out)  and hence
\begin{equation}\label{eq:4.5}
Z (\tilde{G}^{*\vee}_\Sc)^{\mb W_{\tilde F}} \cong Z (G^{*\vee}_\Sc)^{\mb W_F} .
\end{equation}
We define $\tilde{\mc G}$ to be the inner twist of $\tilde{\mc G}^*$ parametrized by the character
of $Z (\tilde{G}^{*\vee}_\Sc)^{\mb W_{\tilde F}}$ that is transformed by \eqref{eq:4.5} into the
character of $Z (G^{*\vee}_\Sc)^{\mb W_F}$ that parametrizes $\mc G$.

The following descriptions are based on recent work of Ganapathy \cite{Gan1,Gan2}. It applies when 
$F$ and $\tilde F$ are $\ell$-close with $\ell$ large enough. The relation between $\mc G$ 
and $\tilde{\mc G}$ is the same as in these papers, although over there it is reached in a slightly 
different way, without \eqref{eq:4.4}. Let $\mc T \subset \mc G$ be the maximal $F$-torus from which
the root datum is built, and let $\mc S \subset \mc T$ be the maximal $F$-split subtorus. In the 
Bruhat--Tits building $\mc B (\mc G,F)$, $S = \mc S (F)$ determines an apartment $\mh A_S$. 

The same constructions can be performed for $\tilde{\mc G}$. Then $X_* (\mc S) \cong X_* (\tilde{\mc S})$
extends to an isomorphism of polysimplicial complexes $\mh A_S \cong \mh A_{\tilde S}$. For every
special vertex $x \in \mh A_S$, we get a special vertex $\tilde x \in \mh A_{\tilde S}$.
For $m \in \Z_{\geq 0}$, there is a refined version of the Moy--Prasad group $G_{x,m}$, see 
\cite{Gan1}. It is a compact open normal subgroup of $G_x$, the $G$-stabilizer of $x$. More precisely,
there is an $\mf o_F$-group scheme $\mc G_x$ (a slightly improved version of the parahoric group schemes 
constructed in \cite{BrTi2}), such that
\[
G_{x,m} = \mc G_x (\varpi_F^m \mf o_F ) \qquad \forall m \in \Z_{\geq 0}.
\]
By construction \cite[\S 2.D.3]{Gan1}, $G_{x,m}$ is totally decomposed in the sense of \cite[\S 1]{Bus}.
This means that, for any ordering of the root system $\Sigma (G,S)$, the product map
\[
(G_{x,m} \cap Z_G (S)) \times \prod\nolimits_{\alpha \in \Sigma (G,S)} (G_{x,m} \cap U_\alpha) 
\longrightarrow G_{x,m}
\]
is a bijection. Here $U_\alpha$ is the root subgroup of $G$ with respect to $\alpha \in \Sigma (G,S)$
(to be distinguished from the earlier $U_\alpha$ when $M$ is not a minimal Levi subgroup of $G$).

All the above applies to $\tilde{\mc G}$ as well. 
The following results generalize \cite{Kaz} to non-split groups.

\begin{thm}\label{thm:4.1} \textup{\cite[Corollary 6.3]{Gan1}} \\
Fix $m \in \Z_{> 0}$ and let $\ell \in \Z_{>0}$ be large enough. The isomorphisms \eqref{eq:4.1}
induce an isomorphism of group schemes
\[
\mc G_x \times_{\mf o_F} \mf o_F / \varpi_F^m \mf o_F \cong 
\mc G_{\tilde x} \times_{\mf o_{\tilde F}} \mf o_{\tilde F} / \varpi_{\tilde F}^m \mf o_{\tilde F}
\]
and group isomorphisms
\[
G_{x,0} / G_{x,m} = \mc G_x ( \mf o_F / \varpi_F^m \mf o_F ) \cong \mc G_{\tilde x} (\mf o_{\tilde F} 
/ \varpi_{\tilde F}^m \mf o_{\tilde F}) = \tilde G_{\tilde x,0} / \tilde G_{\tilde x,m} .
\]
\end{thm}

We endow $G$ with the Haar measure that gives the parahoric subgroup $G_{x,0}$ volume 1. 
The vector space $C_c (G_{x,m} \backslash G / G_{x,m} )$ with the convolution product is an 
associative algebra, denoted $\mc H (G,G_{x,m})$. 

\begin{thm}\label{thm:4.2} \textup{\cite[Theorem 4.1]{Gan2}} \\
Fix $m \in \Z_{>0}$ and let $\ell \in \Z_{>0}$ be large enough. The isomorphisms from Theorem 
\ref{thm:4.1} and the Cartan decomposition give rise to a bijection
\[
\zeta_m : G_{x,m} \backslash G / G_{x,m} \to 
\tilde G_{\tilde x,m} \backslash \tilde G / \tilde G_{\tilde x,m}.
\]
This map extends to an algebra isomorphism
\[
\zeta_m^G : \mc H (G,G_{x,m}) \to \mc H (\tilde G, \tilde G_{\tilde x,m}) .
\]
\end{thm}

In particular $\zeta_m$ induces a group isomorphism $G / G^1 \to \tilde G / \tilde{G}^1$, and hence
a group isomorphism
\begin{equation}\label{eq:4.8}
\overline{\zeta_m^G} : X_\nr (G) = \Irr (G/G^1) \to \Irr (\tilde G / \tilde{G}^1) = \Irr (\tilde G) .
\end{equation}
Let $\Rep (G,G_{x,m})$ be the category of smooth $G$-representations that are generated by their
$G_{x,m}$-fixed vectors. Recall that $G_{x,m}$ is a totally decomposed open normal subgroup of the
good maximal compact subgroup $G_x$ of $G$. From \cite[\S 3.7--3.9]{BeDe} we know that there is 
an equivalence of categories
\begin{equation}\label{eq:4.6}
\begin{array}{ccc}
\Rep (G,G_{x,m}) & \longrightarrow & \Mod (\mc H (G,G_{x,m})) \\
V & \mapsto & V^{G_{x,m}}
\end{array}.
\end{equation}
From \eqref{eq:4.6} and Theorem \ref{thm:4.2} one obtains equivalences of categories
\begin{equation}\label{eq:4.7}
\begin{array}{lclll}
(\zeta_m^G )_* & : & \Mod (\mc H (G,G_{x,m})) & \longrightarrow & 
\Mod (\mc H (\tilde G,\tilde G_{\tilde x,m})) \\
\overline{\zeta_m^G} & : & \Rep (G,G_{x,m}) & \longrightarrow & \Rep (\tilde G, \tilde G_{\tilde x,m})
\end{array},
\end{equation}
which constitute the core of the method of close local fields.
We will need many properties of these equivalences, starting with two easy ones about characters.

\begin{lem}\label{lem:4.6}
\enuma{
\item Via \eqref{eq:4.8}, the equivalence of categories $\overline{\zeta_m^G}$ 
preserves twists by unramified characters.
\item $\zeta_m$ induces an isomorphism
\[
A_G / (A_G \cap G_{x,m}) \longrightarrow A_{\tilde G} / (A_{\tilde G} \cap \tilde G_{\tilde x,m}) .
\]
The effect of $\overline{\zeta_m^G}$ on $A_G$-characters of $G$-representations is push-forward
along this isomorphism.
}
\end{lem}
\begin{proof}
(a) This is clear from \eqref{eq:4.8}.\\
(b) Notice that 
\begin{equation}\label{eq:4.9}
A_G / A_G \cap G_{x,m} \cong X_* (\mc A_G) \otimes_\Z ( F^\times / 1 + \varpi_F^m \mf o_F ) ,
\end{equation}
and similarly for $\tilde G$. Since $\zeta_m$ comes from the isomorphism $X_* (S) \cong X_* (\tilde S)$,
it induces a linear bijection $X_* (A_G) \to X_* (A_{\tilde G})$, and hence an isomorphism from
\eqref{eq:4.9} to its counterpart for $\tilde G$. The $A_G$-characters of representations in
$\Rep (G,G_{x,m})$ are precisely the characters of \eqref{eq:4.9}, and $\overline{\zeta_m^G}$
pushes them forward along $\zeta_m$.
\end{proof}

Let $P$ be a parabolic subgroup of $G$ with a Levi factor $M$, which contains $S$.
By \cite[\S 1.6]{Bus} the normalized parabolic functor $I_P^G$ sends $\Rep (M,M_{x,m})$ to 
$\Rep (G,G_{x,m})$. We will exploit an expression for this functor \cite{Bus} in terms that can be 
transferred to $\tilde G$ with Theorems \ref{thm:4.1} and \ref{thm:4.2}.

Let $P^{op}$ be the parabolic subgroup of $G$ that is opposite to $P$ with respect to $M$.
Let $M_{x,m} = G_{x,m} \cap M$ be the version of $G_{x,m}$ for $M$.
Recall that an element $g \in M$ is called $(P,G_{x,m})$-positive if 
\[
g (G_{x,m} \cap P) g^{-1} \subset G_{x,m} \cap P \quad \text{and} \quad
g (G_{x,m} \cap P^{op}) g^{-1} \supset G_{x,m} \cap P^{op}.
\]
Let $\mc H^+ (M,M_{x,m})$ be the subalgebra of $\mc H (M,M_{x,m})$ consisting of functions that
are supported on $(P,G_{x,m})$-positive elements. In \cite[\S 3.3]{Bus}, which is based on \cite{BuKu2}, 
a canonical injective algebra homomorphism
\[
j_P : \mc H^+ (M,M_{x,m}) \to \mc H (G,G_{x,m})
\]
is given. Let $\tilde P$ and $\tilde M$ be the subgroups of $\tilde G$ corresponding to $P$ and $M$ 
via the equality of based root data. All the above constructions also work in $\tilde G$, and we 
endow the resulting objects with tildes.

\begin{lem}\label{lem:4.7}
\enuma{
\item $\zeta_m^M$ restricts to an algebra isomorphism from
$\mc H^+ (M,M_{x,m})$ to $\mc H^+ (\tilde M, \tilde M_{\tilde x,m})$.
\item The following diagram commutes:
\[
\begin{array}{ccc}
\mc H (G,G_{x,m}) & \xrightarrow{\zeta_m^G} & \mc H (\tilde G,\tilde G_{\tilde x,m}) \\
\uparrow j_P & & \uparrow j_{\tilde P} \\
\mc H^+ (M,M_{x,m}) & \xrightarrow{\zeta_m^M} & \mc H^+ (\tilde M, \tilde M_{\tilde x,m})
\end{array} .
\]
}
\end{lem}
\begin{proof}
(a) The property "$(P,G_{x,m})$-positive" can be expressed in terms of the Cartan decomposition of
$M$. Namely, the elements of a double coset $M_{x,0} g M_{x,0}$ with $g \in Z_M (S)$ are
$(P,G_{x,m})$-positive if and only if 
\begin{equation}\label{eq:4.11}
|\alpha (g)|_F \leq 1 \text{ for all } \alpha \in \Sigma (G,S) \text{ that appear in Lie}(P) .
\end{equation}
(Notice that $|\alpha |_F$ extends naturally to a character of $Z_G (S)$ because $S$ is cocompact 
in $Z_G (S)$.) The map $\zeta_m$ from Theorem \ref{thm:4.2} for $M$ preserve the property 
\eqref{eq:4.11}, because it comes from the isomorphism $X_* (S) \cong X_* (\tilde S)$, which
preserves positivity of roots. Thus $\zeta_m$ maps $(P,G_{x,m})$-positive elements to
$(\tilde P,\tilde G_{\tilde x,m})$-positive elements, and then Theorem \ref{thm:4.2} provides
the desired isomorphism.\\
(b) We endow $M$ (resp. $\tilde M$) with the Haar measure that gives $M_{x,0}$ (resp. 
$\tilde M_{\tilde x,0}$) volume 1. Suppose that $f \in \mc H^+ (M,M_{x,m})$ has support $M_{x,m} g 
M_{x,m}$ with $g \in M$. The map $j_P$ is characterized by: $j_P f$ has support $G_{x,m} g G_{x,m}$ and
\begin{equation}\label{eq:4.12}
j_P f (g) = f(g) \delta_P (g) \mu_M (M_{x,m}) \mu_G (G_{x,m})^{-1} .
\end{equation} 
By Theorem \ref{thm:4.1}
\[
\mu_G (G_{x,m}) = [G_{x,0} : G_{x,m}]^{-1} = [\tilde G_{\tilde x,0} : \tilde G_{\tilde x,m}]^{-1} =
\mu_{\tilde G}(\tilde G_{\tilde x,m}) ,
\]
and similarly for $\mu_M (M_{x,m})$. It is well-known that $\delta_P (g)$ is the product, over all 
$\alpha \in \Sigma (G,S)$ that appear in Lie$(P)$, of the factors $|\alpha (g)|_F^{\dim U_\alpha / 
U_{2 \alpha}}$. The root subgroup $U_\alpha$ contains the root subgroup $U_{2\alpha}$ if $2 \alpha$ 
is also a root, and otherwise $U_{2\alpha} = \{1\}$ by definition. See 
\cite[Lemme V.5.4]{Ren} for a proof (although there a different convention is used, which results in 
replacing $g$ by $g^{-1}$). By Theorem \ref{thm:4.1} $\dim U_\alpha$ equals $\dim U_{\tilde \alpha}$,
where $\tilde \alpha \in \Sigma (\tilde G,\tilde S)$ corresponds to $\alpha$. Furthermore 
$\delta_P$ is trivial on compact subgroups, so $\delta_P (g)$ depends only on $M_{x,m} g M_{x,m}$. 
It follows that 
\begin{equation}\label{eq:4.13}
\delta_P (M_{x,m} g M_{x,m}) = \delta_{\tilde P} (\zeta_m (M_{x,m} g M_{x,m})) .
\end{equation}
Knowing that, we take another look at \eqref{eq:4.12} and we see that $\zeta_m^G \circ j_P =
j_{\tilde P} \circ \zeta_m^M$.
\end{proof}

Let $\mc I_{P,m} : \Mod (\mc H (M,M_{x,m})) \to \Mod (\mc H (G,G_{x,m}))$ be the composition of\\
$\mr{Res}^{\mc H (M,M_{x,m})}_{\mc H^+ (M,M_{x,m})}$ and
\[
\begin{array}{ccc}
\Mod (\mc H^+ (M,M_{x,m})) & \to & \Mod (\mc H (G,G_{x,m})) \\
V & \mapsto & \Hom_{\mc H^+ (M,M_{x,m}))} (\mc H (G,G_{x,m}), V)
\end{array},
\]
where $\mc H (G,G_{x,m})$ is regarded as a left $\mc H^+ (M,M_{x,m})$-module via $j_P$.

\begin{thm}\label{thm:4.3}
\enuma{
\item The equivalences of categories \eqref{eq:4.7} are compatible with normalized parabolic 
induction, in the sense that the following diagram commutes:
\[
\begin{array}{ccc}
\Rep (G,G_{x,m}) & \xrightarrow{\overline{\zeta_m^G}} & \Rep (\tilde G, \tilde G_{\tilde x,m}) \\
\uparrow I_P^G & & \uparrow I_{\tilde P}^{\tilde G} \\
\Rep (M,M_{x,m}) & \xrightarrow{\overline{\zeta_m^M}} & \Rep (\tilde M, \tilde M_{\tilde x,m}) 
\end{array}.
\]
\item The equivalences of categories \eqref{eq:4.7} are compatible with normalized Jacquet
restriction, in the sense that the following diagram commutes:
\[
\begin{array}{ccc}
\Rep (G,G_{x,m}) & \xrightarrow{\overline{\zeta_m^G}} & \Rep (\tilde G, \tilde G_{\tilde x,m}) \\
\downarrow J_P^G & & \downarrow J_{\tilde P}^{\tilde G} \\
\Rep (M,M_{x,m}) & \xrightarrow{\overline{\zeta_m^M}} & \Rep (\tilde M, \tilde M_{\tilde x,m}) 
\end{array}.
\]
\item $\overline{\zeta_m^G}$ and its inverse send supercuspidal representations to supercuspidal
representation. The same holds for unitary supercuspidal representations.
\item $\overline{\zeta_m^G}$ and its inverse preserve temperedness and essential square-integrability.
}
\end{thm}
\begin{proof}
(a) Lemma \ref{lem:4.7} ensures that the diagram
\[
\begin{array}{ccc}
\Mod (\mc H (G, G_{x,m})) & \xrightarrow{(\zeta_m^G )_*} & 
\Mod (\mc H (\tilde G, \tilde G_{\tilde x,m})) \\
\uparrow \mc I_{P,m} & & \uparrow \mc I_{\tilde P,m} \\
\Mod (\mc H (M, M_{x,m})) & \xrightarrow{(\zeta_m^M )_*} & 
\Mod (\mc H (\tilde M, \tilde M_{\tilde x,m})) 
\end{array}
\]
commutes. According to \cite[\S 4.1]{Bus}, the unnormalized parabolic induction functor $\mr{Ind}_P^G$ 
fits in a commutative diagram
\[
\begin{array}{ccc}
\Rep (G,G_{x,m}) & \to & \Mod (\mc H (G, G_{x,m})) \\
\uparrow \mr{Ind}_P^G & & \uparrow \mc I_{P,m} \\
\Rep (M,M_{x,m}) & \to & \Mod (\mc H (M, M_{x,m}))  
\end{array},
\]
where the horizontal arrows are the equivalences of categories from \eqref{eq:4.6}. Of course the
same holds for $\tilde G$. These two commutative diagrams entail that 
\[
\overline{\zeta_m^G} \circ \mr{Ind}_P^G = 
\mr{Ind}_{\tilde P}^{\tilde G} \circ \overline{\zeta_m^M} .
\]
In view of \eqref{eq:4.13}, if we twist this equality on the left hand side by $\delta_P^{1/2}$
and on the right hand side by $\delta_{\tilde P}^{1/2}$, it remains valid. That yields exactly
the desired relation with normalized parabolic induction.\\
(b) By Frobenius reciprocity $J_P^G$ is left adjoint to $I_P^G$, so by part (a) 
$\overline{\zeta_m^M} \circ J_P^G \circ \overline{\zeta_m^G}^{-1}$
is left adjoint to $I_{\tilde P}^{\tilde G}$. Now we use the uniqueness of adjoints.\\
(c) The first claim follows from part (a), or alternatively from part (b). For the second claim, we 
note that a supercuspidal $G$-representation is unitary if and only if its central character is unitary. 
As $A_G$ is cocompact in $Z(G)$, that is equivalent to: the $A_G$-character is unitary. 
By Lemma \ref{lem:4.6}.b, $\overline{\zeta_m^G}$ preserve the latter property. \\
(d) For the property "square integrable modulo centre" one can follow the proof of 
\cite[Th\'eor\`eme 2.17.b]{Badu}, reformulated in the setting of \cite{Gan2}. Combining that with
Lemma \ref{lem:4.6}.a, we find that $\overline{\zeta_m^G}$ also preserves essential square-integrability.

By \cite[Proposition III.4.1]{Wal}, every irreducible tempered representation\\ 
$\tau \in \Rep (G,G_{x,m})$ is a direct summand of a completely reducible representation of the form 
$I_P^G (\pi)$, where $\pi \in \Rep (M,M_{x,m})$ is square-integrable modulo centre. 
By the above and part (a),
\begin{equation}\label{eq:4.14}
\overline{\zeta_m^G} \big( I_P^G (\pi) \big) \cong 
I_{\tilde P}^{\tilde G} \big( \overline{\zeta_m^M} (\pi) \big)
\end{equation}
is also a direct sum of irreducible tempered representations. As $\overline{\zeta_m^G} (\tau)$ is 
a direct summand of \eqref{eq:4.14}, it is tempered.
\end{proof}

Consider an inertial equivalence class $\mf s = [M,\sigma ]_G$, where $S \subset M$. Choose 
$m \in \Z_{>0}$ such that $\Rep (G)^{\mf s} \subset \Rep (G,G_{x,m})$, and similarly for all
Levi subgroups of $G$ containing $M$. This is easy for supercuspidal Bernstein components
and possible in general because parabolic induction preserves
depths \cite[Theorem 5.2]{MoPr}. Fix $\ell \in \Z_{>m}$ so that Theorems \ref{thm:4.1}, \ref{thm:4.2}
and \ref{thm:4.3} apply. We may and will assume that $\sigma$ fulfills Condition \ref{cond:3.1}.
By Theorem \ref{thm:4.3}.d the $\tilde M$-representation $\tilde \sigma = \overline{\zeta_m^M}(\sigma)$
is unitary and supercuspidal. We write $\tilde{\mf s}, \tilde{\mc O}$ etc. for objects constructed
from $\tilde \sigma$.

\begin{prop}\label{prop:4.4}
\enuma{
\item The bijection $\Sigma (G,A_M) \to \Sigma (\tilde{G},A_{\tilde M})$, induced by the
equality of the root data of $\mc G$ and $\tilde{\mc G}$, sends $\Sigma_{\mc O,\mu}$ onto 
$\Sigma_{\tilde{\mc O},\mu}$.
\item Let $\alpha \in \Sigma_{\mc O,\mu}$ with image $\tilde \alpha \in \Sigma_{\tilde{\mc O},\mu}$.
The pullback of $X_{\tilde \alpha}$ along \eqref{eq:4.8} is $X_\alpha$ and $q_\alpha = q_{\tilde \alpha},
q_{\alpha*} = q_{\tilde \alpha *}$.
}
\end{prop}
\begin{proof}
Let $\alpha \in \Sigma_\red (A_M)$, with image $\tilde \alpha \in \Sigma_\red (A_{\tilde M})$.
The groups $\mc M_\alpha$ and $\tilde{\mc M}_{\tilde \alpha}$ correspond via the equality of root data
of $\mc G$ and $\tilde{\mc G}$. For $\chi \in X_\nr (M_\alpha)$, Theorem \ref{thm:4.3}.a implies that
\[
\overline{\zeta_m^M} (\sigma \otimes \chi) = \tilde \sigma \otimes \overline{\zeta_m^M}(\chi) .
\] 
By \eqref{eq:4.7}, $I_{P \cap M_\alpha}^{M_\alpha} (\sigma \otimes \chi)$ is reducible if and only if 
$I_{\tilde P \cap \tilde M_{\tilde \alpha}}^{\tilde M_{\tilde \alpha}} \big( \tilde \sigma \otimes 
\overline{\zeta_m^M}(\chi) \big)$ is reducible. If $\alpha \notin \Sigma_{\mc O,\mu}$, then 
$I_{P \cap M_\alpha}^{M_\alpha} (\sigma \otimes \chi)$ is irreducible for all non-unitary $\chi \in X_\nr 
(M_\alpha)$. It follows that $I_{\tilde P \cap \tilde M_{\tilde \alpha}}^{\tilde M_{\tilde \alpha}} 
\big( \tilde \sigma \otimes \tilde \chi \big)$ is irreducible for all non-unitary $\tilde \chi \in 
X_\nr (\tilde M_{\tilde \alpha})$, and by Corollary \ref{cor:1.1} 
$\tilde \alpha \notin \Sigma_{\tilde{\mc O},\mu}$.

On the other hand, suppose that $\alpha \in \Sigma_{\mc O,\mu}$. Then 
$I_{P \cap M_\alpha}^{M_\alpha} (\sigma \otimes \chi)$ is reducible for a $\chi \in X_\nr (M_\alpha)$ 
with $X_\alpha (\chi) = q_\alpha > 1$.
It is clear from the construction of $X_\alpha$ in \eqref{eq:3.20} that $X_{\tilde \alpha} \circ
\overline{\zeta_m^{M_\alpha}}$ is a multiple of $X_\alpha$. Consequently $I_{\tilde P \cap \tilde M_{
\tilde \alpha}}^{\tilde M_{\tilde \alpha}} \big( \tilde \sigma \otimes \overline{\zeta_m^M}(\chi) \big)$
is reducible and $X_{\tilde \alpha}( \overline{\zeta_m^M}(\chi) ) \in \R_{>0} \setminus \{1\}$. With
Corollary \ref{cor:1.1} we conclude that $\tilde \alpha \in \Sigma_{\tilde{\mc O},\mu}$.\\
(b) By \eqref{eq:4.7} and Theorem \ref{thm:4.2}, the bijection $\zeta_m$ induces a bijection
\[
(M_\sigma^2 \cap M_\alpha^1) / M^1 \longrightarrow (\tilde M_{\tilde \sigma}^2 \cap 
\tilde M_{\tilde \alpha}^1 ) / \tilde{M}^1 \cong \Z .
\]
The element $h_\alpha^\vee$ generates $(M_\sigma^2 \cap M_\alpha^1) / M^1$, while $h_{\tilde \alpha}^\vee$
generates $ (\tilde M_{\tilde \sigma}^2 \cap \tilde M_{\tilde \alpha}^1 ) / \tilde{M}^1$. These
generators are determined by conditions $\nu_F (\alpha (h_\alpha^\vee)) > 0$ and $\nu_{\tilde F} (
\tilde \alpha (h_{\tilde \alpha}^\vee))> 0$, respectively. As 
$\nu_{\tilde F} \circ \tilde \alpha \circ \zeta_m = \nu_F \circ \alpha$, we can conclude that
\begin{equation}\label{eq:4.10}
\zeta_m (h_\alpha^\vee) = h_{\tilde \alpha}^\vee \quad \text{and} \quad 
X_\alpha = X_{\tilde \alpha} \circ \overline{\zeta_m^{M_\alpha}} .
\end{equation}
Then $X_{\tilde \alpha}( \overline{\zeta_m^{M_\alpha}} (\chi) ) = q_\alpha$ and
$I_{\tilde P \cap \tilde M_{\tilde \alpha}}^{\tilde M_{\tilde \alpha}} \big( \tilde \sigma \otimes 
\overline{\zeta_m^M}(\chi) \big)$ is reducible, so $q_\alpha = q_{\tilde \alpha}$.

If $q_{\alpha *} > 1$, then $I_{P \cap M_\alpha}^M (\sigma \otimes \chi')$ is reducible for a $\chi'
\in X_\nr (M_\alpha)$ with $X_\alpha (\chi') = -q_{\alpha *}$. In that case $I_{\tilde P \cap 
\tilde M_{\tilde \alpha}}^{\tilde M_{\tilde \alpha}} \big( \tilde \sigma \otimes 
\overline{\zeta_m^M}(\chi') \big)$ is also reducible and $X_{\tilde \alpha}( \overline{\zeta_m^{M_\alpha}} 
(\chi') ) = -q_{\alpha*}$, so by Corollary \ref{cor:1.1} $q_{\tilde \alpha*} = q_{\alpha *}$.
When $q_{\alpha*} = 1$, $I_{P \cap M_\alpha}^M (\sigma \otimes \chi')$ is irreducible for all
$\chi' \in X_\nr (M_\alpha)$ with $X_\alpha (\chi') \in \R_{<-1}$. That translates to 
$\tilde M_{\tilde \alpha}$, and then Corollary \ref{cor:1.1} implies that $q_{\tilde \alpha *} = 1$.
\end{proof}

We summarise the conclusions of this sections:

\begin{cor}\label{cor:4.5}
Let $\Rep (G)^{\mf s}$ be an arbitrary Bernstein block for a connected reductive group $\mc G$
over a local function field $F$. There exist:
\begin{itemize}
\item a $p$-adic field $\tilde F$, sufficiently close to $F$,
\item a connected reductive $\tilde F$-group $\tilde{\mc G}$ with the same based root datum as $\mc G$,
\item a Bernstein block $\Rep (\tilde G)^{\tilde{\mf s}}$ for $\tilde G$, 
\end{itemize}
such that:
\begin{itemize}
\item $\Rep (G)^{\mf s}$ is equivalent with $\Rep (\tilde G )^{\tilde{\mf s}}$,
\item $\mc H (\mc O,G)$ is isomorphic with $\mc H (\tilde{\mc O}, \tilde G)$,
\item whenever $\alpha \in \Sigma_{\mc O,\mu}$ and $\tilde \alpha \in \Sigma_{\tilde{\mc O},\mu}$
correspond (via Proposition \ref{prop:4.4}), $\lambda (\alpha) = \lambda (\tilde \alpha )$
and $\lambda^* (\alpha ) = \lambda^* (\tilde \alpha )$.
\end{itemize}
\end{cor}
\begin{proof}
It only remains to establish the isomorphism of affine Hecke algebras. From \eqref{eq:4.7} and
Theorem \ref{thm:4.2} we get the bijection $M_\sigma^2 / M^1 \to \tilde M_{\tilde \sigma}^2 / \tilde{M}^1$.
From \eqref{eq:4.10} we obtain the bijection $\Sigma_{\mc O}^\vee \to \Sigma_{\tilde{\mc O}}^\vee$.
Dualizing these two bijections, we obtain an isomorphism from the root datum underlying $\mc H (\mc O,G)$
to the root datum underlying $\mc H (\tilde{\mc O},\tilde G)$. It respects the bases because $\mc G$
and $\tilde{\mc G}$ have the same based root datum. By Proposition \ref{prop:4.4}.b the parameters
$q_\alpha, q_{\alpha*}$ are the same on both sides. As
\[
q_F = [\mf o_F : \varpi_F \mf o_F] = [\mf o_{\tilde F} : \varpi_{\tilde F} \mf o_{\tilde F}] = q_{\tilde F},
\]
also the label functions $\lambda, \lambda^*$ on both sides correspond via $\alpha \mapsto \tilde \alpha$.
\end{proof}

\section{Hecke algebra parameters for simple groups}

\subsection{Principal series of split groups} \
\label{par:split}

The affine Hecke algebras for Bernstein blocks in the principal series of split groups were 
worked out in \cite{Roc1}, under some mild assumptions on the residual characteristic
of $F$. In particular, for roots $\alpha \in \Sigma_{\mc O,\mu}$ one finds $q_\alpha = q_F$ and
$q_{\alpha*} = 1$. We will derive the same conclusion in a different way, which avoids any 
restrictions on the residual characteristic. Using a little input from \cite{BeDe} we will 
evaluate the intertwining operators
\eqref{eq:1.2} directly, which is instructive but unfortunately seems infeasible outside the 
principal series. While the results in this paragraph are not original and the kind of 
calculation is also not new, we have been unable to locate such computations in the literature in 
the generality that is required for \cite{SolLLCQ}. The closest we found is \cite[\S 3]{Cas}, 
which however applies only when the underlying characters of tori are unramified.

Let $\mc G$ be a split connected reductive $F$-group. We may assume that $\mc G$ is a Chevalley
group, so defined over $\Z$. Let $\mc T$ be a maximal $F$-split torus of $\mc G$ and write
$T = \mc T (F)$. We consider an inertial equivalence class $\mf s = [T,\sigma ]_G$, where
$\sigma$ is a character of $T$ that fulfills Condition \ref{cond:3.1}. 

For $\alpha \in \Sigma (\mc G,\mc T)$ the group $\mc M_\alpha$ is generated by
$\mc T$ and the root subgroups $\mc U_\alpha, \mc U_{-\alpha}$. It has root system
$\Sigma (\mc M_\alpha,\mc T) = \{\alpha,-\alpha\}$ and parabolic subgroups $\mc P_\alpha =
\langle \mc T, \mc U_\alpha \rangle$, $\mc P_{-\alpha} = \langle \mc T, \mc U_{-\alpha} \rangle$.
Let $u_\alpha : F \to U_\alpha$ and $u_{-\alpha} : F \to U_{-\alpha}$ be the coordinates coming
from the Chevalley model. 

We assume that $s_\alpha \cdot \sigma = \sigma$, a condition which by \eqref{eq:3.4} is 
necessary for $\sigma \in \Sigma_{\mc O,\mu}$. Then $\sigma \circ \alpha^\vee = (\sigma \circ 
\alpha^\vee)^{-1}$, so $\sigma \circ \alpha^\vee$ has order $\leq 2$ in $\Irr (F^\times)$. When 
the residual characteristic of $F$ is not 2, this implies that $\sigma \circ \alpha^\vee$ has depth 
zero. Of course the cases with $\sigma \circ \alpha^\vee$ of positive depth are more involved.

We start the search for $q_\alpha$ with elements of $I_{P_\alpha}^{M_\alpha}(\sigma \otimes \chi)$ 
that are as close as possible to fixed by the Iwahori subgroup
\[
I = u_\alpha (\mf o_F) \mc T (\mf o_F) u_{-\alpha}(\varpi_F \mf o_F).
\]
For $x \in F^\times$ we write 
\[
s_\alpha (x) = u_\alpha (-x^{-1}) u_{-\alpha}(x) u_\alpha (-x^{-1}) \in N_{M_\alpha}(T) . 
\]
It follows quickly from the Iwasawa decomposition of $M_\alpha$ that 
\[
M_\alpha = P_\alpha I \sqcup P_\alpha s_\alpha I , \quad \text{where } s_\alpha = s_\alpha (1).
\]
Consider the elements $f_1, f_s \in I_{P_\alpha}^{M_\alpha}(\sigma \otimes \chi)$ defined by
\[
\begin{array}{lllll}
\text{supp}(f_1 ) = P_\alpha I & f_1 (u_\alpha (x) t u_{-\alpha}(y)) & = &
(\sigma \chi \delta_{P_\alpha}^{1/2})(t) & x \in F, t \in T, y \in \varpi_F \mf o_F , \\
\text{supp}(f_s ) = P_\alpha s_\alpha I & f_s (u_\alpha (x) t u_{-\alpha}(y) s_\alpha) & = &
(\sigma \chi \delta_{P_\alpha}^{1/2})(t) & x \in F, t \in T, y \in \mf o_F .
\end{array}
\]
We endow $F$ with the Haar measure that gives $\mf o_F$ volume 1. We compute 
\begin{equation}\label{eq:3.33}
J_{P_{-\alpha}|P_\alpha}(\sigma \otimes \chi) f_1 (1_G) = 
\int_F f_1 (u_{-\alpha}(x)) \textup{d}x = \text{vol}(\varpi_F \mf o_F) = q_F^{-1} , \vspace{-2mm}
\end{equation}
\begin{equation}\label{eq:3.13}
\begin{aligned}
& J_{P_{-\alpha}|P_\alpha} (\sigma \otimes \chi) f_1 (s_\alpha) \; = \; 
\int_F f_s \big( u_{-\alpha}(x) s_\alpha \big) \textup{d} x \; = \\
& \int_{F^\times} f_s \big( u_\alpha (-x^{-1}) 
u_{-\alpha}(x) u_\alpha (-x^{-1}) u_\alpha (x^{-1}) s_\alpha \big) \textup{d} x \; = \\
& \int_{F^\times} f_s \big( s_\alpha (x) u_{\alpha}(x^{-1}) s_\alpha \big) \textup{d} x 
\; = \; \int_{F^\times} f_s \big( s_\alpha (x) s_\alpha u_{-\alpha}(-x^{-1}) \big) \textup{d} x \; = \\
& \int_{F^\times} f_s \big( \alpha^\vee (-x^{-1}) u_{-\alpha}(-x^{-1}) \big) \textup{d} x \; = \\
& \int_{F^\times} (\sigma \chi \delta^{1/2}_{P_\alpha}) \circ \alpha^\vee (-x^{-1}) 
f_s \big( u_{-\alpha}(-x^{-1}) \big) \textup{d} x \; = \\
& \sum_{n=1}^\infty \int_{\varpi_F^{-n} \mf o_F^\times} (\sigma \chi \delta^{1/2}_{P_\alpha}) 
\circ \alpha^\vee (-x^{-1}) \textup{d}x.
\end{aligned}
\end{equation}
As $\chi \delta^{1/2}_{P_\alpha}$ is unramified and $\sigma \circ \alpha^\vee$ is quadratic,
\begin{equation}\label{eq:3.5}
(\sigma \chi \delta_P^{1/2}) \circ \alpha^\vee |_{\mf o_F^\times} = 
\sigma \circ \alpha^\vee |_{\mf o_F^\times} \quad \text{is quadratic.}
\end{equation}
If \eqref{eq:3.5} is nontrivial, then
\begin{equation}\label{eq:3.15}
\int_{\varpi_F^{-n} \mf o_F^\times} (\sigma \chi \delta^{1/2}_{P_\alpha}) 
\circ \alpha^\vee (-x^{-1}) \textup{d}x =
(\sigma \chi \delta^{1/2}_{P_\alpha})(\varpi_F^n) \int_{\mf o_F^\times} \sigma 
\circ \alpha^\vee (-x^{-1}) \textup{d}x = 0 .
\end{equation}
In that case $J_{P_{-\alpha}|P_\alpha}(\sigma \otimes \chi) f_1 (s_\alpha) = 0$. On the other
hand, when \eqref{eq:3.5} is trivial: 
\begin{align*}
J_{P_{-\alpha}|P_\alpha}(\sigma \otimes \chi) f_1 (s_\alpha) & = 
\sum_{n=1}^\infty \int_{\varpi_F^{-n} \mf o_F^\times} (\sigma \chi)(\alpha^\vee (\varpi_F^n)) \,
\big| \alpha( \alpha^\vee (\varpi_F^n)) \big|^{1/2} \textup{d}x \\
& = \sum_{n=1}^\infty (\sigma \chi)(\alpha^\vee (\varpi_F))^n \text{vol}(\varpi_F^{-n} 
\mf o_F^\times) \, | \varpi_F^n | \\
& = \sum_{n=1}^\infty (\sigma \chi)(\alpha^\vee (\varpi_F))^n (1 - q_F^{-1}) \; = \; 
\frac{(1 - q_F^{-1})(\sigma \chi)(\alpha^\vee (\varpi_F))}{1 - (\sigma \chi)(\alpha^\vee (\varpi_F))} .
\end{align*}
Similar calculations show that
\begin{align*}
& J_{P_{-\alpha}|P_\alpha}(\sigma \otimes \chi) f_s (s_\alpha) = 1 , \\
& J_{P_{-\alpha}|P_\alpha}(\sigma \otimes \chi) f_s (1_G ) = 0 \quad \text{if }
\sigma \circ \alpha^\vee |_{\mf o_F^\times} \neq 1 , \\
& J_{P_{-\alpha}|P_\alpha}(\sigma \otimes \chi) f_s (1_G ) = 
\frac{1 - q_F^{-1}}{1 - (\sigma \chi)(\alpha^\vee (\varpi_F))} 
\quad \text{if } \sigma \circ \alpha^\vee |_{\mf o_F^\times} = 1 .
\end{align*}

\textbf{Case I: $\sigma \circ \alpha^\vee$ is unramified}\\
Here $\Rep (M_\alpha )^{\mf s}$ is isomorphic with the Iwahori-spherical Bernstein block and \\
$J_{P_{-\alpha}|P_\alpha}(\sigma \otimes \chi)$ restricts to a $\mc H (M_\alpha,I)$-homomorphism
\begin{equation}\label{eq:3.6}
I_{P_\alpha}^{M_\alpha} (\sigma \otimes \chi)^I \to I_{P_{-\alpha}}^{M_\alpha} (\sigma \otimes \chi )^I.
\end{equation}
The space $I_{P_{-\alpha}}^{M_\alpha}(\sigma \otimes \chi )^I$ has a basis $f'_1, f'_s$ where
supp$(f'_1) = P_{-\alpha} I$ and supp$(f'_s) = P_{-\alpha} s_\alpha I$. Abbreviating
$z_\alpha = (\sigma \otimes \chi) \circ \alpha^\vee (\varpi_F)$, the above calculations entail 
that the matrix of \eqref{eq:3.6} respect to the given bases is
\[
\begin{pmatrix}
q_F^{-1} & \frac{1 - q_F^{-1}}{1 - z_\alpha} \\
\frac{1 - q_F^{-1}}{z_\alpha^{-1} - 1} & 1
\end{pmatrix}.
\]
An equivalent result was obtained in \cite[Theorem 3.4]{Cas}.
Similarly one checks that $J_{P_{\alpha} | P_{-\alpha}}(\sigma \otimes \chi)$ restricts to
\[
\begin{pmatrix}
1 & \frac{1 - q_F^{-1}}{z_\alpha - 1} \\
\frac{1 - q_F^{-1}}{1- z_\alpha^{-1}} & q_F^{-1}
\end{pmatrix} : I_{P_{-\alpha}}^{M_\alpha} (\sigma \otimes \chi)^I \to 
I_{P_{\alpha}}^{M_\alpha} (\sigma \otimes \chi )^I .
\]
We find that $J_{P_\alpha | P_{-\alpha}}(\sigma \otimes \chi) J_{P_{-\alpha} | P_{\alpha}}
(\sigma \otimes \chi)$ restricts to
\begin{equation}\label{eq:3.7}
\Big( q_F^{-1} + \frac{(1 - q_F^{-1})^2}{(1 - z_\alpha)(1 - z_\alpha^{-1})} \Big) \mr{id} :
I_{P_\alpha}^{M_\alpha} (\sigma \otimes \chi)^I \to I_{P_\alpha}^{M_\alpha} (\sigma \otimes \chi)^I .
\end{equation}
We already know that $J_{P_\alpha | P_{-\alpha}}(\sigma \otimes \chi) J_{P_{-\alpha} | P_{\alpha}}
(\sigma \otimes \chi)$ is a scalar multiple of the identity on $I_{P_\alpha}^{M_\alpha}(\sigma \otimes
\chi)$, so \eqref{eq:3.7} gives that scalar. We note that \eqref{eq:3.7} has a pole at $z_\alpha = 1$
and that \eqref{eq:3.7} is zero if and only if $z_\alpha = q_F$ or $z_\alpha = q_F^{-1}$. As
$\sigma$ is unitary and $\chi \in \Hom (M_\alpha, \R_{>0})$, this is equivalent to
\begin{equation}\label{eq:3.10}
\sigma \circ \alpha^\vee = 1 \quad \text{and} \quad \chi \circ \alpha^\vee (\varpi_F) \in 
\{ q_F , q_F^{-1} \} .
\end{equation}
Since $M_\sigma^2 = T$, $h_\alpha^\vee$ generates $T / T^1$.
If $\alpha^\vee (\varpi_F^{-1}) = h_\alpha^\vee$, \eqref{eq:3.10} says that $q_\alpha = q_F$ and 
$q_{\alpha*} = 1$. If $\alpha^\vee (\varpi_F^{-1}) = 2 h_\alpha^\vee$, then \eqref{eq:3.10} means
$q_\alpha = q_F^{1/2} = q_{\alpha *}$. But in that case we can also define $X_\alpha (\chi) = 
\chi (\alpha^\vee (\varpi_F^{-1}))$ instead of $X_\alpha (\chi) = \chi (h_\alpha^\vee)$. These 
new $X_\alpha$ also form a root system, which embeds naturally in $R (\mc G,\mc T)^\vee$. From the
presentation after Corollary \ref{cor:1.1} one sees that this redefinition does not change the 
affine Hecke algebra. Hence we can achieve $q_\alpha = q_F, q_{\alpha*} = 1$ in all these cases.\\

\textbf{Case II: $\sigma \circ \alpha^\vee$ is ramified}\\
For $r \in \Z_{>0}$, $M_\alpha$ has compact open subgroups
\[
\begin{array}{lll}
J_r & = & x_\alpha (\varpi_F^r \mf o_F) \mc T (\varpi_F^r \mf o_F) x_{-\alpha}(\varpi_F^r \mf o_F) , \\
H_r & = & x_\alpha (\varpi_F^{2r-1} \mf o_F) \mc T (\varpi_F^r \mf o_F) x_{-\alpha}(\varpi_F \mf o_F) .
\end{array}
\]
Here $\mc T (\varpi_F^r \mf o_F)$ is a shorthand for the kernel of $\mc T (\mf o_F) \to
\mc T (\mf o_F / \varpi_F^r \mf o_F)$.

\begin{lem}\label{lem:3.6}
There exists $r \in \Z_{>0}$ such that $\mc T (\varpi_F^r \mf o_F) \subset \ker (\sigma)$ and
$\Rep (M_\alpha)^{\mf s}$ is a direct factor of 
\[
\Rep (M_\alpha, H_r) \cong \Mod (\mc H (M_\alpha, H_r)) .
\]
\end{lem}
\begin{proof}
Choose an odd $r \in \Z_{>0}$ such that $\mc T (\varpi_F^r \mf o_F) \subset \ker (\sigma)$ 
and $I_{P_\alpha}^{M_\alpha}(\sigma)^{J_r} \neq 0$. Then $I_{P_\alpha}^{M_\alpha}(\sigma \otimes \chi
)^{J_r} \neq 0$ for any $\chi \in X_\nr (T)$ because $J_r$ is compact. Hence
\[
\Rep (M_\alpha )^{\mf s} \subset \Rep (M_\alpha, J_r). 
\]
We note that $J_r$ is a normal subgroup of the hyperspecial parahoric subgroup $\mc M_\alpha (\mf o_F)$
of $M_\alpha$. It is known from \cite{BeDe} that $\Rep (M_\alpha, J_r)$ is a direct product of 
finitely many Bernstein blocks of $\Rep (M_\alpha)$, and that
\begin{equation}\label{eq:3.9}
\Rep (M_\alpha, J_r) \to \Mod (\mc H (M_\alpha, J_r )) : V \mapsto V^{J_r}
\end{equation}
is an equivalence of categories. Consider conjugation by $\alpha^\vee (\varpi_F^{(r-1)/2})$. 
This sends  $J_r$ to $H_r$ and induces equivalences of categories
\[
\Rep (M_\alpha,J_r) \cong \Rep (M_\alpha,H_r) ,\quad 
\Mod (\mc H (M_\alpha,J_r)) \cong \Mod (\mc H (M_\alpha,H_r)) . \qedhere
\]
\end{proof}

Lemma \ref{lem:3.6} tells us that most aspects of $I_{P_\alpha}^{M_\alpha}(\sigma \otimes \chi)$
can already be detected on $I_{P_\alpha}^{M_\alpha}(\sigma \otimes \chi )^{H_r}$. 

\begin{lem}\label{lem:3.7}
The double cosets in $P_\alpha \backslash M_\alpha / H_r$ can be represented by 
\[
\{ 1_G \} \cup \{ (u_{-\alpha}(z) s_\alpha : z \in \mf o_F / \varpi_F^{2r-1} \mf o_F \}.
\]
Similarly $P_{-\alpha} \backslash M_\alpha / H_r$ can be represented by $\{s_\alpha\} \cup
\{ u_{\alpha}(\mf o_F) / u_{\alpha}(\varpi_F^{2r-1} \mf o_F)\}$.
\end{lem}
\begin{proof}
From the Iwasawa decomposition $M_\alpha = P_\alpha \mc M_\alpha (\mf o_F)$ we get
\begin{equation}\label{eq:3.11}
P_\alpha \backslash M_\alpha / H_r \cong (P_\alpha \cap \mc M_\alpha (\mf o_F)) \backslash
\mc M_\alpha (\mf o_F) / H_r .
\end{equation}
Recall that by the Bruhat decomposition of $\mc M_\alpha (k_F)$:
\begin{equation}\label{eq:3.12}
\mc M_\alpha (\mf o_F) = I \sqcup I s_\alpha I = I \sqcup u_\alpha (\mf o_F) \mc T (\mf o_F)
u_{-\alpha}(\mf o_F) s_\alpha .
\end{equation}
Furthermore, we note that $(P_\alpha \cap \mc M_\alpha (\mf o_F)) H_r = I$ and 
\[
(P_\alpha \cap \mc M_\alpha (\mf o_F)) u_{-\alpha}(z) s_\alpha H_r = 
(P_\alpha \cap \mc M_\alpha (\mf o_F)) u_{-\alpha}(z + \varpi_F^{2r-1} \mf o_F) s_\alpha 
\qquad z \in \mf o_F .
\]
In combination with \eqref{eq:3.11} and \eqref{eq:3.12} that yields the desired representatives
for \eqref{eq:3.11}.

The representatives for the second double coset space are found in analogous fashion, now using
\[
\mc M_\alpha (\mf o_F) = s_\alpha I \sqcup s_\alpha I s_\alpha I =
u_{-\alpha}(\mf o_F) \mc T (\mf o_F) u_{\alpha}(\varpi_F \mf o_F) s_\alpha \sqcup
u_{-\alpha} (\mf o_F) \mc T (\mf o_F) u_{\alpha}(\mf o_F)
\]
instead of \eqref{eq:3.12}.
\end{proof}

It follows from Lemma \ref{lem:3.7} that $I_{P_\alpha}^{M_\alpha}(\sigma \otimes \chi )^{H_r}$
has a basis $\{f_1\} \cup \{ f_{z s} : z \in \mf o_F / \varpi_F^{2r-1} \mf o_F \}$. Here
supp$(f_1) = P_\alpha H_r = P_\alpha I$ as before and 
\begin{align*}
& \text{supp} (f_{zs}) = P_\alpha u_{-\alpha} (z) H_r s_\alpha = 
P_\alpha x_{-\alpha}(z + \varpi_F^{2r-1} \mf o_F) s_\alpha , \\
& f_{zs}(u_\alpha (x) t u_{-\alpha}(y) s_\alpha ) = (\sigma \chi \delta_{P_\alpha}^{1/2})(t) 
\qquad x \in F, y \in z + \varpi_F^{2r-1} \mf o_F, t \in T .
\end{align*}
The next result can be deduced from \cite[Theorem 6.3]{Roc1} when the characteristic of $F$ 
is not 2.

\begin{prop}\label{prop:3.8}
Recall that $\sigma \circ \alpha^\vee$ is ramified and $s_\alpha \cdot \sigma = \sigma$.
\enuma{
\item The functions $J_{P_{-\alpha}|P_\alpha}(\sigma \otimes \chi) f_1$ and
$J_{P_{-\alpha}|P_\alpha}(\sigma \otimes \chi) f_{zs}$ (with $z \in \mf o_F / \varpi_F^{2r-1} \mf o_F$)
of $\chi \in X_\nr (M_\alpha)$ do not have any poles. 
\item $\alpha \notin \Sigma_{\mc O,\mu}$.
}
\end{prop}
\begin{proof}
(a) Note that $J_{P_{-\alpha}|P_\alpha}(\sigma \otimes \chi)$ preserves the $H_r$-invariance
of an element $f$ of the given basis. By Lemma \ref{lem:3.7} it suffices to check the values of
$J_{P_{-\alpha}|P_\alpha}(\sigma \otimes \chi) f$ at $\{s_\alpha\} \cup u_\alpha (\mf o_F)$.
From the earlier computations \eqref{eq:3.33} and \eqref{eq:3.13} we know that $J_{P_{-\alpha}|
P_\alpha} (\sigma \otimes \chi) f_1$ does not have poles at $1_G$ or at $s_\alpha$. 
For $y \in \mf o_F \setminus \varpi_F^{2r-1} \mf o_F$ the multiplication rules in $SL_2 (F)$ 
(which surjects on $M_{\alpha,\der}$) enable us to compute
\begin{equation}\label{eq:3.14}
\begin{aligned}
J_{P_{-\alpha}|P_\alpha}(\sigma \otimes \chi) f_1 (u_\alpha (y)) 
& = \int_F f_1 \big( u_{-\alpha}(x) u_\alpha (y) \big) \textup{d}x \\
& = \int_F f_1 \big( u_\alpha (\frac{y}{1 + xy}) \alpha^\vee (\frac{1}{1 + xy}) 
u_{-\alpha}(\frac{x}{1 + xy}) \big) \textup{d} x \\
& = \int_F \big( \sigma \chi \delta_{P_\alpha}^{1/2})(\alpha^\vee (\frac{1}{1 + xy}) \big)
f_1 \big( u_{-\alpha}(\frac{x}{1 + xy}) \big) \textup{d} x .
\end{aligned}
\end{equation}
In terms of the new variable $x' := 1 +xy$ this becomes
\[
|y|^{-1} \int_F \big( \sigma \chi \delta_{P_\alpha}^{1/2})(\alpha^\vee (x')^{-1} \big)
f_1 \big( u_{-\alpha}(\frac{x' -1 }{y x'}) \big) \textup{d} x' 
\]
The integrand is nonzero if and only if $\frac{x' - 1}{y x'} \in \varpi_F \mf o_F$, which is equivalent to 
\[
(x' - 1) / x' \in y \varpi_F \mf o_F \subset \varpi_F \mf o_F. 
\]
That is only possible when $|x'| = 1$, so \eqref{eq:3.14} becomes an integral of a continuous function 
over the compact set $\mf o_F^\times$. In particular it converges and 
$J_{P_{-\alpha}|P_\alpha}(\sigma \otimes \chi) f_1$ does not have any poles.

With calculations as in \eqref{eq:3.13} we check the other basis elements $f_{zs}$:
\begin{align}
\nonumber J_{P_{-\alpha}|P_\alpha}(\sigma \otimes \chi) f_{zs} (s_\alpha) 
& = \int_F f_{zs} \big( u_{-\alpha}(x) s_\alpha \big) \textup{d} x = 
\text{vol}(z + \varpi_F^{2r-1} \mf o_F) = q_F^{1-2r} , \\
\nonumber J_{P_{-\alpha}|P_\alpha} (\sigma \otimes \chi) f_{zs} (u_\alpha (y) ) 
& = \int_F f_{zs} \big( u_{-\alpha}(x) u_\alpha (y) \big) \textup{d} x \\
\nonumber & = \int_{F^\times} f_{zs} \big( u_\alpha (-x^{-1}) u_{-\alpha}(x) u_\alpha (-x^{-1}) 
u_\alpha (y + x^{-1}) \big) \textup{d} x \\
& = \int_{F^\times} f_{zs} \big( s_\alpha (x) u_\alpha (y + x^{-1}) \big) \textup{d} x \\
\nonumber & = \int_{F^\times} (\sigma \chi \delta_{P_\alpha}^{1/2}) \big( \alpha^\vee (x^{-1}) \big)
f_{zs} \big( s_\alpha u_\alpha (y + x^{-1}) \big) \textup{d} x \\
\nonumber & = \int_{F^\times} (\sigma \chi \delta_{P_\alpha}^{1/2}) \big( \alpha^\vee (x^{-1}) \big)
f_{zs} \big( s_\alpha (x) u_\alpha (y + x^{-1}) \big) \textup{d} x \\
\nonumber & = \int_{F^\times} (\sigma \chi \delta_{P_\alpha}^{1/2}) \big( \alpha^\vee (x^{-1}) \big)
f_{zs} \big( u_{-\alpha}(-y - x^{-1}) s_\alpha \big) \textup{d}x .
\end{align}
When $-y \notin z + \varpi_F^{2r-1} \mf o_F$, this integral is supported on a compact subset of $F$,
and it converges. When $-y \in z + \varpi_F^{2r-1} \mf o_F$, the support condition on $x$ becomes
$|x| \geq q_F^{2r-1}$, and the integral reduces to 
\[
\sum_{n = 2r-1}^\infty \int_{\varpi_F^{-n} \mf o_F^\times}  
(\sigma \chi \delta_{P_\alpha}^{1/2})(\alpha^\vee (x^{-1})) \textup{d} x.
\]
Since $\sigma \circ \alpha^\vee$ is ramified and quadratic, it is nontrivial on $\mf o_F^\times$.
Then \eqref{eq:3.5} and \eqref{eq:3.15} show that every term of the above sum is zero. We conclude that
$J_{P_{-\alpha}|P_\alpha}(\sigma \otimes \chi) f_{zs}$ also does not have any poles. \\
(b) Part (a) and Lemma \ref{lem:3.7} show that $J_{P_{-\alpha}|P_\alpha}(\sigma \otimes \chi)$ does not
have any poles on $I_{P_\alpha}^{M_\alpha}(\sigma \otimes \chi )^{H_r}$. Similar computations (which
we omit) show that $J_{P_{\alpha}|P_{-\alpha}}(\sigma \otimes \chi)$ does not
have any poles on $I_{P_{-\alpha}}^{M_\alpha}(\sigma \otimes \chi )^{H_r}$. By Lemma \ref{lem:3.6}
they neither have poles on, respectively, $I_{P_\alpha}^{M_\alpha}(\sigma \otimes \chi )$ and 
$I_{P_{-\alpha}}^{M_\alpha}(\sigma \otimes \chi )$. Then \eqref{eq:1.3} says that
$\mu^{M_\alpha}(\sigma \otimes \chi)$ is nonzero for $\chi \in X_\nr (T)$, which by definition
means $\alpha \notin \Sigma_{\mc O,\mu}$.
\end{proof}

Let us combine the conclusions for all possible $\sigma \circ \alpha^\vee$:

\begin{thm}\label{thm:3.9}
Suppose that $\alpha \in \Sigma_{\mc O,\mu}$, for a principal series Bernstein component of a
$F$-split group $G$. Define $X_\alpha (\chi) = \chi (\alpha^\vee (\varpi_F^{-1}))$. 
Then $\sigma \circ \alpha^\vee = 1$ and $q_\alpha = q_F, q_{\alpha*} = 1$.
\end{thm}

\subsection{Principal series of quasi-split groups} \
\label{par:quasisplit}

We consider a quasi-split non-split connected reductive $F$-group $\mc G$. By Section
\ref{sec:red} we may suppose that $\mc G$ is absolutely simple. Then it is an outer form of Lie
type $A_n, D_n$ or $E_6$. 

Let $\mc T$ be the centralizer of a maximal $F$-split torus $\mc S$ in $\mc G$, and let
$\sigma$ be a character of $T$ satisfying Condition \ref{cond:3.1}. Let Gal$(F_s / \tilde F)$ 
be the normal subgroup of Gal$(F_s/ F)$ that acts trivially on $X^* (\mc T)$, so that 
$\tilde F / F$ is a minimal Galois extension splitting $\mc T$. 

Consider a root $\alpha \in \Sigma_{\mc O,\mu}$. By a suitable choice of a basis of 
$\Sigma (G,S) \subset \Sigma (G,A_T)$, we may assume that $\alpha$ is simple. It corresponds to a 
unique Galois orbit $\mb W_F \alpha_T$ in $\Sigma (\mc G, \mc T)$. Then
\begin{equation}\label{eq:3.28}
U_\alpha (F) = \Big( \prod_{\beta_T \in \mb W_F \alpha_T} U_{\beta_T} (F_s) \Big)^{\mb W_F} 
\cong U_{\alpha_T}(F_s )^{\mb W_{F,\alpha_T}} \cong F_s^{\mb W_{F,\alpha_T}} =: F_\alpha .
\end{equation}
The field $F_\alpha$ does not depend on the choice of $\alpha_T$ (up to isomorphism) and is
known as a splitting field for $\alpha$.

By construction the numbers $q_\alpha, q_{\alpha*}$ depend only on the group $M_\alpha$. 
Parts (b--c) of Proposition \ref{prop:2.4} apply, so
we may even replace $\mc M_\alpha$ by its derived subgroup $\mc M_{\alpha,\der}$.

Suppose for the moment that the elements of $\mb W_F \alpha_T \subset \Sigma (\mc G,\mc T)$ are
mutually orthogonal. Then $\mc M_{\alpha,\der}$ is isomorphic to the restriction of scalars,
from $F_\alpha$ to $F$, of $SL_2$ or $PGL_2$. Now $q_\alpha$ and
$q_{\alpha *}$ can be computed in $SL_2 (F_{\alpha})$ or $PGL_2 (F_{\alpha})$, as in Paragraph
\ref{par:split}. (Recall that even for $PGL_2$ we insisted that $X_\alpha$ is based on $\alpha^\vee$
rather than on $h_\alpha^\vee$.) 
By Theorem \ref{thm:3.9} $\sigma \circ \alpha^\vee = 1$, $q_{\alpha *} = 1$
and $q_\alpha = q_{F_{\alpha}}$ is the cardinality of the residue field of $F_{\alpha_T}$.
From Galois theory for local fields \cite{Ser} it is known that
\begin{multline}
|\mb W_F \alpha_T| = [\mb W_F : \mb W_{F,\alpha_T}] = e_{F_{\alpha} / F} f_{F_{\alpha} / F} = \\
[\mb I_F : \mb I_F \cap \mb W_{F,\alpha_T}] \cdot [\mb W_F / \mb I_F : \mb W_{F,\alpha_T} 
\mb I_F / \mb I_F ] = |\mb I_F \alpha_T| \cdot f_{F_{\alpha} / F} .
\end{multline}
Since $\mb I_F$ is normal in $\mb W_F$, the number
\begin{equation}\label{eq:3.16}
q_{F_{\alpha}} = q_F^{f_{F_{\alpha} / F}} = q_F^{|\mb W_F \alpha_T | / |\mb I_F \alpha_T|}
\end{equation}
depends only on $\alpha$, and not on the choice of $\alpha_T$. This leads to the
possibilities for the Dynkin diagrams (with Galois action indicated by arrows), the relative
root systems and the $q_\alpha$ in Table \ref{fig:DD}. We stress that the parameters 
$q_\alpha$ only come into play when $\alpha \in \Sigma_{\mc O,\mu}$, for 
$\alpha \in \Sigma (A_T) \setminus \Sigma_{\mc O,\mu}$ they are not defined. Recall that the
root system underlying $\mc H (\mc O,G)$ is $\Sigma_{\mc O}^\vee$, which is a rescaled version
of $\Sigma_{\mc O,\mu}^\vee$, so obtained from the dual of the root system on the right
hand side of the table.
\begin{table}
\caption{Dynkin diagrams and parameters for quasi-split groups \label{fig:DD}}
\includegraphics[width=12cm]{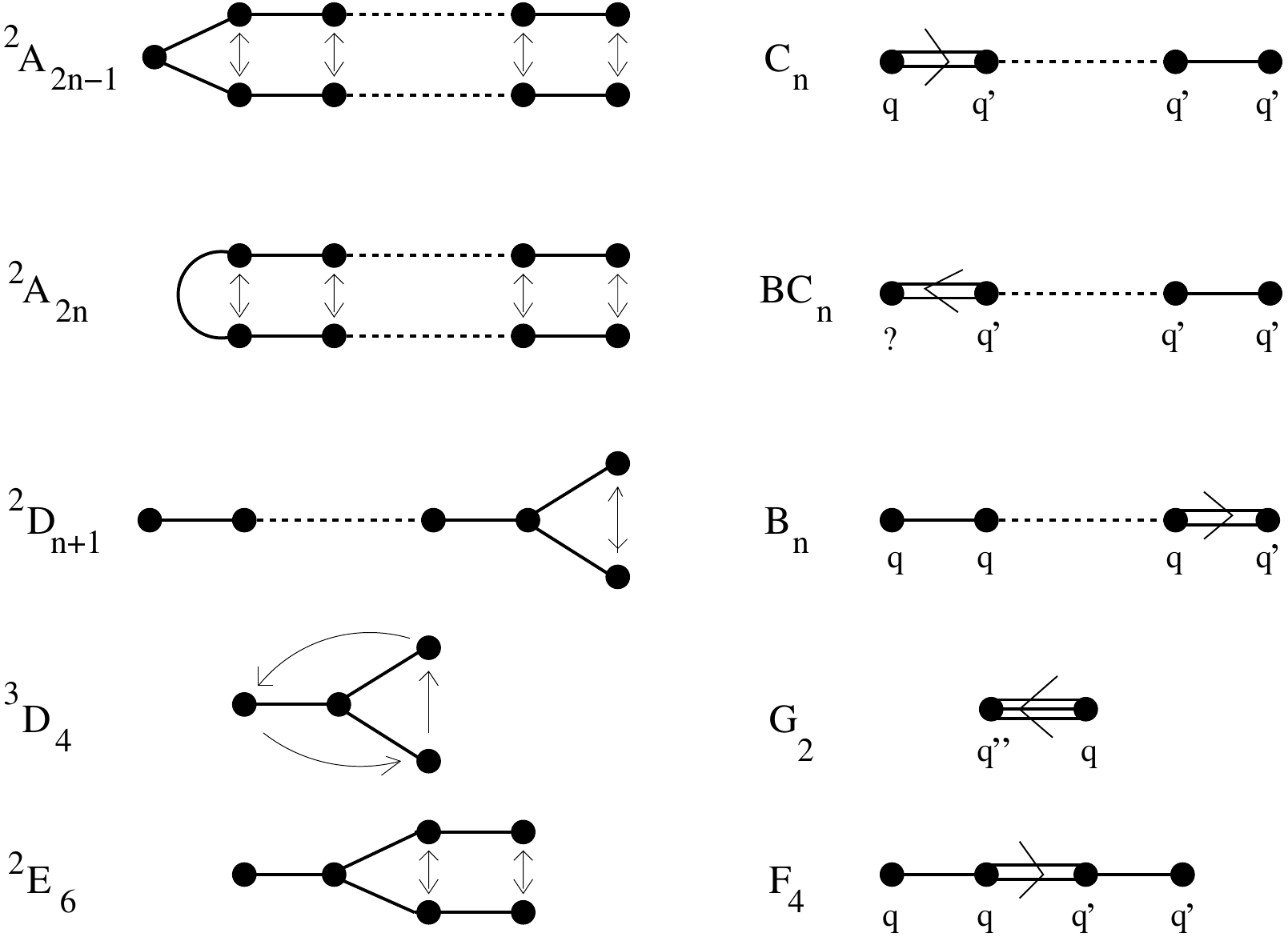}
\end{table}

In Table \ref{fig:DD} $q = q_F$ and $q' \in \{ q_F, q_F^2 \}$, according to \eqref{eq:3.16}. 
For a $F$-group of type ${}^3 D_4$, $[\tilde F : F]$ can be of degree 3 or 6. In both cases 
$[F_{\alpha} : F] = 3$ for the roots $\alpha$ not fixed by $\mb W_F$, so $q'' \in \{ q_F, q_F^3\}$. 
Thus Conjecture \ref{conj:1} holds in all these cases.

It remains to consider the case where the elements of $\mb W_F \alpha_T$ are not orthogonal. From
the above diagrams we see that this happens only once (up to Weyl group conjugacy) for absolutely 
simple groups, namely for certain pairs of roots in type ${}^2 A_{2n}$. With Proposition \ref{prop:2.4} 
we can transfer the determination of $q_\alpha$ and $q_{\alpha *}$ (which no longer needs to be 1) 
to the simply connected cover of $\mc M_{\alpha,\der}$, which is isomorphic to $SU_3$. This does not 
change the $q$-parameters, by Proposition \ref{prop:2.4}.(b--c). Because we cannot reduce the issue 
to $SL_2$ or $PGL_2$, the necessary computations are more involved. 

With Section \ref{sec:red} we can further transfer these computations to the $F$-group $U_3$, which is 
a little easier. Indeed, for that group all the Hecke algebras were computed by means of types by the 
author's PhD student Badea \cite{Bad}. In particular, it was shown in \cite[\S 2.7 and \S 5.2.1]{Bad}  
that only the following possibilities can arise:
\begin{itemize}
\item[(i)] $q_\alpha = q_{F_\alpha} = q_F, q_{\alpha*} = 1$,
\item[(ii)] $q_\alpha = q_F, q_{F_\alpha} = q_F^2, q_{\alpha^*} = 1$,
\item[(iii)] $q_\alpha = q_{F_\alpha} = q_F^2, q_{\alpha*} = q_F$.
\end{itemize}
The option (i) leads to an affine Hecke algebra with all $q_\alpha$ for $\alpha \in \Sigma_{\mc O,\mu}$
equal, which occurs in Table \ref{fig:1}. In case (ii) the connected component $\Sigma^\vee_{\mc O,j}$ 
of $\Sigma_{\mc O}^\vee$ containing $h_\alpha^\vee$ has type $B_m$ (for some $m \leq n$) and 
$q_\beta = q_F^2$ for all other simple roots in $\Sigma^\vee_{\mc O,j}$. The possibility (iii) arises 
only from the Iwahori-spherical principal series. The latter consists of unipotent representations, 
so that Conjecture \ref{conj:1} is automatic.

We have to be a little careful, because it is assumed in \cite{Bad} that the residual characteristic of 
$F$ is not 2. For the Iwahori-spherical principal series that is not a problem, those affine Hecke
algebras are known from \cite{Bor} regardless of the residual characteristic. For ramified characters
of $T \subset U_3 (F)$ it is troublesome, because some computations in \cite{Bad} change 
substantially in residual characteristic 2. To be sure in those cases as well we refer to Theorem
\ref{thm:3.2}, where all the $q$-parameters for $U_n (F)$ are computed in a different way (for 
arbitrary $F$ but with much heavier machinery). From Lemma \ref{lem:3.10} one sees that the only 
options for $U_3 (F)$ in residual characteristic 2 are still (i), (ii) and (iii).

Let us state the above conclusions concisely:

\begin{cor}\label{cor:3.12}
Conjecture \ref{conj:1} holds for all Bernstein blocks in the principal series of a quasi-split
connected reductive group $G$ over $F$. When we base $X_\alpha$ on $\alpha^\vee$, $q_{\alpha*} = 1$
and $q_\alpha = q_{F_\alpha}$ (except for one root in type ${}^2 A_{2n}$). 
\end{cor}

\subsection{Inner forms of Lie type $A_n$} \
\label{par:A}

We consider simple $F$-groups $\mc G$ that are inner forms of a split group of type $A_{n-1}$.
The simply connected cover of $\mc G$ is an inner form of $SL_n$, so isomorphic to the derived
subgroup of an inner form of $GL_n$. In view of Section \ref{sec:red} it suffices to consider
the latter case, so with $G$ isomorphic to $GL_m (D)$ for a division algebra $D$ with centre $F$
and $\dim_F (D) = (n/m)^2$.

For every Bernstein block $\Rep (G)^{\mf s}$ there exists a type $(J,\rho)$ \cite{SeSt6}. 
We can write $\mf s = [M,\sigma]_G$ in the form
\[
M = \prod\nolimits_i GL_{m_i} (D)^{e_i}, \sigma = \boxtimes_i \sigma_i^{\otimes e_i} , 
\]
where the various $\sigma_i$ differ by more than an unramified character. The associated Hecke 
algebra $\mc H (G,J,\rho)$ is a tensor product of affine Hecke algebras of type $GL_{e_i}$ 
\cite{SeSt4}, so the underlying root system has irreducible components of type $A_{e_i -1}$, 
for suitable $e_i \leq n$. The same result was obtained around the same time in \cite{Hei2},
using $\Pi^{\mf s}$.
The parameters of such a type $GL_{e_i}$ affine Hecke algebra were determined explicitly in 
\cite[Th\'eor\`eme 4.6]{Sec}, they are of the form $q_\alpha = q_F^f, q_{\alpha*} = 1$ for a 
specific positive integer $f$. Thus $\lambda$ and $\lambda^*$ are constant 
and equal to $f$ on the underlying root system $A_{e_i -1}$. From \cite[1.13--1.15]{Hei2} or 
\cite{SeSt6} we also see that 
\begin{equation}\label{eq:3.24}
W(M,\mc O) = W(\Sigma_{\mc O,\mu}) \cong \prod\nolimits_{e_i} S_{e_i}
\end{equation}
and $R(\mc O) = \{1\}$. From that, \eqref{eq:1.7} and 
\[
\Mod- \mc H (G,J,\rho) \cong \Rep (G)^{\mf s} \cong \End_G (\Pi^{\mf s}) -\Mod
\]
we deduce that $\mc H (G,J,\rho)$ is Morita equivalent with $\mc H (\mc O ,G )^{op}$. These are 
both affine Hecke algebras, and then Morita equivalence implies that $\mc H (G,J,\rho)$ and
$\mc H (\mc O ,G)$ and $\mc H (\mc O ,G)^{op}$ are isomorphic. We summarise:

\begin{thm}\label{thm:3.1} \textup{[Heiermann, S\'echerre--Stevens]} \\
Let $\mc G$ be an inner form of a simple $F$-split group of type $A_{n-1}$, and let $\mf s$ be
an inertial equivalence class for $G$. Then the root system underlying $\mc H (\mc O ,G)$ has
irreducible components of type $A_{e-1}$ with $e \leq n$. The label functions $\lambda, \lambda^*$
are constant on $A_{e-1}$, and equal to an integer $f$. 
\end{thm}

We note that such parameters already occur for Iwahori-spherical representations. Namely, consider
$GL_m (D)$ where $\dim_F (D) = f^2$. Its Iwahori--Hecke algebra is isomorphic with
an affine Hecke algebra of type $GL_e$ with parameters $q_F^f$.

More explicit information about $f$ comes from \cite[Introduction]{SeSt6}. Every type
$GL_e$ affine Hecke algebra as above comes from a supercuspidal representation $\pi^{\boxtimes e}$
of $GL_{m'/e}(D)^e$ for some $m' \leq m$. Then $f$ equals the torsion number
\[
t_\pi = | X_\nr (GL_{m'/e}(D),\pi)|
\] 
times the reducibility number $s_\pi$. The number $s_\pi$ detects when the (normalized) 
parabolic induction of 
\[
\pi \boxtimes \pi \nu^{s_\pi} \in \Irr \big( GL_{m'/e}(D)^2 \big)
\] 
is a reducible representation of $GL_{2m'/e}(D)$, where $\nu$ stands for $| \cdot |_F$ 
composed with the reduced norm map $GL_{m'/e}(D) \to F^\times$. 
With the Jacquet--Langlands correspondence \cite{Badu,DKV}
one can relate the torsion and reducibility numbers of $\pi$ to the same numbers for a specific
discrete series representation $JL (\pi)$ of $GL_{n m'/e m}(F)$. More information about those
numbers is already known from \cite{BeZe,BuKu1}. From that or from a comparison with Langlands 
parameters as in \cite[p. 57]{AMS3}, one sees that $s_\pi$ divides ${\displaystyle \frac{n m'}{m e}}$ 
and that $t_\pi$ divides ${\displaystyle \frac{n m'}{m e s_\pi}}$. Therefore
\begin{equation}\label{eq:3.26}
f = s_\pi t_\pi \quad \text{divides} \quad \frac{n m'}{m e} \leq \frac{n}{e} .
\end{equation}
We note that in all these cases $\alpha^\vee$ generates $H_M (M_\alpha^1 / M^1)$, because
the derived groups are simply connected. The torsion number $t_\pi$ says precisely that
\[
H_M (M_\sigma^2 \cap M_\alpha^1 / M^1) = \Z t_\pi \alpha^\vee .
\]
Consider a $F$-split connected reductive group $\mc M_\alpha$ with root system of type
$A_{n+m-1}$. Let $\mc M$ be the standard $F$-Levi subgroup of $\mc M_\alpha$ obtained by
omitting a simple root $\alpha$, with root system of type $A_{n-1} \times A_{m-1}$.
Then the simply connected cover of $M_\der$ is isomorphic to $SL_n (F) \times SL_m (F)$. 

Put $\mf s = [M,\sigma]_{M_\alpha}$ for some $\sigma \in \Irr_\cusp (M)$.
The inflation of $\sigma |_{M_\der}$ to the simply connected cover $M_\Sc$ of $M_\der$ 
can be written as a finite direct sum 
\[
\bigoplus\nolimits_i \sigma_i \boxtimes \sigma'_i \text{ with }
\sigma_i \in \Irr_\cusp (SL_n (F)), \sigma'_i \in \Irr_\cusp (SL_m (F)) . 
\] 
From Theorem \ref{thm:3.1}, \eqref{eq:3.26} and Section \ref{sec:red} we obtain the 
following criterion for Hecke algebra parameters in split type A groups:

\begin{cor}\label{cor:3.11}
Let $M_\alpha$, $M$ and $\sigma$ be as above.
\enuma{
\item If $n \neq m$, then $s_\alpha$ does not give rise to an element of $N_{M_\alpha}(M)/M$.
\item Suppose that $n = m$ and that, for any $i$, $\sigma_i$ and $\sigma'_i$ are not
isomorphic. Then $s_\alpha$ does not give rise to an element of $W (M,\mc O)$.
\item Suppose that $n = m$ and that, for at least one $i$, $\sigma_i$ and $\sigma'$
are isomorphic. Then $\Sigma_{\mc O,\mu} = \{\alpha,-\alpha\}$ and $s_\alpha$ gives rise to an 
element of $W(M,\mc O)$ that exchanges the two almost direct simple factors of $M_\der$. 

When $M_\alpha = GL_{2n}(F)$, the $q$-parameters for $\mc H (\mc O,M_\alpha)$ are $q_{\alpha*} = 1$ 
and $q_\alpha = q_F^f$. Here $f$ is the torsion number $t_{\sigma_i} \in \Z_{>0}$, which divides $n$.
}
\end{cor}
\begin{proof}
(a) This is clear, because such an element would have to exchange the two almost direct
simple factors of $M_\der$. \\
(b) Now $s_\alpha$ does give an element of $N_{M_\alpha}(M) / M$, which exchanges the two
almost direct simple factors of $M_\der$. By Proposition \ref{prop:2.2}
we may lift to the simply connected cover $M_\Sc$, picking one irreducible constituent
$\sigma_i \otimes \sigma'_i$ of the inflation of $\sigma |_{M_\der}$. As $M_\Sc$ does not
have nontrivial unramified characters, stabilizing $\mf s$ has become stabilizing $\sigma_i
\otimes \sigma'_i$. Clearly $s_\alpha$ does that if and only if $\sigma_i$ and $\sigma'_i$
are isomorphic.\\
(c) This follows from Theorem \ref{thm:3.1}.
\end{proof}

In part (c) for $M_\alpha \neq GL_{2n}(F)$, it may still be necessary to apply Proposition
\ref{prop:2.4}.d to obtain the precise parameters.
\begin{ex}\label{ex:A}
Consider the inclusion
$\eta : SL_4 (F) \to GL_4 (F)$ and the Levi subgroups $M = GL_2 (F)^2$ and $\tilde M = S(GL_2 (F)^2)$.
Let $\sigma \in \Irr_\cusp (GL_2 (F))$ with 
\[
X_\nr (GL_2 (F),\sigma) = \{1,\chi_-\}.
\] 
We may assume that $\sigma |_{SL_2 (F)}$ decomposes as a direct sum of two irreducible representations, 
both stable under diag$(a,b) \in GL_2 (F)$ for all $a,b \in \mf o_F^\times$. Then $\sigma \otimes \sigma 
\in \Irr (M)$ and $\eta^* (\sigma \otimes \sigma)$ is a direct sum of two irreducible $\tilde M
$-representations $\tilde \sigma_1, \tilde \sigma_2$, permuted by diag$(\varpi_F,1) \in M$. Here 
\[
\eta^* (X_\nr (M,\sigma \otimes \sigma)) = \{ 1, \chi_- \otimes 1 \}
\]
but tensoring by $\chi_- \otimes 1$ exchanges $\tilde \sigma_1$ and $\tilde \sigma_2$. It follows that
\[
X_\nr (\tilde M, \tilde \sigma_1) = X_\nr (\tilde M, \tilde \sigma_2) = \{1\}. 
\]
The root systems of the Hecke algebras are $\{\alpha,-\alpha\}$ and $\{\tilde \alpha, -\tilde \alpha\}$, 
while $h_\alpha^\vee = \eta (h_{\tilde \alpha}^\vee)^2 \in M / M^1$. So this is an instance
of Proposition \ref{prop:2.4}.d.(iii).
\end{ex}

\subsection{Classical groups} \
\label{par:classical}

We look at classical groups associated to Hermitian forms on $F$-vector spaces. 
Let $\mc G^*$ be a symplectic group or a special orthogonal group (not necessarily split). It was 
shown in \cite{Hei2} that $\End_G (\Pi^{\mf s})$ is Morita equivalent with the crossed product of 
$\mc H (\mc O,G)$ and $R(\mc O)$, where $\mc H (\mc O,G)$ is a tensor product of affine Hecke 
algebras with lattice $\Z^e$ and root system $A_{e-1}, B_e$, $C_e$ or $D_e$. When $\mc G^*$ is 
$F$-split, the parameters are computed in \cite{Hei1}, relying on \cite{Moe2}. Later the 
(quasi-)split assumption in \cite{Moe2} was lifted in \cite{MoRe}, which means that \cite{Hei1}
also applies to pure inner forms of quasi-split groups.

We also allow $\mc G^*$ to be a special unitary group. With Section \ref{sec:red} we 
reduce that to $U_n$, a unitary group which splits over a separable quadratic 
extension $\tilde F / F$. According to \cite[Theorem 1.8 and \S C.5]{Hei3}, the above description 
of $\mc H (\mc O,G)$ is also valid for $U_n$. Unfortunately there is no real proof of these
claims in \cite{Hei3}, but it is similar to \cite{Hei2} and in fact an instance of the more
general results of \cite{SolEnd}. Also according to \cite[\S C]{Hei3}, the parameters of these 
affine Hecke algebras can be computed as in \cite{Hei1}. This uses the results of \cite{Moe1,Moe2,Moe3}.

Recall that every $F$-Levi subgroup of $G^*$ is of the form
\begin{equation}\label{eq:3.17}
\mc M^* (F) \cong \prod\nolimits_i GL_{n_i}(F') \times \mc H^* (F) ,
\end{equation}
where $\mc H^*$ is of the same type as $\mc G^*$, but of smaller rank. Here $F' = \tilde F$ for
(special) unitary groups and $F' = F$ otherwise. Let $f$ be residue degree of $F' / F$, so
2 for unramified (special) unitary groups and 1 otherwise.

Let $\mc G$ be a group isogenous to $\mc G^*$ as above and let $\mc M$ be a $F$-Levi subgroup of 
$\mc G$. Consider an inertial equivalence class $\mf s = [M,\sigma ]_G$. Via Lemma \ref{lem:2.1} for 
groups isogenous to $M$, $\sigma$ gives rise to an irreducible representation $\tilde \sigma$ of
$\mc M^* (F)$. By Proposition \ref{prop:2.4} the choice does not matter.
In terms of \eqref{eq:3.17}, $\tilde \sigma$ is the tensor product 
of supercuspidal representations $\tau_i$ of the factors $GL_{n_i}(F')$ and one supercuspidal 
representation of $\mc H^* (F)$. As explained in
\cite[\S 2]{Hei1}, one can adjust $\tilde \sigma$ inside its inertial equivalence class to 
bring the $\tau_i$ in a better position, so that in particular as many of them as possible are equal.
We fix an irreducible component $\Sigma_{\mc O,j}^\vee$ of $\Sigma_{\mc O}^\vee$, as in Section 
\ref{sec:red}. Of the various $\tau_i$ a unique $\tau$ is relevant for $\Sigma_{\mc O,j}^\vee$.
We write $n_i = d_\tau$, so that $\tau \in \Irr (GL_{d_\tau}(F'))$.

\begin{thm}\label{thm:3.2}
Assume we are in the setting described in this paragraph. 
Let $t = t_\tau \in \Z_{>0}$ be the torsion number of $\tau$, a divisor of $d_\tau$.
\enuma{
\item When $\Sigma_{\mc O,j}^\vee \cong C_e$ and $h_\alpha^\vee$ is a long root, there exists an 
integer $a_+ \in \Z_{> 0}$ such that $q_\alpha = q_{F'}^{t a_+}, q_{\alpha*} = 1$ and 
$\lambda (\alpha) = \lambda^* (\alpha) = f t a_+$.
\item When $\Sigma_{\mc O,j}^\vee \cong B_e$ and $h_\alpha^\vee$ is a short root, there exist 
integers $a \geq a_- \geq -1$ such that $q_\alpha = q_{F'}^{t (a+1)/2}, q_{\alpha*} = 
q_{F'}^{t (a_- + 1)/2}$ and $\lambda (\alpha) =  f t (a + a_- + 2)/2$, 
$\lambda^* (\alpha) = f t (a - a_-)/2$.
\item For all other $h_\alpha^\vee \in \Sigma_{\mc O,j}^\vee$, $q_\alpha = q_{F'}^t, 
q_{\alpha *} = 1$ and $\lambda (\alpha) = \lambda^* (\alpha) = f t$.
}
Suppose that $\mc M$ is isogenous to \eqref{eq:3.17} and that the complex dual group of $\mc H^*$ 
consists of matrices of size $N^\vee$. Then
\[
\Big\lfloor \big( \frac{a+1}{2} \big)^2 \Big\rfloor + 
\Big\lfloor \big( \frac{a_- +1}{2} \big)^2 \Big\rfloor \leq N^\vee d_\tau^{-1}
\]
in case (b) and $a_+^2 \leq 2 N^\vee d_\tau^{-1} + 1$ in case (a).
\end{thm}
We note that for a maximal Levi subgroup $M^* = GL_{d_\tau}(F') \times H^*$ of $G^*$,
the root system $\Sigma_{\mc O}^\vee$ has type $B_1$ or is empty.
\begin{proof}
By Corollary \ref{cor:4.5} we may assume that $F$ has characteristic zero.

For $G^*$ the claims (b) and (c) follow from \cite[Proposition 3.4]{Hei1} and \cite[\S C.5]{Hei3}.
The role of the torsion number $t$ is to replace the sublattice $\Z^e$ of $M / M^1$ corresponding 
to $\Sigma_{\mc O,j}^\vee$ by $(t \Z)^e$, which is a direct summand of $M_\sigma^2 / M^1$.
In this process, all the labels $\lambda (\alpha)$ and $\lambda^* (\alpha)$ are multiplied by $t$. 

The numbers $a, a_-$ come from \cite[Proposition 4]{Moe1}, \cite[\S 1.3--1.4]{Moe2} and  
\cite[Th\'eor\`eme 3.1]{Moe3}, where they are computed in terms of reducibility of the parabolic 
induction of a supercuspidal representation $\tau \otimes \pi$ of $GL_{d_\tau}(F') \times H^*$. 
This shows that in general we have to use $F'$ instead of $F$. For (special) unitary groups, the 
factors $GL_{d_\tau}(\tilde F)$ in \eqref{eq:3.17} cause another
factor $f$ in all the parameters, as explained in \cite[\S C]{Hei3}.

Recall that the Jordan block of $\pi \in \Irr (H^*)$ is built from the pairs $(\tau,a)$ that we 
consider (but those with $a \leq 0$ omitted), by adding new pairs according to the rule
\[
\text{if } (\tau,a) \in \mr{Jord}(\pi) \text{ and } a > 2 \text{ then } (\tau,a-2) \in \mr{Jord}(\pi) .
\]
It was shown in \cite[\S 1.4]{Moe2} and \cite[Proposition 4]{Moe1} that
\begin{equation}\label{eq:3.30}
\sum\nolimits_{(\tau,a) \in \mr{Jord}(\pi)} a d_\tau = N^\vee .
\end{equation}
We fix a $\tau$ and let $\tau_-$ be the unramified twist of $\tau$ from which $a_-$ is determined.
Isolating the terms with $\tau$ and $\tau_-$ in \eqref{eq:3.30}, we obtain
\begin{multline}\label{eq:3.31}
N^\vee \geq \sum\nolimits_{a' : (\tau,a') \in \mr{Jord}(\pi)} a' d_\tau + 
\sum\nolimits_{a' : (\tau_-,a') \in \mr{Jord}(\pi)} a' d_{\tau_-} \\
= \Big\lfloor \big( \frac{a+1}{2} \big)^2 \Big\rfloor d_\tau + 
\Big\lfloor \big( \frac{a_- +1}{2} \big)^2 \Big\rfloor d_\tau.
\end{multline}
Case (a) for $\mc G^*$ is not mentioned explicitly in \cite{Hei1}, it is an instance of case (b) 
when we focus on the Weyl group (not on the root system). 
As the lattice containing $\Sigma_{\mc O,j}^\vee$ is isomorphic to $\Z^e$, the construction of 
$h_\alpha^\vee$ entails that $h_\alpha^\vee \notin 2 \Z^e$, so that $\Sigma_{\mc O,j}^\vee$ does not 
have type $C_e$. Still, this root system occurs if $\Sigma_{\mc O,j}^\vee \cong B_e$ and 
$q_\alpha = q_{\alpha*}$. Then we can replace $h_\alpha^\vee$ by $h_{\alpha / 2}^\vee = 
2 h_\alpha^\vee$, $X_\alpha$ by $X_\alpha^2$, $B_e$ by $C_e$, $q_\alpha$ by $q_\alpha^2 = q_\alpha
q_{\alpha*}$ and $q_{\alpha*}$ by 1, without changing the Hecke algebra $\mc H (\mc O,G)$. We find 
\[
\lambda (\alpha / 2) = \lambda^* (\alpha /2) = 
\lambda (\alpha) + \lambda^* (\alpha) = t f (a + 1), 
\]
so the $a_+$ for $2 h_\alpha^\vee$ is 1 plus the $a$ from $B_e$. When $a_+$ is odd, the previously
established bound on $a = a_-$ directly yields the new bound on $a_+$. When $a_+$ is even, 
\eqref{eq:3.31} says
\[
2 \Big( \big( \frac{a + 1}{2} \big)^2 - \frac{1}{4} \Big) \leq N^\vee d_\tau^{-1}, \text{ so }
a_+^2 = (a+1)^2 \leq 2 N^\vee d_\tau^{-1} + 1.
\]
When $\mc G$ is a quotient of $\mc G^*$, Section \ref{sec:red} enables us to reduce to $G^*$.
According to Proposition \ref{prop:2.4}, in the process the labels for $\alpha$ must be 
multiplied by some $N_\alpha \in \{1/2,1,2\}$. But for type $A$ roots nothing really changes along
$G^* \to G$ (the computations can be placed entirely in a general linear group) so $N_\alpha = 1$. 
For other roots $\beta$ either
$N_\beta = 1$ or (if $\Sigma_{\mc O,j}^\vee \cong C_e$) $N_\beta = 2$ or (if $\Sigma_{\mc O,j}^\vee 
\cong B_e$) $N_\beta = 1/2$. In the last two cases types $C_e$ and $B_e$ are exchanged, and the 
relations between the parameters are the same as between cases (a) and (b) of the theorem.

In the remaining cases $\mc G$ is a spin group (or a half-spin group, but by passing to the simply
connected cover, as allowed by Corollary \ref{cor:2.6}, we reduce that case to a spin group).
Then $\mc G$ can be embedded in a general spin group $\mr{GSpin}_n$ -- not necessarily split,
but at least a pure inner form of a quasi-split group. The Levi subgroups of $\mr{GSpin}_n$ follow 
the same pattern as for $SO_n$, and their discrete series representations can be classified as for 
special orthogonal groups, see \cite[\S 1.3, \S 1.5]{Moe2} and \cite[\S 4.4, \S 6.3]{Moe3}. 
Consequently the results of \cite{Hei1,Hei2} also hold for 
$\mr{GSpin}_n$, with the only difference that the lattices in the Hecke algebras have rank one higher 
than for $SO_n$. Thus Theorem \ref{thm:3.2} holds for $\mr{GSpin}_n$, possibly with correction factors
$N_\beta \in \{1/2,1,2\}$ as above for quotients of $G^*$. By Corollary \ref{cor:2.7} the theorem also 
holds for the derived group $\mc G$ of $\mr{GSpin}_n$.
\end{proof}

In the generality of Theorem \ref{thm:3.2} it is hard to make the integers $a_+, a$ and $a_-$ more 
explicit, since they depend in a very subtle way on the involved supercuspidal representations. If one 
restricts to specific classes of Bernstein components, more can be said about the Hecke algebra 
parameters. In particular, for the principal series representations of quasi-split classical groups the 
method in Paragraphs \ref{par:split} and \ref{par:quasisplit} yields the concrete $q$-parameters. 

To check the integrality of the label functions $\lambda,\lambda^*$, we analyse the parity of $a$ and 
$a_-$. As explained in \cite[\S 1]{Hei1}, this boils down to comparing the Langlands parameter $\rho$ 
of a supercuspidal representation of $GL_k (F')$ with the Langlands parameter of a supercuspidal 
representation of $H$ (a classical group of the same type as $G$). 

For $G$ of Lie type $B_n, C_n, D_n$ or ${}^2 D_n$, $\rho$ is self-dual. Then $a$ is odd if and only
if $\rho$ and the complex dual group $H^\vee$ of $H$ have the same type (orthogonal or symplectic).

For $G$ of Lie type ${}^2 A_{n-1}$, $\rho$ is conjugate-dual, that is, the contragredient $\rho^\vee$
is isomorphic to $s \cdot \rho$ for $s \in \mb W_F \setminus \mb W_{\tilde F}$. From
\cite[Theorem 8.1]{GGP} we see that the standard representation of $H^\vee$ is conjugate-orthogonal
or conjugate-symplectic, depending on the size of $H^\vee$. Just as above, $a$ is odd if and only if
$\rho$ and this standard representation have the same type.

The parity of $a_-$ is determined by analogous considerations, starting from a diffe\-rent self-dual 
or conjugate-dual representation $\rho \otimes \chi$ with $\chi$ an unramified character.
More specifically, let $t_\rho$ be the torsion number of $\rho \in \Irr (\mb W_F)$, that is, the 
number of unramified characters $\chi$ such that $\rho \cong \rho \otimes \chi$. Then
there exist precisely $2 t_\rho$ unramified characters $\chi$ of $\mb W_F$ such that $\rho \otimes
\chi$ is self-dual, namely those with $\chi^{2 t_\rho} = 1$. When moreover $\chi^{t_\rho} \neq 1$,
also $\rho \otimes \chi \not\cong \rho$. That provides a unique (up to isomorphism) $\rho_- = \rho
\otimes \chi$ which is self-dual and not isomorphic to $\rho$. The number $a_-$ is computed from 
this $\rho_-$.
The same applies to conjugate-dual representations of $\mb W_{\tilde F}$.

The following result extends \cite[Proposition 1.3]{Hei1}.

\begin{prop}\label{prop:3.3}
\enuma{
\item Let $\rho \in \Irr (\mb W_F)$ be a self-dual and let $\rho_-$ be as above.
\begin{itemize}
\item If $t_\rho$ is odd, then $\rho$ and $\rho_-$ have the same type (orthogonal or 
symplectic).
\item If $t_\rho$ is even, then $\rho$ and $\rho_-$ can have the same or opposite type.
\end{itemize}
\item Let $\tilde \rho \in \Irr (\mb W_{\tilde F})$ be a conjugate-dual and let $\tilde{\rho}_-$
be the unique (up to isomorphism) conjugate-dual twist of $\tilde \rho$ by an unramified character.
\begin{itemize}
\item If $\tilde F / F$ is ramified and $t_{\tilde \rho}$ is odd, then $\tilde \rho$ and 
${\tilde \rho}_-$ have the same type\\ (conjugate-orthogonal or conjugate-symplectic).
\item If $\tilde F / F$ is ramified and $t_{\tilde \rho}$ is even, then $\tilde \rho$ and 
${\tilde \rho}_-$ can have the same or opposite type.
\item If $\tilde F / F$ is unramified, then $\tilde \rho$ and ${\tilde \rho}_-$ have opposite type.
\end{itemize}
}
\end{prop}
\begin{proof}
(a) Since the inertia group $\mb I_{F}$ is normal and $\mb W_{F} / \mb I_{F} \cong \Z$, $\rho$ can
be analysed well by restriction to $\mb I_{F}$. Clifford theory tells us that there exist mutually 
inequivalent irreducible $\mb I_{F}$-representations $\rho_1,\ldots,\rho_t$ such that 
\begin{equation}\label{eq:3.18}
\Res^{\mb W_F}_{\mb I_F} \rho \cong \rho_1 \oplus \cdots \oplus \rho_t
\end{equation}
and a Frobenius element $\Fr$ of $\mb W_{F}$ permutes the $\rho_i$ cyclically. The unramified 
characters $\chi$ that stabilize $\rho$ are precisely those for which $\chi (\Fr^t)$ acts trivially 
on $\rho_1$, so $t$ equals the torsion number $t_\rho$.

If $\rho_1$ is self-dual, then so are all the $\rho_i$, and the $\mb W_{F}$-invariant bilinear form 
on $\rho$ is a direct sum of $\mb I_{F}$-invariant bilinear forms on the $\rho_i$. Then the type of 
$\rho$ is the same as the type of $\rho_1$, which depends only on $\mb I_{F}$ and is not affected by 
twisting with unramified characters. This can happen for even $t$ and for odd $t$.

If $\rho_1$ is not-self dual, then none of the $\rho_i$ is self-dual. In that case $t$ is even and 
the dual of $(\rho_i,V_i)$ is isomorphic to $\rho_{i^\vee}$ for a unique integer $i^\vee$. Further 
the $\mb W_{F}$-invariant bilinear form on $\rho$ restricts on $\rho_i \times \rho_{i^\vee}$ to $z$
times the canonical pairing, for some $z \in \C^\times$. Similarly it restricts on $\rho_{i^\vee} 
\times \rho_i$ to $z^\vee$ times the canonical pairing. It is easy to check that the representation 
$\rho_{i^\vee} \oplus \rho_i$ of $\mb I_{F} \rtimes \langle \Fr^{t/2} \rangle$ is self-dual and 
$z^\vee = \pm z$ where $\pm$ indicates the type of the representation. Then the type of $\rho$ is 
the same as the type of $\rho_{i^\vee} \oplus \rho_i$. 

By self-duality of $\rho$ and $\rho_- = \rho \otimes \chi$ and $\rho \not\cong \rho_-$, we must have
$\chi (\Fr^t) = -1$ and $\chi (\Fr^{t/2}) = \pm i$. In particular the representation $(\rho_{i^\vee} 
\oplus \rho_i) \otimes \chi$ of $\mb I_{F} \rtimes \langle \Fr^{t/2} \rangle$  is not self-dual with 
respect to the same bilinear form as $\rho_{i^\vee} \oplus \rho_i$. To make $(\rho_{i^\vee} \oplus 
\rho_i) \otimes \chi$ self-dual, we can take the bilinear form where in the above description 
$z^\vee$ is replaced by $-z^\vee$. This changes the sign of the 
bilinear form, so $\rho$ and $\rho_-$ have opposite type. \\
(b) When $\tilde F / F$ is ramified, we can pick a representative for $\mb W_F / \mb W_{\tilde F}$
in $\mb I_F$. Then the notions conjugate-dual, conjugate-orthogonal and conjugate-symplectic can be 
defined in the same way for $\mb I_{\tilde F}$-representations. The proof of part (a) applies to
$\tilde \rho \in \Irr (\mb W_{\tilde F})$, when we replace self-dual by conjugate-dual. 
The conclusion is that $\tilde \rho$ and $\tilde \rho_-$ have the same type.

When $\tilde F / F$ is unramified, we pick a representative $s$ for $\mb W_F / \mb W_{\tilde F}$
so that $s^2$ is a Frobenius element of $\mb W_{\tilde F}$. Conjugate-duality is still defined for
$\mb I_{\tilde F}$-representations (because $\mb I_{\tilde F}$ is normal in $\mb W_F$), but the type
of such a representation is not (because $s^2 \notin \mb I_{\tilde F}$). Nevertheless, we can still
decompose $\tilde \rho \in \Irr (\mb W_{\tilde F})$ as $\mb I_{\tilde F}$-representation like in 
\eqref{eq:3.18}. We see that the $\mb W_{\tilde F}$-invariant bilinear pairing between $\tilde \rho$ 
and $s \cdot \tilde \rho$ restricts to a pairing between $({\tilde \rho}_i, V_i)$ and 
$({\tilde \rho}_{i^\vee}, V_{i^\vee})$ for a unique $i^\vee$.
By definition \cite[\S 3]{GGP}, the type of $\tilde \rho$ is given by the sign $\pm$ in
\begin{equation}\label{eq:3.23}
\langle v , v' \rangle = \pm \langle v', \tilde \rho (s^2) v \rangle
\qquad \forall v, v' \in V_{\tilde \rho}.
\end{equation}
We consider $v$ in $V_1$ and $v' \in V_{1^\vee}$ such that the pairing \eqref{eq:3.23} is nonzero.
Then $\tilde \rho (s^2) v$ must also belong to $V_1$, but at the same time $\tilde \rho (s^2)$
permutes the ${\tilde \rho}_i$ cyclically. That renders \eqref{eq:3.23} impossible, unless 
$t_{\tilde \rho} = 1$. But then $\tilde \rho |_{\mb I_{\tilde F}}$ is irreducible and isomorphic to 
$s \cdot {\tilde \rho}^\vee |_{\mb I_{\tilde F}}$. In this situation the
bilinear pairing between $\tilde \rho$ and $s \cdot \tilde \rho$ is already determined by their
structure as $\mb I_{\tilde F}$-representations. 

The same applies to $\tilde \rho_- = \tilde \rho \otimes \tilde \chi$. Then the conjugate-duality of 
$\tilde \rho$ and $\tilde \rho_-$ implies that $\tilde \chi$ is quadratic. It cannot be trivial 
because $\tilde \rho$ and $\tilde \rho_-$ are not isomorphic, so $\tilde \chi$ is the unique 
unramified character of $\mb W_{\tilde F}$ of order two. By \cite[Lemma 3.4]{GGP} $\tilde \chi$ is 
conjugate-symplectic, and by \cite[Lemma 3.5.ii]{GGP} $\tilde \rho$ and $\tilde \rho_-$ have
opposite type.
\end{proof}

Proposition \ref{prop:3.3} is the key to the following result.

\begin{lem}\label{lem:3.10}
We assume the setting of Theorem \ref{thm:3.2}.b.
\enuma{
\item When $\mc G^*$ is a special unitary group which splits over an unramified extension
$\tilde F / F$, $a$ and $a_-$ have different parity. 
\item For all other $\mc G^*$ eligible in Theorem \ref{thm:3.2}: if $t_\tau$ is odd, 
then $a$ and $a_-$ have the same parity.
\item All the labels $\lambda (\alpha), \lambda^* (\alpha)$ in Theorem \ref{thm:3.2} are integers.
}
\end{lem}
\begin{proof}
Let $\rho \in \Irr (\mb W_F)$ be the image of $\tau$ under the LLC for $GL_{d_\tau}(F)$.
The functorial properties of the LLC ensure that $t_\rho = t_\tau$.\\
(b) Assume that $\mc G$ does not have Lie type ${}^2 A_{n-1}$. In the proof of Theorem 
\ref{thm:3.2} we saw how the issue can be reduced from $\mc G$ to $\mc G^*$ or $\mr{GSpin}_n$. 
To $\mc G^*$ and $\mr{GSpin}_n$ we apply Proposition \ref{prop:3.3}.a and the remarks above it.\\
(a) When $\mc G$ does have Lie type ${}^2 A_{n-1}$, Section \ref{sec:red} allows to reduce to 
$\mc G^* = SU_n$, and then to $U_n$. Now we apply Proposition \ref{prop:3.3}.b and the remarks
above it.\\
(c) It is clear that the labels in parts (a) and (c) of Theorem \ref{thm:3.2} are integers.
We recall that the labels in Theorem \ref{thm:3.2}.b are 
\[
\lambda (\alpha) = t_\rho f (a + a_- + 2)/2 \quad \text{and} \quad 
\lambda^* (\alpha) = t_\rho f (a - a_-)/2 .
\]
These are integers, except possibly when $a$ and $a_-$ have different parity. In the cases 
where $\mc G^*$ is an unramified special unitary group, $f = 2$ and again the labels are integers. 
In the other cases with $a$ and $a_-$ of different parity, part (b) of the current lemma tells us 
that $t_\rho$ is even, which makes the labels integral.
\end{proof}

Having checked Conjecture \ref{conj:1}.(i), we turn to Conjecture \ref{conj:1}.(ii).

\begin{lem}\label{lem:3.4}
Consider a root system of type $A_{e-1}, B_e, C_e$ or $D_e$, with label functions 
$\lambda,\lambda^*$ as in Theorem \ref{thm:3.2}. There exist:
\begin{itemize}
\item a simple group $\mc G$ over a nonarchimedean local field $\tilde F$,
\item a Bernstein block $\Rep (\tilde G)^{\mf s}$, which consists of unipotent representations
of $\tilde G = \mc G (\tilde F)$,
\item a $\mf s$-type $(J,\pi)$,
\end{itemize}
such that $\mc H (\tilde G,J,\pi)$ is an affine Hecke algebra with the given root system 
and the given label functions.
\end{lem}
\begin{proof}
When the root system has type $A_{e-1}$ (resp. $D_e$) we take $G = GL_e$ (resp. $SO_{2e}$). 
Choose a non-archimedean local field $\tilde F$ with residue field of order $q_{\tilde F} = q_F^{ft}$,
where $f$ and $t = t_\tau$ are as in Theorem \ref{thm:3.2}.
We take the Iwahori-spherical Bernstein block and let $J$ be an Iwahori subgroup of $\tilde G$.
Then $(J,\mr{triv})$ is a type and $\mc H (\tilde G,J,\mr{triv})$ is an affine Hecke algebra 
with parameters $q_{\tilde F}$. We obtain labels $\lambda (\alpha) = \lambda^* (\alpha) = ft$.

Suppose that the root system is $C_e$ (or $B_e$ with $a_- = -1$, that boils down to the same thing).
We choose the $q$-base $q_F^{ft}$, which can be achieved by considering $\tilde F$-groups.
Thus reduce to the situation where the short root $\alpha$ has label 1 and the long root $\beta$ has 
label $a^+ \in \Z_{>0}$. Now see \cite[7.40--7.42]{Lus-Uni} when $a_+$ is even and \cite[7.56]{Lus-Uni} 
when $a_+$ is odd. 
In each case, a type for the associated Bernstein component is produced in \cite[\S 1]{Lus-Uni}.

Suppose that the root system is $B_e$ and that $a_- \geq 0$. Let $\beta$ be a short root and
$\alpha$ a long root. 
When $a + a_-$ is even we take the $q$-base $q_F^{ft}$ and we reduce to the labels 
\[
\lambda (\alpha) = 1 ,\quad \lambda (\beta) = (a + a_- + 2)/2 ,\quad \lambda^* (\beta) = (a - a_-)/2 .
\] 
Depending on the parities of $\lambda (\beta)$ and $\lambda^* (\beta)$, see
\cite[7.38--7.39]{Lus-Uni} (both even) or \cite[7.48--7.49]{Lus-Uni} (both odd) or
\cite[7.54--7.55]{Lus-Uni} (one even, one odd).

When $a + a_-$ is odd we take the $q$-base $q_F^{ft/2}$ (a power of $q_F$ by Lemma \ref{lem:3.10})  
and we reduce to the labels 
\[
\lambda (\alpha) = 2 ,\quad \lambda (\beta) = a + a_- + 2 ,\quad \lambda^* (\beta) = a - a_- . 
\]
See \cite[11.2--11.3]{Lus-Uni2} for an appropriate Bernstein
component consisting of unipotent representations.
\end{proof}

We covered all simple groups of type $A_n, {}^2 A_n$ or $B_n$, but some simple groups of Lie type 
$C_n, D_n$ or ${}^2 D_n$ remain. With the classification of inner twists via Galois cohomology
and the Kottwitz isomorphism \cite{Kot} we can count them, and realizations of those groups can be 
found in \cite[\S 17.2--17.3]{Spr}:
\begin{itemize}
\item the non-split (non-pure) inner twist of a symplectic group,
\item the two non-pure inner twists of a split even special orthogonal group,
\item the non-pure inner twist of a quasi-split even special orthogonal group,
\item groups isogenous to one of the above.
\end{itemize}
We note that (apart from the last entry) this list consists of classical groups associated to
Hermitian forms on vector spaces over quaternionic division algebras. 
As far as we are aware, much less is known about the representation theory of these groups. They 
are ruled out in \cite{Moe1,Moe2,Moe3}, so it is not clear which Hecke algebra labels can arise. 

For unipotent representations, this is known completely \cite{Lus-Uni,Lus-Uni2,SolLLCunip,SolRamif},
and that indicates that Theorem \ref{thm:3.2} might hold for these groups. The relevant label 
functions $\lambda, \lambda^*$, in the tables \cite[7.44--7.46 and 7.51--7.53]{Lus-Uni}, occur 
also in Theorem \ref{thm:3.2} (with $a - a_-$ odd, like for unitary groups).

\subsection{Groups of Lie type $G_2$} \
\label{par:G2}

Up to isogeny, there are three absolutely simple $F$-groups whose relative root system has
type $G_2$:
\begin{itemize}
\item the split group $G_2$,
\item the quasi-split group ${}^3 D_4$, which splits over a Galois extension $\tilde F / F$
of degree 3 or 6,
\item the non-split inner forms $E_6^{(3)}$, which split over the cubic unramified
extension $F^{(3)} / F$. 
\end{itemize}
Let $G = \mc G (F)$ denote the rational points of one of these groups. Let $M$ be a Levi
subgroup of $G$ and write $\mf s = [M,\sigma]_G, \mc O = X_\nr (M) \sigma$.
When the semisimple rank of $M$ is $\geq 1$, $\Sigma_{\mc O,\mu}$ has rank $\leq 1$. 
For those cases we refer to \eqref{eq:1.8}.

Otherwise $\mc M$ is a minimal $F$-Levi subgroup of $\mc G$.
For $G = G_2 (F)$, $\Rep (G)^{\mf s}$ consists of principal series representations. In Theorem
\ref{thm:3.9} we proved that $q_\alpha = q_F$ and $q_{\alpha*} = 1$ for $\alpha \in \Sigma_{\mc O,\mu}$.
For $G = {}^3 D_4 (F)$ $\Rep (G)^{\mf s}$ also belongs to the principal series. We showed in 
\eqref{eq:3.16} that $q_{\alpha*} = 1, q_\alpha = q_F$ for long roots $\alpha \in \Sigma_{\mc O,\mu}$
and $q_{\beta*} = 1, q_\beta \in \{q_F ,q_F^3\}$ for short roots $\beta \in \Sigma_{\mc O,\mu}$.
Notice that in $\Sigma_{\mc O}^\vee$ the lengths of the roots are reversed.

The group $G = E_6^{(3)}$ involves a central simple $F$-algebra $D$ of dimension $3^2 = 9$.
We assume for the moment that $\mc G$ is simply connected, so that we can apply some reduction steps
from Section \ref{sec:red} more easily.
For a short root $\alpha \in \Sigma_{\mc O,\mu}$, the inclusion $M \to M_\alpha$ is isogenous to 
\[
S(GL_1 (D)^2) \times GL_1 (F) \to S(GL_1 (D)^2) \times SL_2 (F).
\] 
In particular the coroot $\alpha^\vee$ is orthogonal to $M_\der$ and the restriction of $\sigma$
to the image of $\alpha^\vee$ is a direct sum of finitely many characters. Hence the same 
computations as in Paragraph \ref{par:split} apply here, with $M$ instead of $T$. Thus 
$q_\alpha = q_F$ and $q_{\alpha*} = 1$. On the other hand, for a long root 
$\beta \in \Sigma_{\mc O,\mu}$ the inclusion $M \to M_\beta$ is isogenous to 
\[
S(GL_1 (D)^2) \times GL_1 (F) \to SL_2 (D) \times GL_1 (F). 
\]
Again with Section \ref{sec:red} the computation of the parameters can be transferred to 
$GL_1 (D)^2 \to GL_2 (D)$, which is discussed in Paragraph \ref{par:A}. Then Theorem \ref{thm:3.1} 
and \eqref{eq:3.26} show that $q_{\beta*} = 1$ and $q_\beta = q_F^f$ where $f$ divides 3.
All this based on an $X_\alpha$ defined as evaluation at $\alpha^\vee (\varpi_F^{-1})$. 
We still have to take the effect of the isogenies
\begin{equation}\label{eq:3.27}
M_\beta \leftarrow SL_2 (D) \times GL_1 (F) \to GL_2 (D)
\end{equation}
into account. As worked out in Proposition \ref{prop:2.4}, this goes via changing $h_\alpha^\vee$.
Since the derived groups are simply connected, no $\alpha^\vee / 2$ can be involved, and 
this effect comes only from changes in the torsion number $|X_\nr (M,\sigma)|$. That boils 
down to the torsion number of a representation of $GL_1 (D)$, so it can only be 1 or 3. In
terms of cocharacter lattices both maps in \eqref{eq:3.27} are index 2 inclusions, and 2 is
coprime to 3, so actually the torsion numbers do not change along these inclusions.
We conclude that the labels are $\lambda (\alpha) = 1$ and 
$\lambda (\beta) \in \{1,3\}$ (and the same for $\lambda^*$).

When $\mc G$ is not simply connected, we can apply Proposition \ref{prop:2.2} to compare with 
its simply connected cover. If $\Sigma_{\mc O,\mu}$ has rank $>1$, then it is isomorphic to
$A_1 \times A_1$, $A_2$ or $G_2$. In the latter two cases we are not in the instances (ii) or
(iii) of Proposition \ref{prop:2.4}.d, so  Proposition \ref{prop:2.4}.d.(i) tells us that
the parameters do not change when we pass from $\mc G$ to its simply connected cover.
In the first case there could be a change as in Proposition \ref{prop:2.4}.d.(ii) when we
go to a cover of $\mc G$, but that does not bother us because we already understand affine 
Hecke algebras of type $A_1$ completely -- see the discussion before \eqref{eq:1.8}.

\subsection{Groups of Lie type $F_4$} \
\label{par:F4}

Just as for $G_2$ we will analyse all possibilities for the parameters, by reduction to
earlier cases. Up to isogeny there are three absolutely simple $F$-groups with relative root 
system of type $F_4$:
\begin{itemize}
\item the split group $F_4$,
\item the quasi-split group ${}^2 E_6$, split over a separable quadratic extension $F' / F$,
\item the non-split inner form $E_7^{(2)}$, split over the unramified quadratic extension $F^{(2)} / F$. 
\end{itemize}
Supported by Section \ref{sec:red}, we only consider the simply connected version of these groups.
We number the bases of $F_4$ and $E_7$ as follows:

\includegraphics[width=11cm]{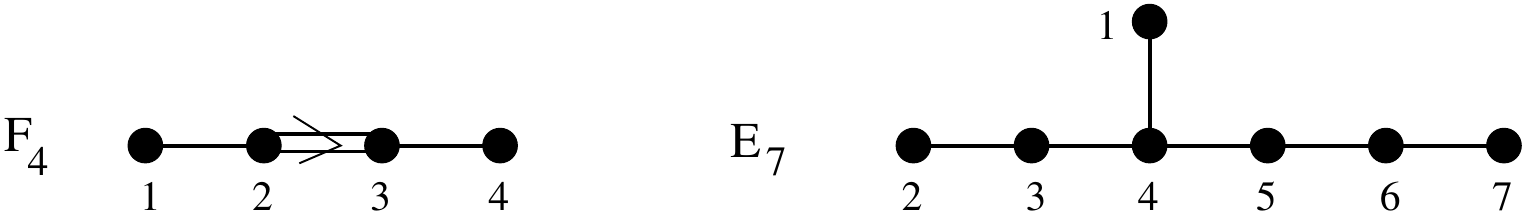}

Let $D$ be a central simple $F$-algebra of dimension $2^2 = 4$. The anisotropic kernel of $E_7^{(2)}(F)$
corresponds to the simple roots labeled $1,5,7$ and is isomorphic to $SL_1 (D)^3$. 

Let $G = \mc G (F)$ denote the rational points of one of the above groups. Fix a maximal $F$-split
torus $S = \mc S (F)$ and let $\Delta$ be a basis of $\Sigma (\mc G,\mc S)$.
Let $M_J = \mc M_J (F)$ be the standard Levi subgroup associated to $J \subset \Delta$. 
Write $\mf s = [M_J ,\sigma ]_G$ and $\mc O = X_\nr (M_J ) \sigma$. We will verify Conjecture 
\ref{conj:1} for $G,M_J$, except that in a few cases for $\mc G = E_7^{(3)}$ we cannot work it out.

Recall from \cite[Proposition 1.3]{Hei2} that $\alpha \in \Sigma_{\mc O,\mu}$ implies 
$s_\alpha \in W(G,M)$. Let $\Sigma_W (A_{M_J})$ be the set of $\alpha \in \Sigma_\red (A_{M_J})$
for which $W(G,M)$ contains $s_\alpha$. Such $s_\alpha$ belong to the normalizer of $W_J$ 
in $W(F_4)$, which links our setup to \cite{How}. It is shown in \cite[Theorem 6]{How}, that
$\Sigma_W (A_{M_J})$ is a root system. The type of $\Sigma_W (A_{M_J})$, as well as a lot 
of other useful data, are collected  in \cite[p. 74]{How}.\\[1mm]

\textbf{$\mathbf{J}$ is empty, $\mathbf{\Sigma_W (A_{M_J}) \cong F_4}$}\\
For $F_4$ and ${}^2 E_6$, $\Rep (G)^{\mf s}$ consists of principal series representations. 
For $G = F_4 (F)$ we proved in Theorem \ref{thm:3.9} that $q_\alpha = q_F$ and $q_{\alpha*} = 1$ for 
all $\alpha \in \Sigma_{\mc O,\mu}$. 

For $G = {}^2 E_6 (F)$, we showed in \eqref{eq:3.16} that $q_{\alpha*} = 1, q_\alpha = q_F$ for 
long roots $\alpha \in \Sigma_{\mc O,\mu}$ and $q_{\beta*} = 1, q_\beta \in \{q_F ,q_F^2\}$ for short 
roots $\beta \in \Sigma_{\mc O,\mu}$. Notice that in $\Sigma_{\mc O}^\vee$ the lengths of the roots 
are reversed.

For $G = E_7^{(2)}(F)$ and $\alpha \in \{\alpha_1, \alpha_2\}$, the inclusion $M_\emptyset \to
M_{\{\alpha\}}$ is isogenous to
\[
GL_1 (F)^2 \times S (GL_1 (D)^3) \longrightarrow GL_2 (F) \times S (GL_1 (D)^3) .
\]
The direct factors $S(GL_1 (D)^3)$ do not influence the rest, so can be ignored for the
computation of the parameters. It follows that $q_\alpha = q_F, q_{\alpha*} = 1$.

For $\alpha \in \{\alpha_3, \alpha_4\}$, we can instead consider the inclusion
\[
GL_1 (F)^2 \times S (GL_1 (D)^3) \longrightarrow GL_1 (F) \times S( GL_2 (D) \times GL_1 (D)) . 
\]
With Section \ref{sec:red} we reduce this to $GL_1 (D)^2 \to GL_2 (D)$, and then Paragraph
\ref{par:A} tells us that $q_{\alpha*} = 1$ and $q_\alpha \in \{ q_F, q_F^2\}$.\\[1mm]

\textbf{$\mathbf{J = \{\alpha_3\}}$ or $\mathbf{J = \{ \alpha_4 \}, 
\Sigma_W (A_{M_J}) \cong B_3}$}\\
These two $J$'s are $W(F_4)$-conjugate, so it suffices to consider $J = \{\alpha_4\}$. 
The parameters for $\alpha_1$ and $\alpha_2$ are the same as when $J$ is empty, so 
$q_{\alpha_1} = q_{\alpha_2} = q_F$ and $q_{\alpha_1 *} = q_{\alpha_2 *} = 1$.

The short simple root of $\Sigma_W (A_{M_J})$ is $\beta = \alpha_2 + 2 \alpha_3 + \alpha_4$,
which is orthogonal to $\alpha_2$ and $\alpha_4$. 
The inclusion $M_J \to M_{J \cup \{\beta\}}$ is isogenous to, depending on the type of $\mc G$:
\[
\begin{array}{cccc}
F_4 & GL_1 (F)^3 \times SL_2 (F) & \to & GL_1 (F)^2 \times SL_2 (F)^2 \\
{}^2 E_6 & GL_1 (F)^2 \times GL_1 (F') \times SL_2 (F') & \to & GL_1 (F)^2 \times SL_2 (F')^2 \\
E_7^{(2)} & GL_1 (F) \times GL_1 (D) \times GL_2 (D) & \to & GL_1 (F) \times SL_2 (D) \times GL_2 (D) 
\end{array}
\]
Again the determination of the $q$-parameters can be simplified with Section \ref{sec:red}. Then 
we see from Paragraph \ref{par:A} that $q_{\beta*} = 1$ and $q_\beta \in \{q_F,q_F^2\}$. \\[1mm]

\textbf{$\mathbf{J = \{\alpha_1\}}$ or $\mathbf{J = \{ \alpha_2 \}, \Sigma_W (A_{M_J}) 
\cong C_3}$}\\
These two $J$'s are $W(F_4)$-conjugate, so it suffices to consider $J = \{\alpha_1\}$.  
The parameters for $\alpha_3$ and $\alpha_4$ are the same as when $J$ is empty.

The long simple root of $\Sigma_W (A_{M_J})$ is $\beta = \alpha_1 + 2 \alpha_2 + 2 \alpha_3$, which is 
orthogonal to $\alpha_1$ and $\alpha_4$. The inclusion $M_J \to M_{J \cup \{\beta\}}$ is isogenous to:
\[
\begin{array}{cccc}
F_4 & SL_2 (F) \times GL_1 (F)^3 & \to & SL_2 (F)^2 \times GL_1 (F)^2 \\
{}^2 E_6 & SL_2 (F) \times GL_1 (F) \times GL_1 (F')^2 & \to & SL_2 (F)^2 \times GL_1 (F')^2 \\
E_7^{(2)} & SL_2 (F) \times GL_1 (F) \times S(GL_1 (D)^3) & \to & SL_2 (F)^2 \times S(GL_1 (D)^3) 
\end{array}
\]
In each of the three cases, this reduces to $GL_1 (F) \to SL_2 (F)$. Hence $q_\beta = q_F$ and 
$q_{\beta *} = 1$.\\[1mm]

\textbf{$\mathbf{J}$ equals $\mathbf{\{\alpha_1, \alpha_4\}}$ or $\mathbf{ \{ \alpha_1,
\alpha_3\}}$ or $\mathbf{\{\alpha_2, \alpha_4\}, \Sigma_W (A_{M_J}) \cong A_1 \times A_1}$}\\
These three subsets of $\Delta$ are associate under the Weyl group $W(F_4)$, so it suffices to 
consider $J = \{\alpha_1, \alpha_3\}$. Up to a sign there are just two possibilities for 
$\alpha \in \Sigma_{\mc O,\mu} \cong A_1 \times A_1$. These can be represented by $\alpha_2$ and 
$\beta = \alpha_2 + 2 \alpha_3 + 2 \alpha_4$. We note that $\beta$ is orthogonal to $\alpha_2$
and $\alpha_3$, but not to $\alpha_1$. The inclusion $M_J \to M_{J \cup \{\beta\}}$
is isogenous to:
\[
\begin{array}{cccc}
F_4 & GL_2 (F) \times GL_2 (F) & \to & SL_3 (F) \times GL_2 (F) \\
{}^2 E_6 & GL_2 (F) \times GL_2 (F') & \to & SL_3 (F) \times GL_2 (F') \\
E_7^{(2)} & GL_2 (F) \times S(GL_2 (D) \times GL_1 (D)) & \to & 
SL_3 (F) \times S(GL_2 (D) \times GL_1 (D)) 
\end{array}
\]
In all three cases this boils down to $GL_2 (F) \to SL_3 (F)$, so $q_\beta = q_F$ and
$q_{\beta *} = 1$.

We also list inclusions isogenous to $M_J \to M_{J \cup \{\alpha_2\}}$:
\[
\begin{array}{cccc}
F_4 & GL_2 (F) \times SO_3 (F) \times GL_1 (F) & \to & SO_7 (F) \times GL_1 (F) \\
{}^2 E_6 & GL_2 (F) \times SO_4^* (F) \times GL_1 (F) & \to & SO_8^* (F) \times GL_1 (F) \\
E_7^{(2)} & GL_2 (F) \times SO'_6 (F) \times GL_1 (D) & \to & SO'_{10}(F) \times GL_1 (D) \\
\end{array}
\]
Here $SO_{2n}^*$ denotes a quasi-split special orthogonal group, while $SO'_{2n}$ stands for a 
non-split inner form of $SO_{2n}$. For the parameter computations, the direct factors
$GL_1 (F)$ and $GL_1 (D)$ can be ignored. In all three cases Theorem \ref{thm:3.2}.b shows that
$q_{\alpha_2} = q_F^{t(a+1)/2}$ and $q_{\alpha_2 *} =  q_F^{t(a_- +1)/2}$, where $t \in \{1,2\}$.
A small correction might still come from the involved isogenies via Proposition \ref{prop:2.4}. 
\\[1mm]

$\mathbf{J = \{ \alpha_1, \alpha_2 \}, \Sigma_W (A_{M_J}) \cong G_2}$\\ 
Now $\alpha_3$ gives rise to a short root of $\Sigma_W (A_{M_J})$, and to a long root of 
$\Sigma_{\mc O}^\vee$. Analysing the inclusion 
$M_J \to M_{J \cup \{\alpha_4\}}$ up to isogeny, we obtain:
\[
\begin{array}{ccccc}
\text{group} & \multicolumn{3}{c}{\text{inclusion}} & q_{\alpha_4} \\
\hline 
F_4 & SL_3 (F) \times GL_1 (F)^2 & \to & SL_3 (F) \times GL_2 (F) & q_F\\
{}^2 E_6 & SL_3 (F) \times GL_1 (F')^2 & \to & SL_3 (F) \times GL_2 (F') & q_{F'} \\
E_7^{(2)} & SL_3 (F) \times SL_1 (D) \times GL_1 (D)^2 & \to & 
SL_3 (F) \times SL_1 (D) \times GL_2 (D) & q_F^f
\end{array}
\]
where $f \in \{1,2\}$. In all three cases $q_{\alpha_4 *} = 1$ by Theorem \ref{thm:3.1}.
Since $\alpha_4^\vee$ is orthogonal to $M_{J,\der}$, the computations behind these $q$-parameters
work equally well in $M_{J \cup \{\alpha_4\}}$, no corrections from isogenies are needed.

For $\alpha_3$ we find
\[
\begin{array}{ccccc}
\text{group} & \multicolumn{3}{c}{\text{inclusion}} & q_{\alpha_3} \\
\hline 
F_4 & GL_3 (F) \times GL_1 (F) & \to & SO_6 (F) \times GL_1 (F) & q_F^{t(a+1)/2} \\
{}^2 E_6 & GL_3 (F) \times SO_2^* (F) \times GL_1 (F') & \to & SO_8^* (F) \times GL_1 (F') & 
 q_F^{t(a+1)/2} \\
E_7^{(2)} & GL_3 (F) \times SO'_4 (F) \times GL_1 (D) & \to & 
SO'_{10} (F) \times GL_1 (D) &  q_F^{t(a+1)/2}
\end{array}
\]
where $t \in \{1,3\}$. The parameter $q_{\alpha_3 *}$ equals $q_F^{t(a_- + 1)/2}$.
Recall the bound on $a$ and $a_-$ from Theorem \ref{thm:3.2}. 

When $\Sigma_{\mc O,\mu}$ has type $G_2$, \cite[Lemma 3.3]{SolEnd} says that $q_{\alpha_3 *} = 1$.
Then $a_- = -1$ and Lemma \ref{lem:3.10} tells us that $a$ is odd. For $F_4$ and $E_6^{(2)}$ that 
means $a = -1$ and $q_{\alpha_3} = 1$, so that actually $\Sigma_{\mc O,\mu}$ does not have type $G_2$.
For $E_7^{(2)}$ it would still be possible that $a=1$, so that $q_{\alpha_3} = q_F^t$. But then
the Langlands parameter of a representation of $SO'_4 (F)$ would be the sum of a three-dimensional
and a one-dimensional representation of $\mb W_F$ which is not compatible with the isogeny to
$SL_1 (D)^2$. Hence this case does not arise, and we conclude that for $J = \{ \alpha_1, \alpha_2 \}$
the root system $\Sigma_{\mc O,\mu}$ has rank $\leq 1$.\\[1mm]

$\mathbf{J = \{ \alpha_3, \alpha_4 \}, \Sigma_W (A_{M_J}) \cong G_2}$\\
Now a long root of $\Sigma_W (A_{M_J})$ comes from $\alpha_1$, and in
$\Sigma_{\mc O}^\vee$ a short root comes from $\alpha_1$. The inclusion $M_J \to M_{J \cup
\{\alpha_1\}}$ is isogenous to:
\[
\begin{array}{cccc}
F_4 & GL_1 (F)^2 \times SL_3 (F) & \to & GL_2(F) \times SL_3 (F) \\
{}^2 E_6 & GL_1 (F)^2 \times SL_3 (F') & \to & GL_2 (F) \times SL_3 (F') \\
E_7^{(2)} & GL_1 (F)^2 \times SL_3 (D) & \to & GL_2 (F) \times SL_3 (D) 
\end{array}
\]
In each case the parameters can be analysed already with $GL_1 (F)^2 \to GL_2 (F)$, and Theorem
\ref{thm:3.1} tells us that $q_{\alpha_1} = q_F, q_{\alpha_1 *} = 1$.

Let us also consider the inclusion $M_J \to M_{J \cup \{\alpha_2\}}$ up to isogenies:
\[
\begin{array}{cccccc}
\text{group} & \multicolumn{3}{c}{\text{inclusion}} & q_{\alpha_2} & q_{\alpha_2 *} \\
\hline
F_4 & GL_1 (F) \times GL_3 (F) & \to & GL_1 (F) \times SO_7 (F) & q_F^{t(a+1)/2} & q_F^{t(a_- + 1)/2}\\
{}^2 E_6 & GL_1 (F) \times GL_3 (F') & \to & GL_1 (F) \times U_6 (F) & q_{F'}^{t(a+1)/2} &
q_F^{t(a_- + 1)/2} \\
E_7^{(2)} & GL_1 (F) \times GL_3 (D) & \to & GL_1 (F) \times SO_6 (D) & ? & ?
\end{array}
\]
Here $t \in \{1,3\}$ and by Theorem \ref{thm:3.2} $0 \geq a \geq a_- \geq -1$. 
When $\Sigma_{\mc O,\mu} \cong G_2$, we know from \cite[Lemma 3.3]{SolEnd}
that $q_{\alpha_2 *} = 1$. With Lemma \ref{lem:3.10} that implies $q_{\alpha_2} = 1$ for $F_4$
and for ${}^2 E_6$ if $F'/F$ is ramified. For ${}^2 E_6$ with $F' / F$ unramified, it is still
possible that $a = 0$, so that $q_{\alpha_2} = q_{F'}^{t/2} = q_F^t$. For the same reasons as
after \eqref{eq:3.27}, no corrections from isogenies are needed.
For $E_7^{(2)}$ the analysis involves quaternionic special orthogonal groups, a case
which remains open.\\[1mm]

$\mathbf{J = \{ \alpha_2, \alpha_3 \}, \Sigma_\red (A_{M_J}) \cong B_2}$\\
Here $\alpha_1$ gives rise to a long root and $\alpha_4$ to a short root of $\Sigma_\red (A_{M_J})$.
We assume that $\alpha_1, \alpha_4 \in \Sigma_{\mc O,\mu}$, otherwise $\Sigma_{\mc O,\mu}$ is 
isomorphic to a root subsystem of $A_1 \times A_1$ and the situation is simpler. 
We would like to say that in $\Sigma_\mc O^\vee$ the relation between the lengths of the roots is 
reversed, but that is not so obvious because $h^\vee_{\alpha_i}$ need not be exactly 
$\alpha_i^\vee (\varpi_F^{-1})$, maybe it has to scaled. 

Up to isogenies, the inclusions $M_J \to M_{J \cup \{\alpha_1\}}$ are:
\begin{equation}\label{eq:3.34}
\begin{array}{cccc}
F_4 & GL_1 (F) \times SO_5 (F) \times GL_1 (F) & \to & SO_7 (F) \times GL_1 (F) \\
{}^2 E_6 & GL_1 (F) \times SO_6^* (F) \times GL_1 (F') & \to & SO_8^* (F) \times GL_1 (F') \\
E_7^{(2)} & GL_1 (F) \times SO'_8 (F) \times GL_1 (D) & \to & SO'_{10}(F) \times GL_1 (D) \\
\end{array}
\end{equation}
Each of the involved isogenies is a twofold cover of the groups listed above, and on the left hand
side that covering does not involve the direct factor $GL_1 (F)$. Hence passing to that cover does
not change $X_\nr (M,\sigma) \cap X_\nr (GL_1 (F))$. For \eqref{eq:3.34} this intersection is trivial,
so also for the analogous setting inside $G$. This shows that 
\[
h_{\alpha_1}^\vee = \alpha_1^\vee (\varpi_F^{-1}) \in M / M^1.
\]
Now the parameters associated with $\alpha_1$ are given by Theorem \ref{thm:3.2}.b, namely
\[
\lambda (\alpha_1) = (a + a_- + 2) / 2 \quad \text{and} \quad
\lambda^* (\alpha_1) = (a - a_-) /2.
\]
Here $a \geq a_- \geq -1$ and $a \leq N^\vee$ with $N^\vee \in \{ 4,6,8\}$ depending on $\mc G$.

Up to isogenies, the inclusion $M_J \to M_{J \cup \{\alpha_4\}}$ is:
\begin{equation}\label{eq:3.35}
\begin{array}{cccc}
F_4 & GL_1 (F) \times Sp_4 (F) \times GL_1 (F) & \to & GL_1 (F) \times Sp_6 (F) \\
{}^2 E_6 & GL_1 (F) \times U_4 (F) \times GL_1 (F') & \to & GL_1 (F) \times U_6 (F) \\
E_7^{(2)} & GL_1 (F) \times SO_4 (D) \times GL_1 (D) & \to & GL_1 (F) \times SO_6 (D)  
\end{array}
\end{equation}
The same argument as for $\alpha_1$ shows that $h^\vee_{\alpha_4} = \alpha^\vee_4 (\varpi_F^{-1})$.
In the root system $\Sigma_{\mc O}^\vee$ we now have the short simple root $h^\vee_{\alpha_1}$ and
the long simple root $h^\vee_{\alpha_4}$. We recall from \cite[Lemma 3.3]{SolEnd} that 
$q_{\alpha_4 *} = 1$ and $\lambda (\alpha_4) = \lambda^* (\alpha_4)$. From \eqref{eq:3.35} and
Theorem \ref{thm:3.2} we deduce that $a_- = -1$ and $q_{\alpha_4} = q_F^a$, at least for $F_4$ and 
${}^2 E_6$. For $E_7^{(2)}$ this involves quaternionic special orthogonal groups, which we could 
not handle in Theorem \ref{thm:3.2}. As explained before Proposition \ref{prop:3.3}, $a$ is an 
odd integer. Moreover, Theorem \ref{thm:3.2} tells us that $(a+1)^2 / 4 \leq N^\vee \in \{5,4\}$. 
It follows that $a \leq 3$, and then
\[
\lambda (\alpha_4) = \lambda^* (\alpha_4) = (a+1)/2 \in \{1,2\} .
\]
We take this opportunity to point out a typo in \cite{Lus-Uni} relevant to us. Namely, when we run
the above arguments with $\sigma$ the unique supercuspidal unipotent representation of $M_J \subset 
G = F_4 (F)$, we obtain the parameters $\lambda (\alpha_1) = 2, \lambda^* (\alpha_1) = 1, 
\lambda (\alpha_4 ) =2$. In \cite[\S 7.31]{Lus-Uni} these are given as $\lambda (\alpha_1) = 3, 
\lambda^* (\alpha_1) = 1, \lambda (\alpha_4 ) = 3$. We already took this into account by not including
labels (3,3,1) for $B_2$ in Table \ref{fig:1}. \\[1mm]
 
$\mathbf{|J| = 3 \text{ or } |J| = 4}$\\
In these cases $\Sigma_{\mc O,\mu}$ has rank $\leq 1$, and we refer to \eqref{eq:1.8}.\\

Summarising: we checked our main conjecture for absolutely simple groups with relative root
system of type $F_4$, except that for the group $E_7^{(2)}$ we are not sure when 
$J = \{\alpha_3,\alpha_4\}$ or $J = \{\alpha_2,\alpha_3\}$. These cases can be settled once we
understand symplectic and special orthogonal groups of quaternionic type better.

\subsection{Groups of Lie type $E_6, E_7, E_8$} \
\label{par:E}

We consider simply connected $F$-split groups of type $E_n$. We number $E_6$ and $E_8$ 
(or rather their bases $\Delta$) as

\includegraphics[width=12cm]{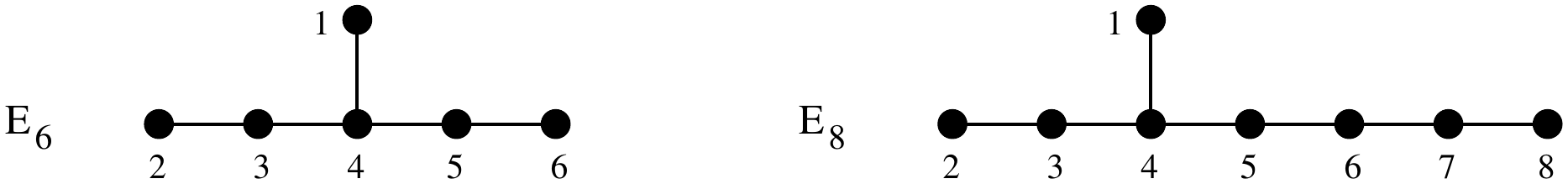}\\
and $E_7$ similarly (as on page \pageref{par:F4}).
The number of inequivalent Levi subgroups is quite large, which renders a case-by-case analysis as 
for $G_2$ and $F_4$ elaborate. An advantage is of course that all these Levi subgroups are simply 
connected and $F$-split, so the analysis of Hecke algebra parameters for $E_n$ consists of the 
principal series (dealt with in Paragraph \ref{par:split}) and contributions from split groups of 
lower rank. For Levi subgroups of semisimple rank $n-1$ the root system $\Sigma_{\mc O,\mu}$ has 
rank $\leq 1$, and before \eqref{eq:1.8} we discussed all such cases. 

For $E_6$ and Levi subgroups of semisimple rank at most 4, the $q$-parameters can be computed
via inclusions $M \to M_\alpha$ where $M_\alpha$ has semisimple rank at most 5. These $\mc M_\alpha$
are not exceptional, so the $q$-parameters can be found from Paragraphs \ref{par:A} and
\ref{par:classical}. For the irreducible components of $\Sigma_{\mc O,\mu}$ of type $A$, 
Conjecture \ref{conj:1} just says that every parameter $q_\alpha$ is a power of $q_F$. That is
readily verified in each case. Therefore we focus on the subsets $J \subset \Delta$ such that 
$\Sigma_W (A_{M_J})$ has a component of type $B_n, C_n, F_4$ or $G_2$.
The possible $J$ can be found by inspecting the tables on \cite[p.75--77]{How}. \\[1mm]

$\mathbf{J = \{ \alpha_1, \alpha_3 \}, \Sigma_W (A_{M_J}) \cong B_3}$\\
The long simple roots $\alpha \in \Sigma_W (A_{M_J})$ correspond to $\alpha_5, \alpha_6 \in 
\Delta$, which are orthogonal to $M_{J,\der}$. Hence the computations reduce to those in
Paragraph \ref{par:split}, and yield $q_{\alpha*} = 1, q_\alpha = q_F$ (or 
$\alpha \notin \Sigma_{\mc O,\mu}$).

The short simple root $\beta$ of $\Sigma_W (A_{M_J})$ comes from $\alpha_4 \in \Delta$.
Here $M_{\beta,\der} \cong SL_4 (F)$ and $M \cap M_{\beta,\der} \cong S(GL_1 (F)^2)$.
If $\beta \in \Sigma_{\mc O,\mu}$, then Corollary \ref{cor:3.11} associates to
$GL_1 (F)^2 \to GL_2 (F)$ the parameters $q_{\beta*} = 1$ and $q_\beta = q_F^f$ with
$f \in \{1,2\}$. Under the isogenies that transfer back to $M \to M_\beta$, $h_\beta^\vee$
remains equal to $\beta^\vee$, so the $q$-parameters do not change.\\[1mm]

$\mathbf{J = \{\alpha_1, \alpha_3, \alpha_4 \}, \Sigma_W (A_{M_J}) \cong B_2 }$\\
The long simple root $\alpha$ of $\Sigma_W (A_{M_J})$ comes from $\alpha_6 \in \Delta$,
which is orthogonal to $M_{J,\der}$. Hence $q_\alpha = q_F$ and $q_{\alpha*} = 1$.

The short simple root $\beta \in \Sigma_W (A_{M_J})$ comes from $\alpha_5 \in \Delta$.
Here $M_{\beta,\der} \cong \text{Spin}_8 (F)$ and $M \cap M_{\beta,\der}$ is a twofold
cover of $SO_6 (F) \times GL_1 (F)$. The $q$-parameters for this setting are known from
Theorem \ref{thm:3.2}:
\[
q_\beta = q_F^{(a+1)/2}, q_{\beta*} = q_F^{(a_- +1)/2} \text{ where } \Big\lfloor \big( 
\frac{a+1}{2} \big)^2 \Big\rfloor + \Big\lfloor \big( \frac{a_- +1}{2} \big)^2 \Big\rfloor \leq 6 ,
\]
so $a \leq 4$. When we apply Proposition \ref{prop:2.4} to $M_{\beta,\der} \to M_\beta$, the
parameters stay the same or (only when $a = a_-$) Proposition \ref{prop:2.4}.d.(iii) applies.\\[1mm]

$\mathbf{J = \{ \alpha_2, \alpha_3, \alpha_5,\alpha_6\}, \Sigma_W (A_{M_J}) \cong G_2}$\\
The long simple root $\alpha \in \Sigma_W (A_{M_J})$ comes from $\alpha_1 \in \Delta$.
That one is orthogonal to 
$M_{J,\der}$, so by Paragraph \ref{par:split} $q_\alpha = q_F$ and $q_{\alpha*} = 1$. 

The short simple root $\beta \in \Sigma_W (A_{M_J})$ comes from $\alpha_4 \in \Delta$.
Now $M_{\beta,\der} \cong SL_6 (F)$ and $M \cap M_{\beta,\der} \cong S(GL_3 (F)^2)$.
The same arguments as above for $J = \{\alpha_1,\alpha_3\}$ shows that here (if $\beta \in 
\Sigma_{\mc O,\mu}$) $q_{\beta*} = 1$ and $q_\alpha = q_F^f$ with $f \in \{1,3\}$.\\

Having checked Conjecture \ref{conj:1} for $E_6$, we turn to the simply connected split 
$F$-groups of type $E_7$ and $E_8$.
For most $J \subset \Delta$, the $q$-parameters of $\mc H (\mc O,G)$ can be analysed as before.
However, some $J$ behave like $\{\alpha_2, \alpha_3\}$ for $F_4$, where we found it hard to
relate the parameters of the two simple roots to each other. For other $J$ (only in $E_8$)
the computation of the $q$-parameters can only be reduced to inclusions of Lie type
$A_2 \times A_1 \times A_2 \to E_6$ or $D_6 \to E_7$ or $E_6 \to E_7$, and we do not know an
effective method in these cases. Therefore we settle for a modest goal:

\begin{lem}\label{lem:3.5}
For groups of Lie type $E_6, E_7$ or $E_8$, Conjecture \ref{conj:1} holds whenever the root 
system $\Sigma_{\mc O,\mu}$ has a component of type $F_4$. 
\end{lem}
\begin{proof}
From \cite[p.75--79]{How} one sees that in only very few cases $\Sigma_{\mc O,\mu}$ has a
component of type $F_4$. For any root $\alpha$ in a type $F_4$ root system, 
\cite[Lemma 3.3]{SolEnd} shows that $q_{\alpha*} = 1$, and then Proposition \ref{prop:2.4} 
entails that no involved isogeny can change the parameters. 

For $G = E_7 (F)$ there is only one $J$ with $\Sigma_W (A_{M_J}) \cong F_4$, namely
$J = \{\alpha_1, \alpha_5, \alpha_7\}$. The $q$-parameters can be obtained in the same way as 
for $E_7^{(2)}(F)$ and $J = \emptyset$, as treated in Paragraph \ref{par:F4}. The only difference
is that an inclusion $S (GL_1 (D)^2) \to SL_2 (D)$ must be replaced by an inclusion
$S (GL_2 (F)^2) \to SL_4 (F)$, but from Paragraph \ref{par:A} we know that exactly the same
$q$-parameters can occur for both these inclusions. Thus $q_\alpha = q_F, q_{\alpha*} = 1$ 
for any long root $\alpha \in \Sigma_{\mc O,\mu} \cong F_4$ and $q_{\beta*} = 1, q_\beta \in 
\{q_F, q_F^2\}$ for any short root $\beta \in \Sigma_{\mc O,\mu}$.

For $E_8 (F)$ and $J = \{ \alpha_1, \alpha_5, \alpha_7\}$ we also have $\Sigma_W (A_{M_J})
\cong F_4$. This case can be handled just as for $E_7$, and leads to the same $q$-parameters.

For $E_8 (F)$ and $J = \{ \alpha_1, \alpha_3, \alpha_5\}$ we have $\Sigma_W (A_{M_J}) \cong
F_4 \times A_1$. The long simple roots of $F_4$ come from $\alpha_7, \alpha_8 \in \Delta$.
These are orthogonal to $M_\der$, so $q_\alpha = q_F$ and $q_{\alpha*} = 1$. According to
\cite[p. 75]{How} the short simple roots $\beta$ of $F_4$ are associated to an inclusion
$S(GL_2 (F)^2) \to SL_4 (F)$. We can use the same $X_\beta$ as for $GL_2 (F)^2 \to GL_4 (F)$,
for which Corollary \ref{cor:3.11} shows that $q_{\beta*} = 1$ and $q_\beta \in \{ q_F, q_F^2 \}$.

The only remaining case with $\Sigma_W (A_{M_J}) \cong F$ is $J = \{ \alpha_1, \alpha_3,
\alpha_4, \alpha_5\}$. Like in the previous case $q_\alpha = q_F, q_{\alpha*} = 1$ for any
long simple root $\alpha \in \Sigma_{\mc O,\mu}$. Both short simple roots $\beta$ of $F_4$ come
from a non-simple root in $E_8$, for which $M \cap M_{\beta,\der} \to M_{\beta,\der}$ is 
isomorphic to the inclusion of a double cover of $SO_8 (F) \times GL_1 (F)$ in $\mr{Spin}_{10}(F)$.
According to Theorem \ref{thm:3.2} the resulting $q$-parameters are 
\[
q_\beta = q_F^{(a+1)/2} \text{ and } q_{\beta*} = q_F^{(a_- + 1)/2}, \text{ where }
\Big\lfloor \big(\frac{a+1}{2} \big)^2 \Big\rfloor + \Big\lfloor \big( \frac{a+1}{2} \big)^2 
\Big\rfloor \leq 8 .
\]
Since $\Sigma_{\mc O,\mu}$ has type $F_4$, $q_{\beta*} = 1$ and $a_- = -1$.
From Lemma \ref{lem:3.10} we know that $a$ and $a_-$ have the same parity, so $a$ is odd.
The estimate shows that $a < 5$, so $a \in \{1,3\}$ and $q_\beta \in \{q_F,q_F^2\}$ as desired.
\end{proof}


\begin{thebibliography}{99}

\bibitem[ABPS]{ABPS2} A.-M. Aubert, P.F. Baum, R.J. Plymen, M. Solleveld,
``Hecke algebras for inner forms of $p$-adic special linear groups",
J. Inst. Math. Jussieu {\bf 16.2} (2017), 351--419

\bibitem[AMS1]{AMS1} A.-M. Aubert, A. Moussaoui, M. Solleveld,
``Generalizations of the Springer correspondence and cuspidal Langlands parameters",
Manus. Math. {\bf 157} (2018), 121--192

\bibitem[AMS2]{AMS2} A.-M. Aubert, A. Moussaoui, M. Solleveld,
``Graded Hecke algebras for disconnected reductive groups",
in: \emph{Geometric Aspects of the Trace Formula
(editors W. M\"uller, S.W. Shin, N. Templier)}, Simons Symposia, Springer, 2018

\bibitem[AMS3]{AMS3} A.-M. Aubert, A. Moussaoui, M. Solleveld,
``Affine Hecke algebras for Langlands parameters",
arXiv:1701.03593v3, 2019

\bibitem[AMS4]{AMS4} A.-M. Aubert, A. Moussaoui, M. Solleveld,
``Affine Hecke algebras for classical $p$-adic groups",
arXiv:2211.08196, 2022

\bibitem[AuXu]{AuXu} A.-M. Aubert, Y. Xu,
``Hecke algebras for $p$-adic reductive groups and local Langlands
correspondence for Bernstein blocks", Adv. Math. {\bf 436} (2024)

\bibitem[Bade]{Bad} M.P. Badea,
``Hecke algebras for covers of principal series Bernstein
components in quasisplit unitary groups over local Fields",
PhD. thesis, Radboud Universiteit Nijmegen, 2020

\bibitem[Badu]{Badu} A.I. Badulescu,	
``Correspondance de Jacquet--Langlands pour les corps locaux de caract\'eristique non nulle",
Ann. Scient. \'Ec. Norm. Sup. 4e s\'erie {\bf 35} (2002), 695--747

\bibitem[BaSa]{BaSa} P. Bak\'ic, G. Savin,
``The Gelfand--Graev representation of SO(2n+1) in terms of Hecke algebras",
Canad. J. Math. {\bf 75.4} (2023), 1343--1368 

\bibitem[BeDe]{BeDe} J. Bernstein, P. Deligne,
``Le "centre" de Bernstein",
pp. 1--32 in: \emph{Repr\'esentations des groupes r\'eductifs sur un corps local}, 
Travaux en cours, Hermann, Paris, 1984

\bibitem[BeRu]{BeRu} J. Bernstein, K.E. Rumelhart,
``Representations of $p$-adic groups", draft, 1993

\bibitem[BeZe]{BeZe} J. Bernstein, A. Zelevinsky,
``Representations of the group $GL(n,F)$ where $F$ is a local non-archimedean field",
Usp. Mat. Nauk {\bf 31.3} (1976), 5--70

\bibitem[Bor]{Bor} A. Borel,
``Admissible representations of a semi-simple group over a local field
with vectors fixed under an Iwahori subgroup'',
Inv. Math. {\bf 35} (1976), 233--259

\bibitem[BrTi]{BrTi2} F. Bruhat, J. Tits, ``Groupes r\'eductifs sur un corps local: 
II. Sch\'emas en groupes. Existence d'une donn\'ee radicielle valu\'ee",
Publ. Math. Inst. Hautes \'Etudes Sci. {\bf 60} (1984), 5--184

\bibitem[Bus]{Bus} C.J. Bushnell,
``Representations of reductive $p$-adic groups: localization of Hecke algebras and applications",
J. London Math. Soc. (2) {\bf 63} (2001), 364--386

\bibitem[BuKu1]{BuKu1} C.J. Bushnell, P.C. Kutzko,
\emph{The admissible dual of GL(N) via compact open subgroups},
Annals of Mathematics Studies {\bf 129}, Princeton University Press, 1993

\bibitem[BuKu2]{BuKu2} C.J. Bushnell, P.C. Kutzko,
``Smooth representations of reductive $p$-adic groups: structure theory via types",
Proc. London Math. Soc. {\bf 77.3} (1998), 582--634

\bibitem[Cas]{Cas} W. Casselman,	
``The unramified principal series representations of $p$-adic groups I. The spherical function",
Compos. Math {\bf 40.3} (1980), 387--406

\bibitem[Del]{Del} P. Deligne,
``Les corps locaux de caract\'eristique $p$, limites de corps locaux de caract\'eristique 0",
pp. 119--157 in: \emph{Repr\'esentations des groupes r\'eductifs sur un corps local},
Travaux en cours, Hermann, 1984

\bibitem[DKV]{DKV} P. Deligne, D. Kazhdan, M.-F. Vign\'eras,
``Repr\'esentations des alg\`ebres centrales simples $p$-adiques",
pp. 33--117 in: \emph{Repr\'esentations des groupes r\'eductifs sur un corps local},
Travaux en cours, Hermann, 1984

\bibitem[GGP]{GGP} W.-T. Gan, B. Gross, D. Prasad, 
``Symplectic local root numbers, central critical L-values and restriction problems in the 
representation theory of classical groups", Ast\'erisque {\bf 346} (2012), 1--109

\bibitem[Gan1]{Gan1} R. Ganapathy,	
``Congruences of parahoric group schemes",
Algebra Number Th. {\bf 13.6} (2019), 1475--1499

\bibitem[Gan2]{Gan2} R. Ganapathy,			
``A Hecke algebra isomorphism over close local fields",
Pacific J. Math. {\bf 319.2} (2022), 307--332

\bibitem[Hei1]{Hei1} V. Heiermann,
``Param\`etres de Langlands et alg\`ebres d'entrelacement",
Int. Math. Res. Not. {\bf 2010}, 1607--1623

\bibitem[Hei2]{Hei2} V. Heiermann,
``Op\'erateurs d'entrelacement et alg\`ebres de Hecke avec param\`etres d'un 
groupe r\'eductif $p$-adique - le cas des groupes classiques",
Selecta Math. {\bf 17.3} (2011), 713--756

\bibitem[Hei3]{Hei3} V. Heiermann,
``Local Langlands correspondence for classical groups and affine Hecke algebras",
Math. Zeitschrift {\bf 287} (2017), 1029--1052

\bibitem[How]{How} R.B. Howlett,
``Normalizers of parabolic subgroups of reflection groups",
J. London Math. Soc. (2) {\bf 21} (1980), 62--80

\bibitem[Hum]{Hum} J.E. Humphreys,
\emph{Reflection groups and Coxeter groups},
Cambridge Studies in Advanced Mathematics {\bf 29}
Cambridge University Press, Cambridge, 1990

\bibitem[IwMa]{IwMa} N. Iwahori, H. Matsumoto,
``On some Bruhat decomposition and the structure
of the Hecke rings of the $p$-adic Chevalley groups'',
Inst. Hautes \'Etudes Sci. Publ. Math {\bf 25} (1965), 5--48

\bibitem[Kaz]{Kaz} D. Kazhdan,
``Representations of groups over close local fields",
J. Analyse Math. {\bf 47} (1986), 175–-179

\bibitem[KaLu]{KaLu} D. Kazhdan, G. Lusztig,
``Proof of the Deligne--Langlands conjecture for Hecke algebras",
Invent. Math. {\bf 87} (1987), 153--215

\bibitem[Kot]{Kot} R.E. Kottwitz,
``Stable trace formula: cuspidal tempered terms",
Duke Math. J. {\bf 51.3} (1984), 611--650

\bibitem[Lus1]{Lus-Cusp1} G. Lusztig,
``Cuspidal local systems and graded Hecke algebras",
Publ. Math. Inst. Hautes \'Etudes Sci. {\bf 67} (1988), 145--202

\bibitem[Lus2]{Lus-Gr} G. Lusztig,
``Affine Hecke algebras and their graded version",
J. Amer. Math. Soc {\bf 2.3} (1989), 599--635

\bibitem[Lus3]{Lus-Uni} G. Lusztig,
``Classification of unipotent representations of simple $p$-adic groups",
Int. Math. Res. Notices {\bf 11} (1995), 517-589

\bibitem[Lus4]{Lus-Uni2} G. Lusztig,
``Classification of unipotent representations of simple $p$-adic groups II",
Represent. Theory {\bf 6} (2002), 243--289

\bibitem[Lus5]{Lus-open} G. Lusztig,
``Open problems on Iwahori--Hecke algebras",
arXiv:2006.08535, EMS Newsletter (2020) 

\bibitem[MiSt]{MiSt} M. Miyauchi, S. Stevens,
``Semisimple types for $p$-adic classical groups",
Math. Ann. {\bf 358} (2014), 257--288

\bibitem[M\oe 1]{Moe1} C. M\oe glin,
``Classification et changement de base pour les s\`eries discr\`etes des groupes unitaires $p$-adiques", 
Pacific J. Math. {\bf 233} (2007), 159–-204		 

\bibitem[M\oe 2]{Moe2} C. M\oe glin, 
``Multiplicit\'e 1 dans les paquets d'Arthur aux places $p$-adiques",
pp. 333--374 in: \emph{On certain L-functions}, 
Clay Math. Proc. {\bf 13}, American Mathematical Society, 2011

\bibitem[M\oe 3]{Moe3} C. M\oe glin,		 
``Paquets stables des s\'eries discr\`etes accessibles par endoscopie tordue: 
leur param\`etre de Langlands", Contemp. Math. {\bf 614} (2014), 295-–336

\bibitem[MoRe]{MoRe} C. M\oe glin, D. Renard,
``Sur les paquets d'Arthur des groupes classiques et unitaires non quasi-d\'eploy\'es",  
pp. 341--361 in: \emph{Relative Aspects in Representation Theory, Langlands Functoriality 
and Automorphic Forms}, Lecture Notes in Mathematics \textbf{2221}, 2018

\bibitem[Mor]{Mor1} L. Morris,
``Tamely ramified intertwining algebras",
Invent. Math. {\bf 114.1} (1993), 1--54
           
\bibitem[MoPr]{MoPr} A. Moy, G. Prasad,
``Jacquet functors and unrefined minimal K-types",
Comment. Math. Helvetici {\bf 71} (1996), 98--121

\bibitem[Ren]{Ren} D. Renard,
\emph{Repr\'esentations des groupes r\'eductifs $p$-adiques},
Cours sp\'ecialis\'es {\bf 17}, Soci\'et\'e Math\'ematique de France, 2010

\bibitem[Roc1]{Roc1} A. Roche,
``Types and Hecke algebras for principal series representations of split reductive $p$-adic groups",
Ann. Sci. \'Ecole Norm. Sup. {\bf 31.3} (1998), 361--413

\bibitem[Roc2]{Roc2} A. Roche,
``The Bernstein decomposition and the Bernstein centre",
pp. 3--52 in: \emph{Ottawa lectures on admissible representations of reductive $p$-adic groups},
Fields Inst. Monogr. {\bf 26}, Amer. Math. Soc., Providence RI, 2009

\bibitem[S\'ec]{Sec} V. S\'echerre, 
``Repr\'esentations lisses de $GL_m (D)$ III: types simples",
Ann. Scient. \'Ec. Norm. Sup. {\bf 38} (2005), 951--977.

\bibitem[S\'eSt1]{SeSt4} V. S\'echerre, S. Stevens, 
``Repr\'esentations lisses de $GL_m (D)$ IV: repr\'esentations supercuspidales", 
J. Inst. Math. Jussieu {\bf 7.3} (2008), 527--574

\bibitem[S\'eSt2]{SeSt6} V. S\'echerre, S. Stevens,
``Smooth representations of $GL(m,D)$ VI: semisimple types",
Int. Math. Res. Notices (2011)

\bibitem[Ser]{Ser} J.-P. Serre,
\emph{Local fields},
Springer Verlag, 1979

\bibitem[Sha]{Sha} F. Shahidi,	
``A proof of Langlands' conjecture on Plancherel measures; complementary series of $p$-adic groups",
Ann. Math. {\bf 132.2} (1990), 273--330

\bibitem[Sil1]{Sil1} A.J. Silberger,
``Isogeny restrictions of irreducible admissible representations are
finite direct sums of irreducible admissible representations",
Proc. Amer. Math. Soc. {\bf 73.2} (1979), 263--264

\bibitem[Sil2]{Sil2} A.J. Silberger,
\emph{Introduction to harmonic analysis on reductive p-adic groups},
Mathematical Notes {\bf 23}, Princeton University Press, 1979

\bibitem[Sil3]{Sil3} A.J. Silberger,
``Special representations of reductive $p$-adic groups are not integrable",
Ann. Math {\bf 111} (1980), 571--587

\bibitem[Sol1]{SolAHA} M. Solleveld,
``On the classification of irreducible representations of affine Hecke algebras 
with unequal parameters", Representation Theory {\bf 16} (2012), 1--87

\bibitem[Sol2]{SolFunct} M. Solleveld,
``Langlands parameters, functoriality and Hecke algebras",
Pacific J. Math. {\bf 304.1} (2020), 209--302

\bibitem[Sol3]{SolHecke} M. Solleveld,
``Affine Hecke algebras and their representations",
Indagationes Mathematica {\bf 32.5} (2021), 1005--1082

\bibitem[Sol4]{SolEnd} M. Solleveld,
``Endomorphism algebras and Hecke algebras for reductive $p$-adic groups",
J. Algebra {\bf 606} (2022), 371--470, and correction in arXiv:2005.07899v3 (2023)

\bibitem[Sol5]{SolLLCunip} M. Solleveld,
``A local Langlands correspondence for unipotent representations",
Amer. J. Math. {\bf 145.3} (2023), 673--719

\bibitem[Sol6]{SolRamif} M. Solleveld,
``On unipotent representations of ramified $p$-adic groups",
Representation Theory {\bf 27} (2023), 669--716

\bibitem[Sol7]{SolLLCQ} M. Solleveld,
``On principal series representations of quasi-split reductive $p$-adic groups", 
arXiv:2304.06418, 2023

\bibitem[Spr]{Spr} T.A. Springer,
\emph{Linear algebraic groups 2nd ed.},
Progress in Mathematics {\bf 9}, Birkh\"auser, Boston MA, 1998

\bibitem[Tad]{Tad} M. Tadi\'c,
``Notes on representations of non-archimedean $SL (n)$",
Pacific J. Math. {\bf 152.2} (1992), 375--396

\bibitem[Wal]{Wal} J.-L. Waldspurger,
``La formule de Plancherel pour les groupes $p$-adiques (d'apr\`es Harish-Chandra)",
J. Inst. Math. Jussieu {\bf 2.2} (2003), 235--333

\end{thebibliography}
\end{document}